\documentclass[11pt]{amsart}
\usepackage{amscd}     
\usepackage{amssymb}
\usepackage{amsmath, amsthm, graphics, dsfont}
\usepackage[all]{xy} 

\usepackage[margin=0.9in]{geometry}

\usepackage{graphicx} 

\usepackage[colorlinks=true, linkcolor=blue, citecolor=blue, pagebackref=true]{hyperref}

\usepackage{extarrows}
\usepackage{latexsym}
\usepackage{times}
\usepackage{tikz}
\usepackage{tikz-cd}
\usepackage{comment}
\usepackage{url}
\usepackage{mathrsfs}
\usepackage{amsmath,amstext,amsthm,amssymb,amsfonts,amscd}
\usepackage{leftindex}
\usepackage{mathrsfs}
\usepackage[colorlinks=true]{hyperref}
\usepackage[T1]{fontenc}
\usepackage[latin1]{inputenc}
\usepackage{color,tocvsec2}
\usepackage{charter}
\usepackage{enumerate}
\usepackage{lscape}
\usepackage[all]{xy}
\usepackage[normalem]{ulem}

\RequirePackage[T1]{fontenc}

\newcommand{\End}{\text{End}}
\newcommand{\HBC}{H_\text{BC}}
\newcommand{\chBC}{\text{ch}_\text{BC}}
\newcommand{\TdBC}{\text{Td}_\text{BC}}
\newcommand{\Td}{\text{Td}}
\newcommand{\Hom}{\text{Hom}}

\newcommand{\id}{\text{id}}
\newcommand{\str}{\text{Tr}_{\text{s}}}
\newcommand{\pr}{\text{pr}}

\newcommand{\Vect}{\text{Vect}}
\newcommand{\Gr}{\text{Gr}}

\newcommand{\ch}{\text{ch}}

\newcommand{\diff}{\text{d}}
\newcommand{\Diff}{\text{Diff}}
\newcommand{\dpar}{\partial}
\newcommand{\dbar}{\overline{\partial}}

\newcommand{\coh}{\text{coh}}

\newtheorem{thm}{Theorem}[section]
\newtheorem{lemma}[thm]{Lemma}
\newtheorem{prop}[thm]{Proposition}
\newtheorem{coro}[thm]{Corollary}
\newtheorem{defi}[thm]{Definition}
\newtheorem{eg}[thm]{Example}
\newtheorem{rmk}[thm]{Remark}

\newtheorem*{thm*}{Theorem}

\numberwithin{equation}{section}

\title[Superconnections, Bott-Chern, and orbifold RRG, I]{Superconnections and Bott-Chern valued Riemann-Roch-Grothendieck for complex orbifolds, I}

\author[Ma]{Qiaochu Ma}
\address{Department of Mathematics\\ Texas A\&M University\\ College Station, TX 77840\\ USA}
\email{qiaochu@tamu.edu}

\author[Tang]{Xiang Tang}
\address{Department of Mathematics\\ Washington University in St. Louis\\ St. Louis, MO, 63130\\ USA}
\email{xtang@math.wustl.edu}

\author[Tseng]{Hsian-Hua Tseng}
\address{Department of Mathematics\\ Ohio State University\\ Columbus,  OH 43210\\ USA}
\email{hhtseng@math.ohio-state.edu}

\author[Wei]{Zhaoting Wei}
\address{Department of Mathematics\\ East Texas A\&M University\\ Commerce, TX 75428\\ USA}
\email{Zhaoting.Wei@etamu.edu}

\begin{document}

\date{\today}

\begin{abstract}
We use flat antiholomorphic superconnections to study orbifold Chern character following the method introduced by Bismut, Shen, and Wei on general complex orbifolds. We show the uniqueness of orbifold Chern character by proving a Riemann-Roch-Grothendieck theorem for orbifold embeddings.
\end{abstract}

\maketitle

\setcounter{tocdepth}{1}

\tableofcontents

\section{Introduction}

\subsection{Riemann-Roch, Index Theory, and the Orbifold Setting}

The Riemann-Roch theorem originated as a formula relating meromorphic functions and divisors on a compact Riemann surface. Hirzebruch extended it to holomorphic vector bundles on smooth projective varieties. In its analytic form, for a holomorphic vector bundle $E$ on a compact complex manifold $X$, the Hirzebruch-Riemann-Roch formula is
\begin{equation}\label{eq:intro-HRR}
  \chi(X,E)=\int_X \operatorname{ch}(E)\operatorname{Td}(TX),
\end{equation}
where $\chi(X,E)$ is the holomorphic Euler characteristic
\cite{Hir66}.

Grothendieck's reformulation shifted the emphasis from a numerical identity on a single space to functoriality under proper morphisms. If $f\colon X\to Y$ is proper and $\mathcal F$ is a coherent sheaf on $X$, then
\begin{equation}\label{eq:intro-GRR}
  \operatorname{ch}(f_!\mathcal F)\operatorname{Td}(TY)
  =
  f_*\bigl(\operatorname{ch}(\mathcal F)\operatorname{Td}(TX)\bigr).
\end{equation}
Thus, after the correction given by the Todd class, the Chern character is compatible with direct image in $K$-theory and  pushforward in cohomology. Another valuable point is that
\eqref{eq:intro-GRR} is an identity of characteristic classes on $Y$, not just an equality of numbers. Therefore it is compatible with composition and may be applied before a further pushforward to a point \cite{berthelot1966seminaire}.

For a closed embedding $i\colon X\hookrightarrow Y$, \eqref{eq:intro-GRR} can be rewritten as 
\begin{equation}\label{eq:intro-GRR-embedding}
\operatorname{ch}(i_*\mathcal F)
=
i_*\left(\operatorname{ch}(\mathcal F)\operatorname{Td}(N_{X/Y})^{-1}\right),
\end{equation}
where $N_{X/Y}=i^*TY/TX$ is the normal bundle. This is the form most relevant to the present paper: it recovers the characteristic class of a sheaf on $Y$ supported on $X$ from data intrinsic to $X$ and its embedding into $Y$. The goal of this paper is to establish an analytic orbifold analogue of this class-valued embedding formula, rather than only a formula obtained after integration over the target.

Index theory supplies the analytic counterpart of the numerical Riemann-Roch theorem. For a Hermitian holomorphic vector bundle $E$ on a compact complex manifold, its Dolbeault complex is elliptic and
\[
  \operatorname{ind}(\bar\partial_E)=\chi(X,E).
\]
The Atiyah-Singer index theorem generalized the Hirzebruch-Riemann-Roch theorem \eqref{eq:intro-HRR} to general elliptic operators on closed manifolds \cite{AS68}. The heat-equation method of Atiyah-Bott-Patodi gives a local analytic derivation of the index density \cite{ABP73}, while
Atiyah-Patodi-Singer theory extends the theory to manifolds with boundary through the eta invariant and its boundary correction \cite{APS75}. These results identify analytic and topological indices, but the final invariant is still a number after pushforward to a point. Grothendieck-Riemann-Roch, by contrast, governs characteristic classes under direct images between spaces. The theorem proved here is of this class-valued kind, although its characteristic classes are constructed by
analytic and differential-geometric methods.

For coherent analytic sheaves, constructing such classes is more subtle than for vector bundles. A coherent sheaf on a compact complex manifold may not have a finite global resolution by holomorphic vector bundles, and its direct image under an embedding may fail to admit such a resolution even when the original sheaf does. O'Brian-Toledo-Tong developed Chern characters and Grothendieck-Riemann-Roch formulas for coherent analytic sheaves in Hodge cohomology \cite{OTT81,OTT85}, and Grivaux constructed characteristic classes in rational Deligne cohomology \cite{Grivaux}.
Bismut-Shen-Wei subsequently used flat antiholomorphic superconnections to construct Chern characters of coherent analytic sheaves taking values in Bott-Chern cohomology and to prove Riemann-Roch-Grothendieck for proper holomorphic maps between compact complex manifolds \cite{bismut2023coherent}.

Bott-Chern cohomology is the natural target for the present construction because the problem is complex analytic and the spaces under consideration are not assumed to be K\"ahler. Hermitian characteristic forms are $d$-closed forms of pure type, and their variation under a change of metric is $\partial\bar\partial$-exact. They therefore define intrinsic Bott-Chern classes. This target preserves the type decomposition, removes dependence on auxiliary metrics, and admits a current realization compatible with pushforward under closed embeddings. When the $\partial\bar\partial$-lemma holds, these classes recover the corresponding Dolbeault and de Rham information; in the potentially non-K\"ahler setting, they may retain strictly finer complex-geometric
information \cite{BC65,bismut1990complex,bismut2011hypoelliptic,bismut2023coherent}.

We now turn from manifolds to orbifolds. Equivariant index theory provides the natural bridge. If a compact group $G$ acts on a compact manifold $M$ and $D$ is a $G$-equivariant elliptic operator, its equivariant index is the virtual representation
\[
  \operatorname{Ind}_G(D)=[\ker D]-[\operatorname{coker}D]\in R(G).
\]
Equivalently, it determines the virtual character
\[
  \operatorname{Ind}_g(D)  
=  \operatorname{Tr}\bigl(g|_{\ker D}\bigr)
  -
  \operatorname{Tr}\bigl(g|_{\operatorname{coker}D}\bigr),
  \qquad g\in G.
\]
Segal's equivariant $K$-theory provides the natural home for the equivariant symbol and index, with $K_G(\mathrm{pt})=R(G)$ \cite{segal1968equivariant}. The fixed-point formulas of Atiyah-Bott and Atiyah-Segal express $\operatorname{Ind}_g(D)$ in terms of characteristic classes on the $g$-fixed point submanifold $M^g$ and the action of $g$ on its normal bundle \cite{AB67,ASeg68}. The Atiyah-Segal completion theorem gives a complementary relation between equivariant $K$-theory and the $K$-theory of the Borel construction \cite{AS69}.

The equivariant Chern character carries the same refinement at the level of characteristic classes. On the $g$-fixed locus, it records both the curvature of the bundle and the action of $g$ on its fibers. Evaluation at different group elements provides representation-theoretic information invisible to the ordinary Chern character. Fixed loci, stabilizer-weighted traces, and eigenvalue-dependent normal factors are therefore intrinsic features of equivariant index theory.

Orbifolds, introduced by Satake as $V$-manifolds, provide the intrinsic setting in which local finite-group actions and stabilizer representations vary from chart to chart \cite{Satake}. They are locally modeled on finite quotients but need not admit a presentation as a single global quotient. For $X=[M/G]$ with $G$ finite, an elliptic operator on $X$ may be studied as a $G$-equivariant operator on $M$, and its orbifold
index is 
\[
  \operatorname{Ind}_{\mathrm{orb}}(D)  =  \frac{1}{|G|}\sum_{g\in G}\operatorname{Ind}_g(D).
\]
Kawasaki extended this construction to general compact orbifolds and proved orbifold Riemann-Roch and index theorems \cite{Kawasaki,Kawasaki81}. His formula expresses the holomorphic Euler characteristic as a sum of characteristic-class integrals over the fixed sectors, now understood as the components of the inertia orbifold.

Kawasaki's theorems explain why orbifold characteristic classes must retain stabilizer traces and why moving normal directions contribute sectorwise correction factors. They also motivate the passage from $X$ to its inertia orbifold $IX$. This passage records the representation theory of orbifold stabilizers and should be distinguished from the use of Bott-Chern cohomology, which is dictated by the complex-analytic and potentially non-K\"ahler nature of the problem.

Kawasaki's conclusion is numerical after integration over the inertia orbifold. The present paper instead constructs a class 
  $\chBC(\mathcal F) \in H_{\mathrm{BC}}^{(=)}(IX,\mathbb{C})$, 
the orbifold Bott-Chern cohomology, before integration and proves its functoriality under direct image. The two features of the target have distinct origins: the inertia orbifold records stabilizer characters, while Bott-Chern cohomology retains the $\partial\bar\partial$-sensitive information carried by characteristic forms of pure type.

Class-valued orbifold Riemann-Roch theorems have also been developed in algebraic and analytic geometry. To\"en proved an algebraic Riemann-Roch theorem for Deligne-Mumford stacks using the inertia stack; see also Edidin's account \cite{Toen,Edi13}. In the analytic setting, Ma defined a Bott-Chern-valued orbifold Chern character for holomorphic orbifold vector bundles and proved an embedding formula for stabilizer-preserving embeddings under the assumption that the direct image admits a finite locally free resolution \cite{ma2005orbifolds}.

\subsection{Context and purpose}

The purpose of this paper is to construct a Morita-invariant orbifold Chern character for coherent analytic complexes on a compact complex orbifold and to prove its Riemann-Roch-Grothendieck formula for orbifold embeddings. The construction converts coherent analytic information into differential forms that can be pulled back, multiplied, and pushed forward. It therefore determines the characteristic class of a direct image before final integration and retains information that a numerical index forgets.

Three distinct difficulties determine the form of the construction. First, the absence of finite global locally free resolutions is addressed by using flat antiholomorphic superconnections. These provide finite-rank smooth representatives for coherent complexes and encode, through their higher antiholomorphic-form components, the homotopy-coherent data that a global holomorphic resolution may fail to capture \cite{Quillen1985,block2010duality,bismut2023coherent}.

Second, the possible failure of the $\partial\bar\partial$-lemma requires a cohomology theory that preserves pure type and metric transgression. At the level of forms,
\[
H_{\mathrm{BC}}^{p,q}(X)=
\frac{\{\alpha\in\Omega^{p,q}(X):d\alpha=0\}}{\partial\bar\partial\Omega^{p-1,q-1}(X)}.
\]
Its pushforward under embeddings can be constructed by realizing it as a current. The example in Section \ref{sec:example} shows concretely that a characteristic form may vanish in Dolbeault cohomology while remaining nonzero in Bott-Chern cohomology; a Dolbeault-valued formula would therefore lose part of the complex-analytic direct-image information \cite{MR1902630}.

Third, orbifold stabilizers carry representation-theoretic information that ordinary cohomology on the underlying space cannot detect. If $X=[U/G]$, the natural receptacle is the inertia orbifold
\[
  I[U/G]  =  \bigsqcup_{(g)\subset G}[U^g/Z_G(g)],
\]
whose points are pairs $(x,g)$ with $g$ in the stabilizer of $x$. On the $g$-sector, the Chern form contains the trace of $g$ on the bundle. An orbifold embedding may also enlarge stabilizers without changing the local underlying space, in which case the direct image includes induction of representations. This phenomenon is encoded by pushforward on the inertia orbifold. Describing orbifolds by proper \'etale groupoids and proving Morita invariance ensures that the resulting class depends on the orbifold, not on a chosen presentation \cite{moerdijk1997orbifolds,moerdijk2002orbifolds,adem2007orbifolds}.

These points mark the precise difference from index theory. Equivariant and Kawasaki-type index theorems integrate sectorwise characteristic forms to obtain an integer or a virtual character. We construct the sectorwise Bott-Chern class first and prove its functoriality under direct image. The result combines the fixed-sector geometry of orbifold index theory, the class-valued functoriality of Grothendieck-Riemann-Roch, and the finer complex-analytic information detected by Bott-Chern cohomology.

\subsection{Main results}

For a compact complex orbifold $X$, let $K(X)$ be the Grothendieck group of coherent analytic sheaves and let $IX$ be its inertia orbifold. We construct a homomorphism
\begin{equation}\label{eq:intro-ch}
  \chBC\colon K(X)\longrightarrow H_{\mathrm{BC}}^{(=)}(IX,\mathbb{C}),
  \qquad
  H_{\mathrm{BC}}^{(=)}(IX,\mathbb{C})  =  \bigoplus_p H_{\mathrm{BC}}^{p,p}(IX,\mathbb{C}).
\end{equation}
It is invariant under Morita equivalence and agrees, on holomorphic orbifold vector bundles, with Ma's orbifold Chern character \cite{ma2005orbifolds}. We then prove the following result. 

\begin{thm}\label{thm:RRG_embeddings}
Let $i_{X,Y}\colon X\hookrightarrow Y$ be an embedding of compact complex orbifold groupoids and let $\mathcal F\in D^b_{\mathrm{coh}}(X)$. Then
\begin{equation}\label{eq: GRR for embedding}
  \chBC(i_{X,Y,*}\mathcal F)
  =
  Ii_{X,Y,*}\!\left(\chBC(\mathcal F)\TdBC(N_{X/Y})^{-1}\right)
\end{equation}
in $H_{\mathrm{BC}}^{(=)}(IY,\mathbb{C})$. Here $Ii_{X,Y}\colon IX\to IY$ is the induced map, and the Todd factor is defined sectorwise using the stabilizer eigenbundle decomposition of the normal bundle.
\end{thm}

The formula (\ref{eq: GRR for embedding}) applies even when $i_{X,Y,*}\mathcal F$ has no finite locally free resolution. Its source-side expression simultaneously records the complex normal geometry and the induction of stabilizer representations. It extends Ma's vector-bundle construction to coherent analytic complexes, removes the finite-resolution hypothesis on the direct image, and includes embeddings that change isotropy.

\begin{thm}\label{thm:unique_Chern}
The homomorphism \eqref{eq:intro-ch} is the unique group homomorphism that agrees with Ma's character on holomorphic orbifold vector bundles, is functorial under pullback, and satisfies \eqref{eq: GRR for embedding} for orbifold embeddings.
\end{thm}

Theorem \ref{thm:unique_Chern} shows that the superconnection construction is not an arbitrary extension from vector bundles to coherent sheaves: normalization, pullback, and direct-image compatibility determine it uniquely. Theorem \ref{thm:RRG_embeddings} is therefore a structural characterization as well as a computational formula.

The present paper combines the fixed-sector geometry of Kawasaki and Ma with the superconnection approach of Bismut-Shen-Wei. It is analytic rather than algebraic: its classes are represented by differential forms and currents of pure type, without imposing an algebraic presentation or passing to a Chow theory. It is intrinsic to orbifolds rather than tied to one global group action: all constructions are Morita invariant. It also treats coherent complexes without assuming finite global locally free resolutions and permits stabilizer groups to change under embeddings.

After the first version of this work appeared, Xu proved an equivariant Bott-Chern Riemann-Roch-Grothendieck theorem for proper maps of compact complex manifolds with finite group actions \cite{Xu2025}. Xu treats a broader class of maps in the global-quotient setting. Our Riemann-Roch-Grothendieck theorem is restricted to embeddings, but applies to general compact complex orbifolds and includes isotropy-changing embeddings. The two results therefore address complementary aspects of equivariant and orbifold Riemann-Roch-Grothendieck theory.

\subsection{Construction and proof strategy}

Let $B(X)$ be the dg-category of flat antiholomorphic superconnections on $X$. We construct it from a groupoid presentation, prove Morita invariance, and establish an equivalence of the associated homotopy category $\underline{B}(X)$:
\begin{equation}\label{eq:intro-equivalence}
  D^b_{\mathrm{coh}}(X)  \simeq   \underline{B}(X).
\end{equation}
If $(E^\bullet,A_E'')$ represents a coherent complex and a generalized Hermitian metric determines the adjoint $A_E'$, then $A_E^2=[A_E'',A_E']$. Writing $\beta\colon IX\to X$ for the natural map and $\tau$ for the tautological stabilizer automorphism, the Chern form is, up to the conventional normalization,
\[
  \text{Tr}_s\!\left[\tau\exp(-\beta^*A_E^2)\right].
\]
The Bianchi identity gives closedness, variation of the generalized metric is $\partial\bar\partial$-exact, and the supertrace gives additivity under mapping cones. The resulting class therefore descends to the derived category and defines \eqref{eq:intro-ch}.

For Theorem \ref{thm:RRG_embeddings}, every orbifold embedding is factored as an iso-spatial embedding, which changes stabilizers without changing the local complex space, followed by a stabilizer-preserving embedding. This factorization separates the genuinely orbifold issue from the normal geometry. In the iso-spatial case, there is no normal-bundle correction: locally, the formula is the character identity for induction along an inclusion of finite stabilizer groups. In the stabilizer-preserving case, deformation to the normal cone reduces the problem to the zero section of the normal bundle, where the Koszul complex yields the inverse Bott-Chern Todd class. Combining the two cases gives the formula for an arbitrary embedding.

The uniqueness Theorem \ref{thm:unique_Chern} is proved by Hironaka's flattening result and d\'evissage in $G$-theory. After suitable blow-ups, one separates a locally free part and reduces the remaining contribution to coherent sheaves with lower-dimensional support. This explains why the embedding formula, rather than a particular superconnection representative, characterizes the final class.

\subsection{Outlook and organization}

Theorem \ref{thm:RRG_embeddings} is a first step toward establishing the Riemann-Roch-Grothendieck formula for general proper holomorphic maps of complex orbifolds. A map $f\colon X\to Y$ factors through its graph
\[
  X\hookrightarrow X\times Y\longrightarrow Y.
\]
The present paper treats the graph embedding, while the projection requires a families theorem for orbifold Dolbeault operators. Such a theorem would combine the superconnection model for coherent complexes with the analytic variation of a family and with sectorwise isotropy data. The construction developed here isolates the characteristic class and the embedding mechanism that any such extension must recover. It also provides a basis for twisted versions involving flat $\mathbb{C}^*$-gerbes.

Section \ref{sec:gpd_orb_reviews} develops groupoids and orbifold embeddings. Sections \ref{sec:H_BC} and \ref{section: coherent sheaves on complex orbifolds}
treat Bott--Chern cohomology, inertia groupoids, and coherent orbifold sheaves. Sections \ref{Section: ahfs}--\ref{sec:chern} construct flat antiholomorphic superconnections, generalized metrics, curvature, and the Chern character. Section \ref{sec:RRG_embeddings} proves Theorems~\ref{thm:RRG_embeddings} and~\ref{thm:unique_Chern}, and Section \ref{sec:example} gives the non-K\"ahler example. The appendix contains results about coherent sheaves.

\subsection{Acknowledgment}
We would like to thank J. Block, X. Ma, I. Moerdijk, S. Shen, and T. Y. Yu for inspiring discussions. We would like to thank G. Xu for helpful discussions and suggestions on Lemma \ref{lemma: global resolution special case}. We are very grateful to the anonymous referees for valuable comments and suggestions. Ma's research was supported by the NSF grant DMS-2247313. Tang's research is supported in part by NSF grants DMS-1952551, DMS-2350181, and Simons Foundation Collaboration Grant MPS-TSM-00007714.  Tseng's research is supported in part by Simons Foundation Collaboration Grant \# 636211. Wei's research is supported in part by AMS-Simons Research Enhancement Grants for PUI Faculty. GPT-5.6 Sol was used to improve the writing of the paper.

\section{A review of groupoids and orbifolds}\label{sec:gpd_orb_reviews}

This section consists of two parts. In the first part, we review the related material for groupoids and orbifolds; in the second part, we introduce the groupoid presentation of an orbifold embedding, which is crucial for our development of an orbifold Riemann-Roch-Grothendieck theorem. 
\subsection{Lie groupoids}
We briefly review the groupoid approach to orbifolds, mostly following \cite{moerdijk2002orbifolds}

\begin{defi}\label{defi: groupoid}
A groupoid is a (small) category in which each morphism is invertible. Alternatively, 
a groupoid $\mathcal{G}$ consists of a set $G_0$ of objects and a set $G_1$ of arrows.  There are maps $s$ and $t\colon G_1\rightrightarrows G_0$ which are called the source map and the target map, respectively. Moreover we have the unit map $u\colon G_0\to G_1$, the inverse map $i\colon G_1\to G_1$, and the composition map $m\colon G_1\times_{G_0}G_1\to G_1$.
\end{defi}

\begin{defi}\label{defi: groupoid homomorphism}
A homomorphism $\phi:\mathcal{H}\to\mathcal{G}$ is by definition a functor. In more detail, a homomorphism consists of two maps (both) denoted by $\phi: H_0\to G_0$ and $\phi: H_1\to G_1$, which together commute with all structure maps.
\end{defi}

We need to consider groupoids in the category of smooth manifolds. 

\begin{defi}\label{defi: Lie groupoid}
A Lie groupoid is a groupoid $\mathcal{G}$ for which $G_0$ and $G_1$ are smooth manifolds and the structure maps $s$, $t$, $u$, $i$, and $m$ are smooth. Furthermore, $s$ and $t\colon G_1\rightrightarrows G_0$  are required to be submersive.
\end{defi}

\begin{defi}\label{defi:Lie groupoid homomorphism}
A homomorphism $\phi\colon\mathcal{H}\to\mathcal{G}$ between Lie groupoids is by definition a smooth functor. In more detail, a homomorphism consists of two smooth maps (both) denoted by $\phi\colon H_0\to G_0$ and $\phi\colon H_1\to G_1$, which together commute with all structure maps.
\end{defi}

\begin{defi}\label{defi: groupoid equivalence}
A homomorphism $\phi\colon\mathcal{H}\to\mathcal{G}$ between Lie groupoids is called an equivalence if
\begin{enumerate}
\item the following map is a surjective submersion:
$$
t\circ \pi_1\colon G_1\times_{G_0}H_0\to G_0;
$$
\item the following square is cartesian:
{
$$
\begin{CD}
H_1 @>\phi>> G_1\\
@V(s,t)VV @VV(s,t)V\\
H_0\times H_0 @>\phi\times \phi>> G_0\times G_0.
\end{CD}
$$}
\end{enumerate}
\end{defi}

\begin{defi}\label{defi: groupoid Morita equivalence}
Two Lie groupoids $\mathcal{G}$ and $\mathcal{G}^{\prime}$ are called Morita equivalent if there exists a third groupoid $\mathcal{H}$ and equivalences
$$
\mathcal{G}\overset{\phi}{\leftarrow} \mathcal{H} \overset{\phi^{\prime}}{\rightarrow} \mathcal{G}^{\prime}.
$$
\end{defi}

It can be shown that this defines an equivalence relation on Lie groupoids. To define an orbifold, we will work with a proper \'etale Lie groupoid.

\begin{defi}\label{defi: proper groupoid}
A Lie groupoid $\mathcal{G}$ is called a proper groupoid if the map $(s,t)\colon G_1\to G_0\times G_0$ is proper.
\end{defi}

\begin{defi}\label{defi: etale groupoid}
A Lie groupoid $\mathcal{G}$ is called an \'{e}tale groupoid if $s$ and $t$ are local diffeomorphisms.
\end{defi}

\begin{defi}\label{defi: effective groupoid}
Let $\mathcal{G}$ be an \'{e}tale groupoid. Then any arrow $g\colon x\to y$ in $\mathcal{G}$ induces a well-defined germ of a diffeomorphism $\tilde{g}\colon (U_x, x)\xrightarrow{\sim} (V_y,y)$ as $\tilde{g}=t\circ \hat{g}$, where $\hat{g}\colon U_x\to G_1$ is a section of the source map $s\colon G_1\to G_0$ defined on a sufficiently small neighborhood $U_x$ of $x$ and with $\hat{g}(x)=g$.

We call $\mathcal{G}$ effective (or reduced) if the assignment $g\mapsto \tilde{g}$ is faithful, or equivalently, if for each $x\in G_0$ the map $g\mapsto \tilde{g}$ is an injective group homomorphism $G_x\to \Diff_x(G_0)$.
\end{defi}

Now we can define orbifolds.

\begin{defi}\label{defi: orbifold}
An orbifold groupoid is a proper \'etale groupoid $\mathcal{G}$, and an orbifold $X$ has its underlying space given by the quotient space $G_0/G_1$ of an orbifold groupoid.

Two orbifolds $X=[G_0/G_1]$, $Y=[H_0/H_1]$ are said to be equivalent if the corresponding proper \'etale groupoid presentations $\mathcal{G}$ and $\mathcal{H}$ are Morita equivalent. 
\end{defi}

Next we describe the notion of groupoid actions on spaces.

\begin{defi}\label{defi: G-space and vector bundle}
Let $\mathcal{G}$ be a Lie groupoid. A (right) $\mathcal{G}$-space is a manifold $M$ equipped with an action by $\mathcal{G}$. This action is given by two smooth maps\footnote{The map $\pi$ is usually assumed to be submersive.} 
\begin{equation*}
\pi\colon M\to G_0,\quad \mu\colon M\times_{G_0}G_1\to M,
\end{equation*}
which satisfy the usual identities for an action. A left $\mathcal{G}$-space is similarly defined.

A vector bundle over $\mathcal{G}$ is a $\mathcal{G}$-space $E$ for which $\pi\colon E\to G_0$ is a vector bundle and the action of $\mathcal{G}$ on $E$ is fiberwise linear. A left vector bundle over $\mathcal{G}$ is similarly defined.
\end{defi}

We need the following construction.

\begin{defi}\label{defi: fiber product}
Let $\mathcal{G}$ be a Lie groupoid. For a right $\mathcal{G}$-space $E$ and a left $\mathcal{G}$-space $F$, we define 
$$
E\times_{\mathcal{G}} F:=E\times_{G_0}F/\sim,
$$
where $\sim$ is the relation $(eg,f)\sim (e,gf)$ for $g\in G_1$. $E\times_{\mathcal{G}} F$ is called the fiber product of $E$ and $F$.
\end{defi}

\begin{defi}\label{defi: principal G-bundle}
Let $\mathcal{G}$ be a Lie groupoid. A (right) $\mathcal{G}$-space $P$ is called free if $pg_1=pg_2$ implies $g_1=g_2$. It is called proper if the map $P\times_{G_0} G_1\to P\times_{G_0}P$, $(p,g)\mapsto (p, pg)$, is proper.

A (right) $\mathcal{G}$-space $P$ is called a principal $\mathcal{G}$-bundle if it is both free and proper. A left principal $\mathcal{G}$-bundle is defined similarly.
\end{defi}

\begin{defi}\label{defi: bi-bundle}
Let $\mathcal{G}$ and $\mathcal{H}$ be two Lie groupoids. A $\mathcal{G}$-$\mathcal{H}$ principal bibundle is a  manifold $P$ with commuting left
$\mathcal{G}$-space and right $\mathcal{H}$-space structures that are both principal, and satisfy the
following extra conditions:
\begin{enumerate}
\item The quotient space $\mathcal{G}\backslash P$ (with its quotient topology) is diffeomorphic to $H_0$ in a way
that identifies the map $P\to H_0$ with the quotient map $P\to \mathcal{G}\backslash P$.
\item The quotient space $P\slash \mathcal{H}$ (with its quotient topology) is diffeomorphic to $G_0$ in a way
that identifies the map $P\to G_0$ with the quotient map $P\to  P\slash \mathcal{H}$.
\end{enumerate}
\end{defi}

The following proposition gives an equivalent definition of Morita equivalence.
\begin{prop}\label{prop: equiv defi of Morita equiv}
Two Lie groupoids $\mathcal{G}$ and $\mathcal{H}$ are Morita equivalent in the sense of Definition \ref{defi: groupoid Morita equivalence} if and only if there exists a $\mathcal{G}$-$\mathcal{H}$ principal bibundle $P$.
\end{prop}
\begin{proof}
See \cite[Theorem 2.2]{behrend2011differentiable}.
\end{proof}

\subsection{Complex groupoids}
We consider groupoids in the category of complex manifolds.
\begin{defi}\label{defi: complex Lie groupoid}
A Lie groupoid $\mathcal{G}=(G_0, G_1)$ is called a complex Lie groupoid if $G_0$ and $G_1$ are complex manifolds and all structure maps are holomorphic.
\end{defi}

\begin{defi}\label{defi: complex orbifold}
A \emph{complex orbifold} is an orbifold $X$ that has a representation by a complex proper \'etale Lie groupoid $\mathcal{G}$. We call a complex proper \'etale Lie groupoid a \emph{complex orbifold groupoid}.

A morphism between complex orbifold groupoids is a generalized morphism in the sense of Definition \ref{defi: generalized morphism} below where diffeomorphisms are replaced by holomorphic maps. 

Two complex orbifold groupoids are Morita equivalent if there is a principal bibundle that carries a bi-invariant complex structure such that the associated structure maps are holomorphic. 
\end{defi}

\begin{rmk}\label{rmk:global quotient}
In this paper, we assume to work with proper \'etale groupoids that are Morita equivalent to compact group actions. Under this assumption, a complex orbifold groupoid is Morita equivalent to a transformation groupoid associated with a compact group $K$ action on a manifold $M$ that carries a $K$-invariant transverse holomorphic structure. 
\end{rmk}
 
See \cite{tommasini2012orbifolds} for relations between Definition \ref{defi: complex orbifold} and other definitions of complex orbifolds.  

In the remainder of this section, we discuss a few basic facts about complex orbifolds.

\begin{defi}\label{defi: compact complex orbifold}
A complex orbifold represented by a groupoid $\mathcal{G}$ is \emph{compact} if its underlying topological space $G_0/G_1$ is compact. 
\end{defi}

\begin{defi}\label{defi: neighborhood}
Let $\mathcal{G}=(G_0,G_1)$ be a complex orbifold groupoid and $x\in G_0$ be a point. Then there exists a neighborhood $U$ of $x$ in $G_0$ such that $\mathcal{G}|_U$ is the action groupoid of a finite group $G$ acting holomorphically on $U$. We call $(U, G)$ a neighborhood of $x$ in $\mathcal{G}$.
\end{defi}

\begin{prop}\label{prop: existence of nbh}
For any point $x\in G_0$, a neighborhood of $x$ in $\mathcal{G}$ exists.
\end{prop}
\begin{proof}
This is a special case of the linearization theorem \cite{weinstein:linearization}.
\end{proof}

For compact complex orbifolds (as in Definition \ref{defi: compact complex orbifold}), we have the following further result.

\begin{prop}\label{prop: local chart compact}
For a compact complex orbifold, we can choose a representing orbifold groupoid $\mathcal{G}=(G_0,G_1)$ such that for every $x\in G_0$ and $U$ an invariant open neighborhood of $x$ in $G_0$, there is an arbitrarily small $\mathcal{G}$-invariant open subset $V$ of $G_0$ satisfying the following properties:
\begin{itemize}
\item $V$ is a disjoint union of open subsets $\{V_i\}$ such that each point in the $\mathcal{G}$-orbit of $x$ belongs to one and only one $V_i$;
\item  the restriction of $\mathcal{G}$ to each $V_i$ is isomorphic to the transformation groupoid $\Gamma_i\rtimes V_i\rightrightarrows V_i$ with $\Gamma_i$ a finite group. 
\end{itemize}
Such a neighborhood $V$ will be called a $\mathcal{G}$-invariant neighborhood throughout the paper. 
\end{prop}
\begin{proof}
Since the orbifold is compact, we can choose to work with a special orbifold groupoid $\mathcal{G}=(G_0,G_1)$ such that every $\mathcal{G}$-orbit on $G_0$ is a finite set. This property, together with the linearization theorem for proper \'etale groupoids \cite{weinstein:linearization}, gives the claim.
\end{proof}

\subsection{Generalized morphisms}
\label{subsect: proper generalized morphisms}
We introduce the notion of generalized morphisms between groupoids.
\begin{defi}\label{defi: generalized morphism}
Let $\mathcal{G}$ and $\mathcal{H}$ be two Lie groupoids. A generalized morphism from $\mathcal{G}$ to $\mathcal{H}$ is a triple $(Z,\rho, \sigma)$ where $Z$ is a manifold endowed with a left $\mathcal{G}$-action and a free and proper right $\mathcal{H}$-action which commute and $\rho, \sigma$ are maps,
$
G_0\overset{\rho}{\leftarrow} Z \overset{\sigma}{\rightarrow} H_0,
$
such that $\rho$ induces a diffeomorphism $Z/\mathcal{H}\simeq G_0$.
\end{defi}

\begin{rmk}\label{rmk: groupoid from generalized morphism}
We can define a groupoid $\mathcal{Z}=(Z_0, Z_1)$ from a generalized morphism $(Z,\rho, \sigma)$ where $Z_0=Z$ and 
$$
Z_1=\{(z,g,h)\in Z\times G_1\times H_1| s(g)=\rho(z),~ t(h)=\sigma(z)\}.
$$
We define $s(z,g,h)=z$ and $t(z,g,h)=gzh$. In addition, we define
$$
(z_1,g_1,h_1)\cdot (z_2,g_2,h_2)=(z_2, g_1g_2, h_2h_1), \quad \text{when } z_1=g_2z_2h_2.
$$
We have natural groupoid homomorphisms $\phi_{\rho}\colon \mathcal{Z}\to \mathcal{G}$ and $\phi_{\sigma}\colon \mathcal{Z}\to \mathcal{H}$. We define
$$
\phi_{\rho}(z)=\rho(z) \text{ and } \phi_{\sigma}(z)=\sigma(z) \quad \text{on }Z_0;\qquad \phi_{\rho}(z,g,h)=g \text{ and } \phi_{\sigma}(z,g,h)=h^{-1}\quad \text{on } Z_1.
$$
It is easy to check that $\phi_{\rho}$ and $\phi_{\sigma}$ are groupoid homomorphisms and $\phi_{\rho}$ is a Morita equivalence morphism. If $(Z,\rho, \sigma)$ is a principal bibundle, then $\phi_{\sigma}$ is also a Morita equivalence morphism.
\end{rmk}

\begin{rmk}
Locally, an orbifold groupoid is a transformation groupoid associated with a finite group $G$ action on a manifold $X$. In this way, a generalized morphism may be represented as a groupoid homomorphism $G\ltimes X\to H\ltimes Y$  defined by a map $\underline{f}\colon X\to Y$ equivariant with respect to the group homomorphism $\bar{f}: G\to H$.
\end{rmk}

We next introduce the notion of proper morphisms between complex orbifold groupoids following \cite[Section 7]{tu2004non}. We begin with some basic constructions.

For a generalized morphism $(Z,\rho, \sigma)$ from $\mathcal{G}$ to $\mathcal{H}$, we consider the quotient $\mathcal{G}\backslash Z$. Let $\tilde{\rho}$ be the map
$$
\tilde{\rho}\colon \mathcal{G}\backslash Z\to \mathcal{G}\backslash G_0
$$
induced by $\rho\colon Z\to G_0$. We have the following commutative diagram,
{
\begin{equation}
\begin{CD} G_0 @<\rho<< Z\\ @VV\pi_{\mathcal{G}}V @VV \pi_{\mathcal{G}\backslash Z}V\\ \mathcal{G}\backslash G_0 @<\tilde{\rho}<< \mathcal{G}\backslash Z \end{CD}
\end{equation}}

Meanwhile, since $\sigma$ satisfies $\sigma(gz)=\sigma(z)$, the map $\sigma\colon Z\to H_0$ induces a map
\begin{equation}\label{eq: sigma tilde}
\tilde{\sigma}\colon \mathcal{G}\backslash Z\to H_0.
\end{equation}
Furthermore, $\sigma$ and $\tilde{\sigma}$ are both $\mathcal{H}$-invariant maps, hence they induce maps
$$
\theta\colon Z/\mathcal{H}\cong G_0\to H_0/\mathcal{H}, \quad \tilde{\theta}\colon \mathcal{G}\backslash Z/\mathcal{H}\to H_0/\mathcal{H}
$$
which make the following diagram commute:
{
\begin{equation}
\begin{CD} Z @>\sigma>> H_0\\ @VV\pi_{\mathcal{H}\backslash Z}V @VV \pi_{\mathcal{H}}V\\ G_0 \cong Z/\mathcal{H} @>\theta>> H_0/\mathcal{H} \\  @VVV @VV\id V \\  \mathcal{G}\backslash Z/\mathcal{H}@>\tilde{\theta}>> H_0/\mathcal{H}.
\end{CD}
\end{equation}}

We know that $\rho$ induces a diffeomorphism $Z/\mathcal{H}\simeq G_0$, hence $\mathcal{G}\backslash Z/\mathcal{H}\simeq  \mathcal{G}\backslash G_0$ and the map $Z/\mathcal{H}\to \mathcal{G}\backslash Z/\mathcal{H}$ agrees with the quotient map $G_0\to   \mathcal{G}\backslash G_0$. 

In summary, we have the following commutative diagram,
{
\begin{equation}
\begin{tikzcd}
Z \arrow[d, "\pi_{\mathcal{G}\backslash Z}"'] \arrow[rd, "\sigma"]                                                        &                                    \\
\mathcal{G}\backslash Z \arrow[d, "\tilde{\rho}"'] \arrow[r, "\tilde{\sigma}"] & H_0 \arrow[d, "\pi_{\mathcal{H}}"] \\
\mathcal{G}\backslash G_0 \arrow[r, "\tilde{\theta}"]                                   & H_0/\mathcal{H}.                \end{tikzcd}
\end{equation}}

The following definitions are taken from \cite[Introduction]{tu2004non}.
\begin{defi}\label{defi: G-compact}
Let $Z$ be a space with a $\mathcal{G}$-action. A subspace $Y$ of $Z$ is called $\mathcal{G}$-compact if $Y$ is $\mathcal{G}$-invariant and $Y/\mathcal{G}$ is compact.
\end{defi}

\begin{defi}\label{defi: locally proper morphism}
A generalized morphism $(Z,\rho, \sigma)\colon\mathcal{G}\to \mathcal{H}$ between Lie groupoids is called \emph{locally proper} if the action of $\mathcal{G}$ on $Z$ is proper.
\end{defi}

\begin{rmk}
Let $f \colon \mathcal{G}\to \mathcal{H}$ be a groupoid morphism. By \cite[Proposition 7.4]{tu2004non}, the associated generalized groupoid morphism is locally proper if and only if  $(f, r, s) \colon G_1\to 
H_1 \times G_0
\times G_0$ is proper.
\end{rmk}

\begin{prop}\label{prop: locally proper between proper groupoids}
If $\mathcal{G}$ is a proper Lie groupoid, then any generalized morphism $(Z,\rho, \sigma)\colon\mathcal{G}\to \mathcal{H}$  is locally proper.
\end{prop}
\begin{proof}
The map $\phi\colon G_1\times_{G_0} Z\to Z\times Z$ defined by $\phi(g, z)=(gz,z)$ fits in the following diagram
{
\begin{equation}
\begin{tikzcd}
G_1\times_{G_0} Z \arrow[rd, "{(t,s)}"'] \arrow[r, "\phi"] & Z\times Z \arrow[d, "{( \rho, \rho)}"] \\
                                                           & G_0\times G_0,                         
\end{tikzcd}
\end{equation}}
which is commutative. For any compact subset $K\subset Z$, $K\times K$ is a compact subset of $Z\times Z$. Hence $(\rho, \rho)(K\times K)$ is a compact subset in $G_0\times G_0$. Since $\mathcal{G}$ is a proper Lie groupoid, the preimage $(t,s)^{-1}((\rho, \rho)(K\times K))$ is a compact subset in $G_1\times_{G_0} Z$. Since the diagram commutes, we know that $\phi^{-1}(K\times K)\subset (t,s)^{-1}((\rho, \rho)(K\times K))$ is also compact. So the map $G_1\times_{G_0}Z\to Z\times Z$ is proper.
\end{proof}

Now we are ready to present the main definition of this section.

\begin{defi}\label{defi: proper morphism}
A generalized morphism $(Z,\rho, \sigma)\colon\mathcal{G}\to \mathcal{H}$ between Lie groupoids is called \emph{proper} if it is locally proper and, in addition, for any compact subset $K$ in $H_0$, $\sigma^{-1}(K)$ is $\mathcal{G}$-compact.
\end{defi}

\begin{rmk}
$(Z,\rho, \sigma)$ is proper if and only if it is locally proper and the map $\tilde{\sigma}$ in \eqref{eq: sigma tilde} is proper.
\end{rmk}

\subsection{Orbifold embeddings}\label{subsec: orbifold embeddings}
A morphism $f\colon \mathcal{G}\to \mathcal{H}$ between Lie groupoids is called an \emph{embedding} if each $f_i\colon G_i\to H_i$, $i=0,1$, is an embedding of manifolds. The definition of embedding of orbifolds, which we develop in Definition \ref{defi: orbifold embedding} below, is more involved.

Let $\mathcal{G}$ and $\mathcal{H}$ be complex orbifold groupoids and $(Z,\rho, \sigma)$ be a generalized morphism from $\mathcal{G}$ to $\mathcal{H}$. Let $\tilde{\mathcal{G}}$ be the pullback of $\mathcal{G}$ to $Z$ under $\rho\colon Z\to G_0$. More explicitly, we have $\tilde{\mathcal{G}}=(\widetilde{G}_0,\widetilde{G}_1)$ where 
\begin{equation}
\widetilde{G}_0=Z\,\, \text{and } \widetilde{G}_1=\{(z_1,g,z_2)\in Z\times G_1\times Z|\rho(z_1)=t(g),~\rho(z_2)=s(g)\}.
\end{equation}
The source and target maps of $\tilde{\mathcal{G}}$ are the obvious projections to $Z$. Similarly, we can define $\tilde{\mathcal{H}}$ to be the pullback of $\mathcal{H}$ to $Z$ under $\sigma\colon Z\to H_0$.
More explicitly, we have $\tilde{\mathcal{H}}=(\widetilde{H}_0,\widetilde{H}_1)$ where $\widetilde{H}_0=Z$ and 
\begin{equation}\label{eq: tilde H}
\widetilde{H}_1=\{(z_1,h,z_2)\in Z\times H_1\times Z|\sigma(z_1)=t(h),~\sigma(z_2)=s(h)\}.
\end{equation}
We can define a morphism $\phi\colon \tilde{\mathcal{G}}\to \tilde{\mathcal{H}}$ as follows: $\phi_0\colon \widetilde{G}_0\to \widetilde{H}_0$ is the identity map on $Z$. As for $\phi_1\colon \widetilde{G}_1\to \widetilde{H}_1$, consider an element $(z_1,g,z_2)\in \widetilde{G}_1$. Notice that we have
$$
\rho(gz_2)=t(g)=\rho(z_1)\in G_0.
$$
Since the $\mathcal{H}$-action on $Z$ is free and $\rho$ induces an isomorphism $Z/\mathcal{H}\overset{\simeq}{\longrightarrow} G_0$, there exists a unique $h\in H_1$ such that
$$
gz_2=z_1h.
$$
We define $\phi_1(z_1,g,z_2)$ to be 
\begin{equation}
\phi_1(z_1,g,z_2):=(z_1,h,z_2).
\end{equation}

\begin{defi}\label{defi: orbifold embedding}
A generalized morphism $(Z,\rho, \sigma)$ from $\mathcal{G}$ to $\mathcal{H}$ is an \emph{embedding} if it satisfies:
\begin{enumerate}
\item The map $\sigma\colon Z\to H_0$ is a local embedding.
\item The image $\sigma(Z)$ is a closed subset of $H_0$ with locally {\em finite resolution}, i.e. for any $x$ in the image of $\sigma\colon Z\to H_0$, there is a neighborhood $U$ of $x$, and there are $z_1, ..., z_k\in Z$ with neighborhoods $V_1, ..., V_k$ such that for all $i$, $\sigma|_{V_i}: V_i\to \sigma(V_i)$ is a diffeomorphism, and $\sigma(Z)\cap U=\bigcup_{i}\sigma(V_i)$. 
\item The morphism $\phi\colon \tilde{\mathcal{G}}\to \tilde{\mathcal{H}}$ is a \emph{saturated embedding} of groupoids.
\end{enumerate}
Here, ``saturated'' means that for any $h\in \widetilde{H}_1$, there exists a $g\in \widetilde{G}_1$ such that $s(h)=s(\phi(g))$ and $t(h)=t(\phi(g))$.
\end{defi}

We explain Condition (2) in Definition \ref{defi: orbifold embedding} with the following example. 

\begin{eg}\label{ex:Z2embedding} Consider the $\mathbb{Z}_2$ action on $\mathbb{C}$ by $z\mapsto -z$. The associated groupoid is the transformation groupoid $\mathcal{G}=\mathbb{Z}_2\ltimes \mathbb{C}\rightrightarrows \mathbb{C}$. Similarly, we consider the $\mathbb{Z}_2\times \mathbb{Z}_2$ action on $\mathbb{C}\times \mathbb{C}$ by 
\begin{equation*}
\mathbb{Z}_2\times \mathbb{Z}_2\ni (\gamma_1, \gamma_2):  (z_1, z_2)\mapsto (\gamma_1(z_1), \gamma_2(z_2)) .
\end{equation*}
 The associated groupoid is $\mathcal{H}=(\mathbb{Z}_2\times \mathbb{Z}_2)\ltimes (\mathbb{C}\times \mathbb{C})\rightrightarrows \mathbb{C}\times \mathbb{C}$. We consider the morphism $f:\mathcal{G}\to \mathcal{H}$ by 
\[
f_0(z)=(z, z),\ f_1(\gamma, z)=(\gamma, \gamma, z, z),\quad \gamma\in \mathbb{Z}_2, z\in \mathbb{C}. 
\]
It is not hard to see that $f$ represents the diagonal embedding $$\Delta: [\mathbb{C}/\mathbb{Z}_2] \hookrightarrow[\mathbb{C}/\mathbb{Z}_2]\times [\mathbb{C}/\mathbb{Z}_2].$$ 

The bibundle $C$ is of the form $C=\mathbb{C}\times \mathbb{Z}_2\times \mathbb{Z}_2$ with the map $\sigma: C\to \mathcal{H}_0=\mathbb{C}\times \mathbb{C}$ by 
\[
\sigma(z, \gamma_1, \gamma_2)=(\gamma_1(z), \gamma_2(z)),\quad z\in \mathbb{C},\ (\gamma_1, \gamma_2)\in \mathbb{Z}_2\times \mathbb{Z}_2.
\]
We observe that the map $\sigma$ is an embedding on each component of $C=\mathbb{C}\times \mathbb{Z}_2\times \mathbb{Z}_2$. And the image of $\sigma$ is the union of the two axes $\{(z, 0)\in \mathbb{C}\times \mathbb{C}\,|\, z\in \mathbb{C}\}\cup \{(0,z)\in \mathbb{C}\times \mathbb{C}\,|\, z\in \mathbb{C}\}$, which is a closed subset of $\mathbb{C}\times \mathbb{C}$, but not a submanifold of $H_0$. 

We can choose $U$ to be the whole manifold $H_0=\mathbb{C}\times \mathbb{C}$, and $V_1=\mathbb{C}\times \{e\}\times \{e\}$, and $V_2=\mathbb{C}\times \{\gamma\}\times \{e\}$, where $e$ is the identity element of $\mathbb{Z}_2$ and $\gamma$ is the nontrivial element. It is straightforward to check that the restrictions of $\sigma$ on $V_1$ and $V_2$ are embeddings, and 
$\sigma(V_1)\cup\sigma(V_2)=\sigma(C)\cap U=\sigma(C).$ 
This shows that $\sigma(C)$ has the locally finite resolution property as in Definition \ref{defi: orbifold embedding} (2). 
\end{eg}

For our purposes, we need to distinguish the following two types of groupoid embeddings.

\begin{defi}\label{defi: stabilizer-preserving embedding}
A generalized morphism $(Z,\rho, \sigma)$ from $\mathcal{G}$ to $\mathcal{H}$ is a \emph{stabilizer-preserving} embedding if it is an embedding in the sense of Definition \ref{defi: orbifold embedding} and, in addition, the morphism $\phi\colon \tilde{\mathcal{G}}\to \tilde{\mathcal{H}}$ is an isomorphism of groupoids.
\end{defi}

\begin{rmk}
The notion of stabilizer-preserving morphism is also considered in a different context in \cite[Definition 5.2]{CDH}.    
\end{rmk}

\begin{defi}\label{defi: iso-spatial embedding}
A generalized morphism $(Z,\rho, \sigma)$ from $\mathcal{G}$ to $\mathcal{H}$ is an \emph{iso-spatial} embedding if it is an embedding in the sense of Definition \ref{defi: orbifold embedding} and, in addition, the map $\sigma\colon Z\to H_0$ is a surjective submersion.
\end{defi}

\begin{rmk}\label{rmk: two kind embeddings local case}
Locally, the groupoid $\mathcal{G}$ is given by a manifold $X$ with an action of a finite group $G$. A similar description is valid for the groupoid $\mathcal{H}$. When we restrict to the local case, an iso-spatial embedding is given by an inclusion of the finite group $\bar{f}\colon G\hookrightarrow H$ and the identity map on the base manifold $X$; and a stabilizer-preserving embedding is given by the identity map of the finite group and an embedding of complex manifolds $\underline{f}\colon X\hookrightarrow Y$.
\end{rmk}

We have the following factorization result for embeddings.

\begin{prop}\label{prop: decompose an embedding into two types}
Any groupoid embedding can be written as the composition of an iso-spatial embedding followed by a stabilizer-preserving embedding.
\end{prop}
\begin{proof}
Let $(Z,\rho, \sigma)$ be an embedding from $\mathcal{G}$ to $\mathcal{H}$. We consider the groupoid $\mathcal{H}|_{\text{Im} \sigma}$. Since $(Z,\rho, \sigma)$ is an embedding, the generalized morphism $(Z,\rho, \sigma)$ from $\mathcal{G}$ to $\mathcal{H}|_{\text{Im} \sigma}$ is an iso-spatial embedding, and the morphism $\sigma\colon \mathcal{H}|_{\text{Im} \sigma}\hookrightarrow \mathcal{H}$ is a stabilizer-preserving embedding. It is easy to check that their composition is exactly $(Z,\rho, \sigma)$.
\end{proof}

\subsection{Inertia groupoids}\label{subsec: inertia groupoids}
Following \cite[Section 6.4]{moerdijk2002orbifolds}, we can define the inertia groupoid of a complex orbifold groupoid $\mathcal{G}=(G_0, G_1)$. We consider
\begin{equation}
B_{\mathcal{G}}=\{g\in G_1| s(g)=t(g)\}.
\end{equation}
It is easy to see that $B_{\mathcal{G}}$ is a disjoint union of complex manifolds whose dimensions may vary with different components even when $G_0$ is connected. We also note that $\mathcal{G}$ acts on the space $B_{\mathcal{G}}$ by conjugation. Then the \emph{inertia groupoid} $I\mathcal{G}$ is defined to be the semi-direct product $B_{\mathcal{G}}\ltimes \mathcal{G}$. Namely,
\begin{equation}
(I\mathcal{G})_0=B_{\mathcal{G}},
\end{equation}
and for $g_1$ and $g_2\in B_{\mathcal{G}}$, an arrow $g_1\to g_2$ consists of $g\in G_1$ such that $gg_1g^{-1}=g_2$. In other words
\begin{equation}
(I\mathcal{G})_1=\{(g_1,g_2,g)|g_1,g_2\in B_{\mathcal{G}} \text{ and } gg_1g^{-1}=g_2\},
\end{equation} 
and $s(g_1,g_2,g)=g_1$, $t(g_1,g_2,g)=g_2$. It is clear that $I\mathcal{G}$ is also a proper, \'{e}tale complex Lie groupoid.

\begin{defi}\label{defi: beta_G}
There is a proper holomorphic groupoid map $\beta_{\mathcal{G}}\colon I\mathcal{G}\to \mathcal{G}$ which maps $g\in (I\mathcal{G})_0$ to $s(g)=t(g)\in G_0$ and maps $(g_1,g_2,g)\in (I\mathcal{G})_1$ to $g\in G_1$.
\end{defi}

\begin{defi}\label{defi: tau_G}
On $I\mathcal{G}$ we have the tautological section $\tau_{\mathcal{G}}$ with value in $G_1$ given by
\begin{equation}
\tau_{\mathcal{G}}(g)=(g,g,g) \text{, for } g\in (I\mathcal{G})_0.
\end{equation}
\end{defi}

\begin{rmk}\label{rmk:inertia_local}
Locally an orbifold groupoid $\mathcal{G}$ is a finite group $G$ acting on a complex manifold $X$. In this case, the inertia groupoid $I\mathcal{G}$ is
\begin{equation}\label{eq: inertia groupoid local 1}
G \ltimes \coprod_{g\in G} X^g,
\end{equation}
where $X^g$ is the fixed point set of $g$, and the $G$-action is given as follows: for any $g'\in G$,
\begin{equation}
g'\colon X^g\to X^{g'g(g')^{-1}}, \quad x\mapsto g'x.
\end{equation}
It is easy to see that  \eqref{eq: inertia groupoid local 1} is Morita equivalent to 
\begin{equation}\label{eq: inertia groupoid local 2}
 \coprod_{(g)\in \text{Conj}(G)}Z_G(g)\ltimes X^g,
\end{equation}
where $\text{Conj}(G)$ is the set of conjugacy classes of $G$ and $Z_G(g)$ is the centralizer of $g$ in $G$.

From the viewpoint of \eqref{eq: inertia groupoid local 2}, the map $\beta_{\mathcal{G}}$ is the natural map 
$$
Z_G(g)\ltimes X^g\to G\ltimes X,
$$
which are inclusions on both $Z_G(g)$ and $X^g$. The map $\tau_{\mathcal{G}}$ is given on each component by
\begin{equation}
\tau_{\mathcal{G}}(x)=(g,x)\text{, for } x\in X^g.
\end{equation}
\end{rmk}

Next, we consider induced (generalized) morphisms between inertia groupoids.
\begin{defi}\label{defi: generalized morphism for inertial groupoid}
Let $(Z,\rho, \sigma)\colon \mathcal{G}\to \mathcal{H}$ be a generalized morphism. We define a generalized morphism
$$
(IZ,I\rho, I\sigma)\colon I\mathcal{G}\to I\mathcal{H}
$$
as follows: let 
\begin{equation}\label{eq: IZ}
IZ:=\{(g,z,h)\in G_1\times Z\times H_1|gzh=z\}.
\end{equation}
This implies
\begin{equation}
s(g)=t(g)=\rho(z), \text{ and }  s(h)=t(h)=\sigma(z),
\end{equation}
hence we define $I\rho\colon IZ\to (I\mathcal{G})_0$ and $I\sigma\colon IZ\to (I\mathcal{H})_0$ as
\begin{equation}\label{eq: Irho and Isigma}
I\rho(g,z,h):=g \text{, and } I\sigma(g,z,h):=h.
\end{equation}

The groupoid $I\mathcal{G}$ acts on $IZ$ from the left by
\begin{equation}\label{eq IG acts on IZ}
(g_1,g_2,g)\cdot (g_1,z,h):=(g_2,gz,h)=(gg_1g^{-1},gz,h),
\end{equation}
and $I\mathcal{H}$ acts on $IZ$ from the right by
\begin{equation}\label{eq IH acts on IZ}
 (g,z,h_2)\cdot (h_1,h_2,h):=(g,zh, h_1)=(g,zh, h^{-1}h_2h).
\end{equation}
\end{defi}

\begin{rmk}
Since the $\mathcal{H}$-action on $Z$ is free, the $h$ component of $IZ$ as in \eqref{eq: IZ} is uniquely determined by $(g,z)$ and the condition $gzh=z$. Hence $IZ$ is isomorphic to 
\begin{equation}\label{eq: IZ simplified}
\{(g,z)\in G_1\times Z|s(g)=t(g)=\rho(z)\}.
\end{equation}
\end{rmk}

\begin{rmk}\label{rmk: inertia morphism local case}
In the local case, let $f\colon G\ltimes X\to H\ltimes Y$ be a morphism which consists of $\underline{f}\colon X\to Y$ and $\bar{f}\colon G\to H$. Then the induced morphism 
$$
If\colon \coprod_{(g)\in \text{Conj}(G)}Z_G(g)\ltimes X^g\to  \coprod_{(h)\in \text{Conj}(H)}Z_H(h)\ltimes Y^h
$$ 
consists of $\overline{If}=\bar{f}\colon Z_G(g)\to Z_H(\bar{f}(g))$ and $\underline{If}$ is given on each component by $\underline{f}|_{X^g}\colon X^g\to Y^{\bar{f}(g)}$.
\end{rmk}

Let $(Z,\rho, \sigma)\colon \mathcal{G}\to \mathcal{H}$ be a Morita equivalence. It is easy to see that 
$$
(IZ,I\rho, I\sigma)\colon I\mathcal{G}\to I\mathcal{H}
$$
is also a Morita equivalence. In other words, the inertia groupoid is invariant under Morita equivalence as in Definition \ref{defi: groupoid Morita equivalence}, and hence we can consider the inertia orbifold of a complex orbifold. It is not hard to check the following property (See Proposition \ref{prop: inertia morphism of stablizer-preserving embedding is  a stablizer-preserving embedding1}). 

\begin{prop}\label{prop: inertia morphism of stablizer-preserving embedding is  a stablizer-preserving embedding}
Let $(Z,\rho, \sigma)\colon \mathcal{G}\to \mathcal{H}$ be a generalized morphism between complex orbifold groupoids. If $(Z,\rho, \sigma)$ is a stabilizer-preserving embedding, then $(IZ,I\rho, I\sigma)$ is also a stabilizer-preserving embedding.
\end{prop}

\section{Bott-Chern cohomology of complex orbifolds}\label{sec:H_BC}
In this section, we introduce the Bott-Chern cohomology for complex orbifolds. For reference see \cite{angella2013cohomologies}, \cite[Section 2.1]{Angella2014}, \cite{Angella2015}, and \cite[Section 6.8]{demailly1997complex}.

\subsection{Definition}
Let $X$ be an $n$-dimensional complex orbifold with a groupoid representation $\mathcal{G}=(G_0, G_1)$. Let $$\Omega^{p,q}_{\mathcal{G}}(G_0,\mathbb{C})$$ be the space of $\mathcal{G}$-equivariant $(p,q)$-forms on $G_0$. $\Omega^{\bullet,\bullet}_{\mathcal{G}}(G_0,\mathbb{C})$ is a bigraded algebra with the ordinary wedge product.
It is clear that the operators $\dpar$ and $\dbar$ map $\Omega^{p,q}_{\mathcal{G}}(G_0,\mathbb{C})$ to $\Omega^{p+1,q}_{\mathcal{G}}(G_0,\mathbb{C})$ and $\Omega^{p,q+1}_{\mathcal{G}}(G_0,\mathbb{C})$, respectively.

\begin{rmk}
In the local case, $\mathcal{G}=G\ltimes X$ and $\Omega^{p,q}_{\mathcal{G}}(G_0,\mathbb{C})$ is nothing but $\Omega^{p,q}_G(X,\mathbb{C})$, the space of $G$-invariant $(p,q)$-forms on $X$.
\end{rmk}

\begin{defi}\label{defi: Bott-Chern cohomology}
Write $\diff=\dpar+\dbar$ for the de Rham differential. The {\em Bott-Chern cohomology group} $\HBC^{p,q}(\mathcal{G},\mathbb{C})$ is defined as
\begin{equation}\label{eq: Bott-Chern cohomology}
\HBC^{p,q}(\mathcal{G},\mathbb{C}):=\big(\Omega^{p,q}_{\mathcal{G}}(G_0,\mathbb{C})\cap\ker \diff \big) /\dbar\dpar \Omega^{p-1,q-1}_{\mathcal{G}}(G_0,\mathbb{C}).
\end{equation}
\end{defi}
$\HBC^{\bullet,\bullet}(\mathcal{G},\mathbb{C})$ inherits the structure of a bigraded algebra from $\Omega^{\bullet,\bullet}_{\mathcal{G}}(G_0,\mathbb{C})$. As in \cite[Section 2.1]{bismut2023coherent}, we put
\begin{equation}\label{eq: Bott-Chern cohomology =}
\Omega^{(=)}_{\mathcal{G}}(G_0,\mathbb{C})=\bigoplus_{p=0}^n\Omega^{p,p}_{\mathcal{G}}(G_0,\mathbb{C}), ~\HBC^{(=)}(\mathcal{G},\mathbb{C})=\bigoplus_{p=0}^n\HBC^{p,p}(\mathcal{G},\mathbb{C}).
\end{equation}
We can similarly define the real vector space $\Omega^{(=)}_{\mathcal{G}}(G_0,\mathbb{R})$ and $\HBC^{(=)}(\mathcal{G},\mathbb{R})$, which are algebras as well.

\begin{rmk}
For a general complex groupoid, we need to use the Bott-Shulman double complexes to define its Bott-Chern cohomology. In the case of a proper \'{e}tale groupoid, Definition \ref{defi: Bott-Chern cohomology} is sufficient.
\end{rmk}

\subsection{Current and Bott-Chern cohomology}
A review of the theory of currents in complex geometry can be found in \cite[Section 3.1 and Section 3.2]{griffiths1994principles}. 

Let $\mathcal{D}^{p,q}_{\mathcal{G}}(G_0,\mathbb{C})$ denote the space of $\mathcal{G}$-equivariant $(p,q)$-currents on $G_0$. We can also define the operators $\dpar$ and $\dbar$ which map $\mathcal{D}^{p,q}_{\mathcal{G}}(G_0,\mathbb{C})$ to $\mathcal{D}^{p+1,q}_{\mathcal{G}}(G_0,\mathbb{C})$ and $\mathcal{D}^{p,q+1}_{\mathcal{G}}(G_0,\mathbb{C})$, respectively. Again, let $\diff=\dpar+\dbar$.
 
By an argument similar to that in \cite{angella2013cohomologies}, we can prove that
\begin{equation}
\HBC^{p,q}(\mathcal{G},\mathbb{C})\cong \big(\mathcal{D}^{p,q}_{\mathcal{G}}(G_0,\mathbb{C})\cap\ker \diff \big) /\dbar\dpar \mathcal{D}^{p-1,q-1}_{\mathcal{G}}(G_0,\mathbb{C}).
\end{equation}

\subsection{Functorial properties of Bott-Chern cohomology}\label{subsec: functorial Bott-Chern cohomology}
Let $\mathcal{G}$ and $\mathcal{H}$ be complex orbifold groupoids of dimensions $n$ and $n^{\prime}$ respectively. Let $f=(Z,\rho, \sigma)\colon\mathcal{G}\to \mathcal{H}$ be a generalized morphism. Then we have morphisms of bigraded algebras $$\sigma^*\colon\Omega^{\bullet,\bullet}_{\mathcal{H}}(H_0,\mathbb{C})\to  \Omega^{\bullet,\bullet}_{\mathcal{G}\times \mathcal{H}}(Z,\mathbb{C}),\quad  \rho^*\colon \Omega^{\bullet,\bullet}_{\mathcal{G}}(G_0,\mathbb{C}) \to  \Omega^{\bullet,\bullet}_{\mathcal{G}\times \mathcal{H}}(Z,\mathbb{C}).$$ Moreover $\rho^*\colon \Omega^{\bullet,\bullet}_{\mathcal{G}}(G_0,\mathbb{C})\to \Omega^{\bullet,\bullet}_{\mathcal{G}\times \mathcal{H}}(Z,\mathbb{C})$ is an isomorphism. Thus we can define $f^*\colon= (\rho^*)^{-1}\circ \sigma^*$, and it induces a morphism of bigraded algebras $\HBC^{\bullet,\bullet}(\mathcal{H},\mathbb{C})$ to $\HBC^{\bullet,\bullet}(\mathcal{G},\mathbb{C})$. Also, $f^*$ preserves the corresponding real vector spaces.

By duality, we can define the pushforward map $f_*=\sigma_*\circ (\rho_*)^{-1}$ which maps $\mathcal{D}^{p,q}_{\mathcal{G}}(G_0,\mathbb{C})$ to $\mathcal{D}^{p-n+n^{\prime},q-n+n^{\prime}}_{\mathcal{H}}(H_0,\mathbb{C})$. 
Moreover, by Definition \ref{defi: generalized morphism for inertial groupoid} we have a generalized morphism $If\colon I\mathcal{G}\to I\mathcal{H}$ between inertia groupoids, hence a morphism
\begin{equation}\label{eq: pushforward current inertia}
If_*\colon \HBC^{\bullet,\bullet}(I\mathcal{G},\mathbb{C})\to \HBC^{\bullet-n+n^{\prime},\bullet-n+n^{\prime}}(I\mathcal{H},\mathbb{C}).
\end{equation}

\begin{prop}\label{prop: Bott-Chern is independent of choice}
The Bott-Chern cohomology is independent of the choice of the groupoid representation. In particular, if two complex orbifold groupoids $\mathcal{G}$ and $\mathcal{H}$ are Morita equivalent, then we have 
\begin{equation}
\HBC^{p,q}(\mathcal{G},\mathbb{C})\cong 
\HBC^{p,q}(\mathcal{H},\mathbb{C}).
\end{equation}
\end{prop}
\begin{proof}
By Definition \ref{defi: groupoid Morita equivalence}, we only need to consider the Bott-Chern cohomology under an equivalence $\psi\colon \mathcal{H}\to \mathcal{G}$ as in Definition \ref{defi: groupoid equivalence}. By the same argument as in the proof of \cite[Lemma 2.2]{adem2007orbifolds}, we obtain that 
\begin{equation}
\Omega^{p,q}_{\mathcal{G}}(G_0,\mathbb{C})\cong \Omega^{p,q}_{\mathcal{H}}(H_0,\mathbb{C}).
\end{equation}
Hence $\HBC^{p,q}(\mathcal{G},\mathbb{C})\cong 
\HBC^{p,q}(\mathcal{H},\mathbb{C})$.
\end{proof}

\subsection{Pushforward map \texorpdfstring{$If_*$}{} for iso-spatial embeddings}\label{subsec: pushforward cohomology iso-spatial}
In this subsection we consider the case that $f\colon \mathcal{G}\to \mathcal{H}$ is an iso-spatial embedding. We give an explicit description of the pushforward map 
$$
If_*\colon \HBC^{\bullet,\bullet}(I\mathcal{G},\mathbb{C})\to \HBC^{\bullet-n+n^{\prime},\bullet-n+n^{\prime}}(I\mathcal{H},\mathbb{C}),
$$
at the level of differential forms when we restrict to the local case. 

As in Remark \ref{rmk: two kind embeddings local case}, the iso-spatial embedding is locally given by the inclusion of finite groups $G\hookrightarrow H$ and the identity map on the complex manifold $X$. Hence, by Remark \ref{rmk: inertia morphism local case} the induced morphism between inertia groupoids, 
\begin{equation}\label{eqn:If_local}
If\colon  \coprod_{(g)\in C(G)}Z_G(g)\ltimes X^g\to  \coprod_{(h)\in C(H)}Z_H(h)\ltimes X^h,
\end{equation}
consists of the inclusion of finite groups $G\hookrightarrow H$ together with the identity map $X^g\to X^g$. For each $\omega\in \Omega^{p,q}_G(X^g,\mathbb{C})$, we define  
\begin{equation}\label{eq: pushforward forms iso-spatial inertia}
If_*(\omega)\colon=\frac{1}{|Z_G(g)|}\sum_{h\in Z_H(g)}h^*\omega,
\end{equation}
where $Z_G(g)$ and $Z_H(g)$ are the centralizers of $g$ in $G$ and $H$, respectively.

\begin{prop}\label{prop: pushforward map of forms for iso-spatial embedding local case}
Let $f\colon G\ltimes X\to H\ltimes X$ be an iso-spatial embedding and $If$ be the induced morphism between inertia groupoids as in (\ref{eqn:If_local}). Then the pushforward map $If_*$ in \eqref{eq: pushforward forms iso-spatial inertia} gives the pushforward map $If_*$ in \eqref{eq: pushforward current inertia}.
\end{prop}
\begin{proof}
Let $\omega\in \Omega^{\bullet}_{Z_G(g)}(X^g,\mathbb{C})$, we can consider $\omega$ as a $Z_G(g)$-invariant current on $X^g$ by
\begin{equation}
\omega(\theta):=\frac{1}{|Z_G(g)|}\int_{X^g}\omega\wedge \theta,\quad \text{where } \theta\in \Omega^{\bullet}_{Z_G(g)}(X^g,\mathbb{C}) \text{ is a testing form}.
\end{equation}

It is clear that for any form $\theta\in \Omega^{\bullet}_{Z_H(g)}(X^g,\mathbb{C})$, we have $If^*(\theta)=\theta$ considered as a form in $\Omega^{\bullet}_{Z_G(g)}(X^g,\mathbb{C})$. Therefore we have 
\begin{equation}\label{eq: form pullback pairing}
\omega(If^*(\theta))=\frac{1}{|Z_G(g)|}\int_{X^g}\omega\wedge\theta.
\end{equation}
On the other hand, by \eqref{eq: pushforward forms iso-spatial inertia} we have
\begin{equation}
If_*(\omega)(\theta)=\frac{1}{|Z_H(g)|}\int_{X^g}If_*(\omega)\wedge \theta=\frac{1}{|Z_H(g)||Z_G(g)|}\sum_{h\in Z_H(g)}\int_{X^g}h^*\omega\wedge\theta.
\end{equation}
Since $\theta$ is $Z_H(g)$-invariant, we have 
\begin{equation}
h^*\omega\wedge\theta=h^*(\omega\wedge \theta), \text{ for any } h\in Z_H(g),
\end{equation}
hence 
\begin{equation}
If_*(\omega)(\theta)=\frac{1}{|Z_H(g)||Z_G(g)|}\sum_{h\in Z_H(g)}\int_{X^g}h^*(\omega\wedge\theta)=\frac{|Z_H(g)|}{|Z_H(g)||Z_G(g)|}\int_{X^g}\omega\wedge\theta,
\end{equation}
which is exactly \eqref{eq: form pullback pairing}. Hence \eqref{eq: pushforward forms iso-spatial inertia} indeed gives the pushforward map $If_*$ in \eqref{eq: pushforward current inertia}.
\end{proof}

\section{Coherent sheaves on complex orbifolds}\label{section: coherent sheaves on complex orbifolds}
In this section, we discuss the notion of coherent sheaves on complex orbifolds from the viewpoint of groupoids. We follow the construction in \cite[Section 2.1]{moerdijk2001etale}.

\begin{defi}\label{defi: G-sheaf}
Let $\mathcal{G}=(G_0, G_1)$ be a complex orbifold groupoid. A right $\mathcal{G}$-sheaf is a sheaf $\mathcal{F}$ on $G_0$ together with a right action of $\mathcal{G}$. In more detail, each $g\colon x\to y$ in $G_1$ gives a map 
$$
\mathcal{F}_x\to \mathcal{F}_y, ~ a\mapsto a\cdot g,
$$
where $\mathcal{F}_x$ is the stalk of $\mathcal{F}$ at $x$. It satisfies the conditions that $(a\cdot g_1)\cdot g_2=a\cdot g_1g_2$ and $a\cdot 1_x=a$. If $\mathcal{F}$ admits additional structures, say vector space, algebra, etc., then we require the $\mathcal{G}$-action to preserve such structures. We define left $\mathcal{G}$-sheaves in the same way.
\end{defi}

The following definitions can be found in \cite[Section 3]{moerdijk1997orbifolds}.

\begin{defi}\label{defi: structure sheaf O_G}
Let $\mathcal{G}=(G_0, G_1)$ be a complex orbifold groupoid. We define the structure sheaf $\mathcal{O}_{\mathcal{G}}$ on $\mathcal{G}$ as the sheaf of germs of holomorphic functions on $G_0$. This sheaf carries natural left and right $\mathcal{G}$-actions.
\end{defi}

\begin{defi}\label{defi: sheaf of O-G modules}
Let $\mathcal{G}=(G_0, G_1)$ be a complex orbifold groupoid. A sheaf of $\mathcal{O}_{\mathcal{G}}$ modules $\mathcal{F}$ is a sheaf on $G_0$ equipped with a left $\mathcal{G}$ action and an $\mathcal{O}_{\mathcal{G}}$ action which are compatible. Compatibility means that 
\begin{equation}
g\cdot (fs)=(g\cdot f) (g\cdot s),
\end{equation}
where $g\in G_1$, $f\in \mathcal{O}_{\mathcal{G}}$, and $s\in \mathcal{F}$. 

We denote the category of sheaves of $\mathcal{O}_{\mathcal{G}}$ modules by $\mathcal{O}_{\mathcal{G}}\text{-Mod}$.
\end{defi}

The following result is essential in the construction of derived functors.

\begin{prop}\label{prop: enough injectives}
The category $\mathcal{O}_{\mathcal{G}}\text{-Mod}$ is an abelian category with enough injectives.
\end{prop}
\begin{proof}
The claim is local, so we reduce to the case that $\mathcal{G}$ is the action groupoid of a finite group acting on a complex manifold. The claim then follows from \cite[Propositions 5.1.1 and 5.1.2]{grothendieck1957quelques}.
\end{proof}

Let $f\colon \mathcal{G}\to \mathcal{H}$ be a holomorphic morphism between complex orbifold groupoids. We can define the pull back functor $$f^*\colon \mathcal{O}_{\mathcal{H}}\text{-Mod}\to \mathcal{O}_{\mathcal{G}}\text{-Mod}$$ as in \cite[Section 3.1]{moerdijk2001etale}. In more detail, for a sheaf $\mathcal{F}\in \mathcal{O}_{\mathcal{H}}\text{-Mod}$, the map $f_0\colon G_0\to H_0$ gives a sheaf $f^{-1}\mathcal{F}$ on $G_0$, which has a natural left $\mathcal{G}$-action: for each $g\colon x\to y$ in $G_1$, the action of $g$ on $a\in (f^{-1}\mathcal{F})_x$ is simply given by
$$
g\cdot a:=f(g)\cdot a.
$$
Then we define $f^*\mathcal{F}$ as
\begin{equation}
f^*\mathcal{F}:=\mathcal{O}_{\mathcal{G}}\otimes_{f^{-1}\mathcal{O}_{\mathcal{H}}}f^{-1}\mathcal{F},\quad \text{with }\mathcal{G} \text{-action }\quad g(s\otimes a):=(g\cdot s)\otimes (f(g)\cdot a).
\end{equation}
It is clear that $f^*\mathcal{F}$ is a sheaf of $\mathcal{O}_{\mathcal{G}}$ modules. We define pullbacks of generalized morphisms in the same way as in Corollary \ref{coro: functoriality of superconnections}.

\begin{defi}\label{defi: coherent sheaves on orbifolds}
Let $\mathcal{G}$ be a complex orbifold groupoid. A sheaf $\mathcal{F}$ of $\mathcal{O}_{\mathcal{G}}$-modules is called a \emph{coherent} sheaf if it satisfies the following conditions.
\begin{enumerate}
\item $\mathcal{F}$ is finite type, i.e. for every $x\in G_0$ there exists an open neighborhood $(U,G)$ of $x$ (Definition \ref{defi: neighborhood}) and a $\mathcal{G}$-sheaf $\mathcal{M}$, a finite rank free sheaf on $G_0$, such that there exists a $\mathcal{G}|_{U}$-equivariant surjective map $\mathcal{M}|_{U}\twoheadrightarrow \mathcal{F}|_U$;
\item For every $(U,G)$  and any $\mathcal{G}|_U$-equivariant  map $\phi\colon \mathcal{M}|_U\to \mathcal{F}|_U$, the kernel of $\phi$ is also finite type.
\end{enumerate}
We denote the category of coherent sheaves on $\mathcal{G}$ by $\coh(\mathcal{G})$.   
\end{defi}

\begin{rmk}
Condition (1) in Definition \ref{defi: coherent sheaves on orbifolds} has the flavor of global resolutions of sheaves. Here we discuss the role of such a condition in the context of Riemann-Roch theorems in algebraic geometry. 

The Grothendieck-Riemann-Roch theorem for Deligne-Mumford stacks proved by To\"en \cite{Toen} is in the setting closest to this paper's. Toen's theorem requires a few assumptions on the Deligne-Mumford stacks, including the property that every coherent sheaf is the quotient of a locally free sheaf, see \cite[HYPOTH\'ESE 4.9 (2)]{Toen}. This property is called the {\em resolution property} and is a property of the stack. In \cite{totaro}, Totaro proved that (under minor assumptions) an algebraic stack has the resolution property if and only if it is a quotient stack of a quasi-affine scheme by a general linear group. 

In view of the above, the resolution property plays a prominent role for Riemann-Roch theorem for Deligne-Mumford stacks and it is equivalent to some global geometric property of the stacks. For this reason, Condition (1) in Definition \ref{defi: coherent sheaves on orbifolds}, which we view as analogous to the resolution property in algebraic geometry, is a reasonable requirement to impose when considering Riemann-Roch for orbifolds.

\end{rmk}

The following result provides the main tool for our study of coherent sheaves on complex orbifolds. 

\begin{prop}\label{prop: global resolution}
Let $\mathcal{G}$ be a compact complex orbifold groupoid. For a bounded complex of sheaves  $\mathcal{F}^{\bullet}$ of $\mathcal{O}_{\mathcal{G}}\text{-Mod}$ with coherent cohomologies, there exists a bounded complex of finite-dimensional $C^{\infty}$-, left $\mathcal{G}$-vector bundles $Q^{\bullet}$ on $G_0$ and a quasi-isomorphism
$$
Q^{\bullet}\overset{\sim}{\longrightarrow} \mathcal{F}^{\bullet,\infty}.
$$
\end{prop}
\begin{proof} 
The proof of this result is similar to \cite[Proposition 6.3.2]{bismut2023coherent}. Without loss of generality, we may assume $\mathcal{F}^i=0$ for $i<0$. 

Let
$
k=\sup\{i\,|\,\mathcal{H}^i(\mathcal{F}^{\bullet})=0\}.
$
We know $k$ is finite and $k\in \mathbb{N}$. We use induction on $k$. We will also prove that we can choose $Q^{\bullet}$ such that $Q^i=0$ for $i>k$. The $k=0$ case follows from Lemma \ref{lemma: global resolution special case}. The induction is exactly the same as the proof of \cite[Proposition 6.3.2]{bismut2023coherent}.
\end{proof}

It is clear that for a morphism $f\colon \mathcal{G}\to \mathcal{H}$ between complex orbifold groupoids, the induced functor $f^*\colon \mathcal{O}_{\mathcal{H}}\text{-Mod}\to \mathcal{O}_{\mathcal{G}}\text{-Mod}$ restricts to a functor 
$$f^*\colon \coh(\mathcal{H})\to \coh(\mathcal{G}).$$

\begin{prop}\label{prop: coherent and Morita equivalence}
If $f\colon \mathcal{G}\to \mathcal{H}$ is an equivalence between complex orbifold groupoids in the sense of Definition \ref{defi: groupoid equivalence}, then the induced functor $f^*\colon \coh(\mathcal{H})\to \coh(\mathcal{G})$ is an equivalence of categories.
\end{prop}
\begin{proof}
This follows from the works of Moerdijk and Pronk \cite[Definition 3.1, Remark (3)]{moerdijk1997orbifolds} by replacing the sheaf of smooth functions with the sheaf of holomorphic functions. 
\end{proof}

We take the following definition from \cite[Section 3.2]{moerdijk2005lie}.
Let $f= (Z, \rho, \sigma)\colon \mathcal{G}\to \mathcal{H}$ be a generalized morphism between complex orbifold groupoids. The pullback functor $f^*\colon \mathcal{O}_{\mathcal{H}}\text{-Mod}\to \mathcal{O}_{\mathcal{G}}\text{-Mod}$ has a right adjoint 
$$f_*\colon \mathcal{O}_{\mathcal{G}}\text{-Mod}\to \mathcal{O}_{\mathcal{H}}\text{-Mod}.$$ 
More explicitly, given an $\mathcal{O}_{\mathcal{G}}$-module $E$, we pull it back to $Z$ to get an $\mathcal{O}_{\mathcal{Z}}$-module $\rho^*(E)$. Then we push $\rho^*(E)$ to $\mathcal{H}$ by taking the $\mathcal{G}$-invariant sections of the usual pushforward of sheaves between complex manifolds under $\sigma$. The pushforward inherits a natural $\mathcal{H}$-action as $\rho^*(E)$ is an $\mathcal{G}\times \mathcal{H}$ module. The $\mathcal{G}$-action on $Z$ commutes with the $\mathcal{H}$-action and $\mathcal{O}_{\mathcal{Z}}$-module structure. The pushforward $\sigma_*\rho^*(E)$ is an $\mathcal{O}_{\mathcal{H}}$-module. We define
\begin{equation}\label{eq: definitionofpushforward}
f_*(E):=\sigma_*\rho^*(E).
\end{equation}

\begin{prop}\label{prop: pushforward adjoint}
The push forward functor $f_*\colon \mathcal{O}_{\mathcal{G}}\text{-Mod}\to \mathcal{O}_{\mathcal{H}}\text{-Mod}$ defined in (\ref{eq: definitionofpushforward}) is right adjoint to the pullback functor $f^*\colon \mathcal{O}_{\mathcal{H}}\text{-Mod}\to \mathcal{O}_{\mathcal{G}}\text{-Mod}$.
\end{prop}
\begin{proof}
It is clear that $\sigma_*$ is right adjoint to $\sigma^*$. Since $\rho^*$ is an equivalence, we get the desired result.
\end{proof}

\begin{coro}\label{coro: pushforward functorial}
The pushforward functor defined in (\ref{eq: definitionofpushforward}) is functorial. More precisely, for $f\colon \mathcal{G}\to \mathcal{H}$ and $p\colon \mathcal{H}\to \mathcal{K}$, we have
\begin{equation}
(p f)_*=p_*\circ f_*\colon \mathcal{O}_{\mathcal{G}}\text{-Mod}\to \mathcal{O}_{\mathcal{K}}\text{-Mod}.
\end{equation}
\end{coro}
\begin{proof}
It is immediate from Proposition \ref{prop: pushforward adjoint}.
\end{proof}

\begin{prop}\label{prop:pushforwardofcoherentsheaves}
For a proper generalized morphism $f\colon \mathcal{G}\to \mathcal{H}$, the induced functor $f_*\colon \mathcal{O}_{\mathcal{G}}\text{-Mod}\to \mathcal{O}_{\mathcal{H}}\text{-Mod}$ restricts to a functor $f_*\colon \coh(\mathcal{G})\to \coh(\mathcal{H})$.
\end{prop}
\begin{proof}
It suffices to prove that $f_*(\mathcal{O}_{\mathcal{G}})$ is coherent. By \eqref{eq: definitionofpushforward}, it reduces to proving that $\sigma_*(\mathcal{O}_{\mathcal{Z}})$ is coherent, which is clear by definition.
\end{proof}

\begin{defi}\label{defi: bounded derived category}
Let 
$$D^b_{\coh}(\mathcal{G})$$
be the derived category of bounded complexes of objects in $\mathcal{O}_{\mathcal{G}}\text{-Mod}$  with coherent cohomologies.
\end{defi}

\begin{defi}\label{defi: derived tensor product}
For $\mathcal{F}_1$ and $\mathcal{F}_2\in D^b_{\coh}(\mathcal{G})$, we define their derived tensor product $$\mathcal{F}_1\otimes^L  \mathcal{F}_2$$ via flat resolutions as in \cite[\href{https://stacks.math.columbia.edu/tag/06YH}{Tag 06YH}]{stacks-project}.
\end{defi}

\begin{defi}\label{defi: derived pull back}
For a generalized morphism $f\colon \mathcal{G}\to \mathcal{H}$, we define the \emph{derived pull back} $$Lf^*\colon D^b_{\coh}(\mathcal{H})\to D^b_{\coh}(\mathcal{G})$$ as in \cite[\href{https://stacks.math.columbia.edu/tag/07BD}{Tag 07BD}]{stacks-project}.
\end{defi}

\begin{defi}\label{defi: derived push forward}
For a proper generalized morphism $f\colon \mathcal{G}\to \mathcal{H}$, we define the \emph{derived pushforward} 
$$Rf_*\colon D^b_{\coh}(\mathcal{G})\to D^b_{\coh}(\mathcal{H})$$ by injective resolutions as in \cite[Section III.6]{gelfand2003methods}. By Proposition \ref{prop: enough injectives}, this definition makes sense.
\end{defi}

By \cite[Theorem III.7.1]{gelfand2003methods}, we have a natural isomorphism 
\begin{equation}
R(f\circ g)_*\overset{\sim}{\longrightarrow}Rf_*\circ Rg_*.
\end{equation}

According to \cite[Proposition  III.8.12]{gelfand2003methods}, we can compute $Rf_*$ via soft resolutions instead of injective resolutions. Notice that \cite[Proposition III.8.12]{gelfand2003methods} is about $Rf_!$, the right derived functor of the proper pushforward functor $f_!$. Nevertheless, in our case $f_*$ and $f_!$ coincide, as both $\mathcal{G}$ and $ \mathcal{H}$ are required to be compact.

Let $K(D^b_{\coh}(\mathcal{G}))$ denote the Grothendieck group of $D^b_{\coh}(\mathcal{G})$. Using the same argument as in \cite[Tag 0FCP]{stacks-project}, we get
$$
K(D^b_{\coh}(\mathcal{G}))\cong K(\mathcal{G}).
$$
For a holomorphic map $f\colon \mathcal{G}\to \mathcal{H}$, the derived pushforward map $Rf_*$ induces a homomorphism 
\begin{equation}
f_!\colon K(\mathcal{G})\to K(\mathcal{H}).
\end{equation}

\section{Antiholomorphic flat superconnection on complex orbifolds}\label{Section: ahfs}
Here we introduce the theory of antiholomorphic flat superconnections on complex orbifolds. 

A superconnection is a way of packaging several kinds of geometric and analytic information into a single object. Introduced by Quillen on a $\mathbb Z_2$-graded vector bundle $E=E^+\oplus E^-$, it extends an ordinary connection by allowing differential-form components of different degrees. Its curvature $\mathbb {A}^2$ therefore incorporates not only the usual curvature of the bundle, but also the grading and any algebraic or operator-theoretic data built into the superconnection. The associated form
\[
\operatorname{Ch}(\mathbb A) =\operatorname{Str}\!\left(e^{-\mathbb A^2}\right)
\]
is closed and represents the Chern character of the corresponding virtual bundle \cite{Quillen1985}. In this sense, a superconnection does not measure a single numerical quantity. Rather, it converts geometric and spectral information into a differential form representing a topological characteristic class. For a family of Dirac operators, Bismut's superconnection makes this principle particularly effective: after introducing a scaling parameter, its small-time limit recovers the local curvature density in the family index theorem, while its large-time behavior reflects the Chern character of the analytic index bundle. The transgression formula shows that the cohomology class is independent of the scaling parameter, thereby connecting the local geometric expression with the global analytic index \cite{Bismut1986}, \cite{BerlineVergne1987}.

In complex geometry, Block \cite{block2006duality} introduced the corresponding notion of an antiholomorphically flat superconnection. On a $\mathbb {Z}$-graded smooth vector bundle $E^\bullet$, this is an operator $A''$ acting on $\Omega^{0,\bullet}(X;E^\bullet)$, satisfying the graded $\bar\partial$-Leibniz rule and the flatness condition
\[
(A'')^2=0.
\]

Thus, $A''$ can be regarded as a derived version of the flat $\bar\partial$-connection that defines an ordinary holomorphic vector bundle. The higher-degree components of $A''$ encode not merely one holomorphic bundle, but the differential and homotopy coherent data needed to describe a complex of sheaves. Block proved that the resulting differential graded category of cohesive modules recovers the bounded derived category of coherent sheaves on a compact complex manifold \cite{block2010duality}. Bismut, Shen, and Wei \cite{bismut2023coherent} build on this description by equipping an antiholomorphic superconnection with a generalized Hermitian metric and constructing the associated curvature and Chern character forms. In this setting, the flat superconnection $A''$ records the derived holomorphic structure of the coherent sheaf, while the curvature obtained after adjoining its metric-dependent adjoint produces a differential form representative of the sheaf's Chern character in Bott-Chern cohomology. Their Riemann-Roch-Grothendieck theorem then describes the behavior of this Bott-Chern Chern character under holomorphic direct image, with the Todd class accounting for the relative tangent geometry. From this viewpoint, their result is a complex-geometric counterpart of the families index theorem: it identifies the characteristic class obtained from the analytic superconnection construction with the class predicted by the geometry of holomorphic pushforward.

\subsection{Basic constructions}
Our presentation of this section follows that of \cite{block2010mukai}.

\begin{lemma}
On a complex orbifold groupoid $\mathcal{G}$, the bundle of $(0,1)$-cotangent vectors $\overline{T^{*}G_0}$ is a $\mathcal{G}$-$\mathcal{G}$ bibundle.
\end{lemma}
\begin{proof}
This is true because the $\mathcal{G}$-action on $G_0$ preserves the complex structure. 
\end{proof}

\begin{defi}\label{defi: antiholo superconn}
Let $X$ be a complex orbifold with a groupoid representation $\mathcal{G}=(G_0, G_1)$. An \emph{antiholomorphic flat superconnection} on $\mathcal{G}$ is a bounded, finite rank, $\mathbb{Z}$-graded, left $\mathcal{G}$-equivariant, $C^{\infty}$-vector bundle $E^{\bullet}$ on $G_0$ together with a $\mathcal{G}$-equivariant superconnection\footnote{We will use $\wedge^{\bullet}\overline{T^{*}G_0}  \otimes E^\bullet$ to denote the tensor product of two $\mathcal{G}$-equivariant bundles $\wedge^{\bullet}\overline{T^{*}G_0}$ and $E^\bullet$} of total degree $1$,
$$
A^{E^{\bullet}\prime\prime}\colon \wedge^{\bullet}\overline{T^{*}G_0}  \otimes E^{\bullet}\to \wedge^{\bullet}\overline{T^{*}G_0}  \otimes E^{\bullet},
$$
such that $A^{E^{\bullet}\prime\prime}\circ A^{E^{\bullet}\prime\prime}=0$.

In more detail, $A^{E^{\bullet}\prime\prime}$ decomposes into
\begin{equation}\label{eq: decomposition of anti super conn}
A^{E^{\bullet}\prime\prime}=v_0+\nabla^{E^{\bullet}\prime\prime}+v_2+\ldots,
\end{equation}
where 
$$
\nabla^{E^{\bullet}\prime\prime}\colon E^{\bullet}\to \overline{T^{*}G_0}  \otimes E^{\bullet}
$$
 is a $\mathcal{G}$-equivariant $\dbar$-connection and
\begin{equation}\label{eq: cohesive module}
v_i\colon E^{\bullet}\to \wedge^{i}\overline{T^{*}G_0}  \otimes E^{\bullet+1-i}, \quad i=0,2,...
\end{equation}
is $\mathcal{G}$-equivariant and $C^{\infty}(G_0)$-linear.
\end{defi}

Antiholomorphic flat superconnections on $\mathcal{G}$ form a dg-category denoted by $$B(\mathcal{G}).$$ In more detail, for two flat antiholomorphic superconnections on $X$, $(E^{\bullet}, A^{E^{\bullet}\prime\prime})$ and $(F^{\bullet}, A^{F^{\bullet}\prime\prime})$ with
$$
A^{E^{\bullet}\prime\prime}=v_0+\nabla^{E^{\bullet}\prime\prime}+v_2+\ldots\quad \text{and}\quad
A^{F^{\bullet}\prime\prime}=u_0+\nabla^{F^{\bullet}\prime\prime}+u_2+\ldots,
$$ 
a morphism $\phi\colon (E^{\bullet}, A^{E^{\bullet}\prime\prime})\to (F^{\bullet}, A^{F^{\bullet}\prime\prime})$ of degree $k$ is given by
$$
\phi=\phi_0+\phi_1+\ldots\quad \text{where}\quad \phi_i\colon  E^{\bullet}\to \wedge^{i}\overline{T^{*}G_0}  \otimes {\color{black}{F}}^{\bullet+k-i}, \quad i\geq 0
$$
are $\mathcal{G}$-equivariant and $C^{\infty}(G_0)$-linear.

The differential of $\phi$ is given by 
$$
d\phi=A^{F^{\bullet}\prime\prime}\phi-(-1)^k\phi A^{E^{\bullet}\prime\prime}.
$$
More explicitly, the $l$-th component $(d\phi)_l\colon E^{\bullet}\to \wedge^{l}\overline{T^{*}G_0}  \otimes {\color{black}{F}}^{\bullet+k+1-l}$ is given by
\begin{equation}
(d\phi)_l=\sum_{i\neq 1}\big(u_i\phi_{l-i}-(-1)^k\phi_{l-i}v_i\big)+\nabla^{F^{\bullet}\prime\prime}\phi_{l-1}-(-1)^k\phi_{l-1}\nabla^{E^{\bullet}\prime\prime}.
\end{equation}

\subsection{Pullback and equivalence}
We want to study the functoriality property of $B(\mathcal{G})$. First, we define the pullbacks under holomorphic morphisms.

\begin{prop}\label{prop: pull back functor}
Let $\psi\colon \mathcal{G}\to \mathcal{H}$ be a holomorphic morphism between complex orbifold groupoids. Then $\psi$ induces a pullback dg-functor
\begin{equation}
\psi^*_b\colon B(\mathcal{H})\to B(\mathcal{G}).
\end{equation}
\end{prop}
\begin{proof}
We have pullback bundles, pullback connections, and pullback bundle maps.
\end{proof}

\begin{rmk}
Here we use the notation $\psi^*_b$ to emphasize that we pull back objects of the dg-category $B(\mathcal{H})$.
\end{rmk}

The following result tells us that the dg-category is independent of the representative groupoid.
\begin{prop}\label{prop: Morita equivalent and quasi-equivalent}
If $\mathcal{G}$ and $\mathcal{H}$ are Morita equivalent, then the dg-categories of flat antiholomorphic superconnections on $\mathcal{G}$ and $\mathcal{H}$ are quasi-equivalent.
\end{prop}
\begin{proof}[Proof of Proposition \ref{prop: Morita equivalent and quasi-equivalent}]
The proof is essentially the same as that of \cite[Theorem 7.2, Corollary 7.3]{block2010mukai}. We give a proof here to show the geometric picture.

By Proposition \ref{prop: equiv defi of Morita equiv}, there exists a $\mathcal{G}$-$\mathcal{H}$ principal bibundle $P$. We will construct a dg-functor $P_*\colon B(\mathcal{H})\to B(\mathcal{G})$ using $P$. First, for a left $\mathcal{H}$-vector bundle $E$ on $\mathcal{H}$, we construct the fiber product $P\times_{\mathcal{H}} E$ as in Definition \ref{defi: fiber product}. From the definition of the principal bibundle $P$, it is easy to show that $P\times_{\mathcal{H}} E$ is a vector bundle over $\mathcal{G}$.

We can also consider $P$ as a left $\mathcal{H}$-space with $h\cdot p:=ph^{-1}$. Hence we can define the fiber product $P\times_{\mathcal{H}} P$. Similarly, we can consider $P$ as a right $\mathcal{G}$-space and define $P\times_{\mathcal{G}} P$.

\begin{lemma}\label{lemma: bibundle fiber product}
There are canonical isomorphisms $\mathcal{H}\simeq P\times_{\mathcal{G}}P$ and $\mathcal{G}\cong P\times_{\mathcal{H}} P$, which are compatible with the actions.
\end{lemma}
\begin{proof}[Proof of Lemma \ref{lemma: bibundle fiber product}] See \cite[Proposition 4.6.2]{del2013lie}.
\end{proof}

\begin{lemma}\label{lemma: isom of forms}
For complex orbifold groupoids $\mathcal{G}$ and $\mathcal{H}$ and a $\mathcal{G}$-$\mathcal{H}$ principal bibundle $P$, we have a canonical isomorphism of $\mathcal{G}$-$\mathcal{H}$ bimodules,
\begin{equation}\label{eq: isom of forms}
\wedge^{i}\overline{T^{*}G_0}  \times_{\mathcal{G}} P\cong P\times_{\mathcal{H}} \wedge^{i}\overline{T^{*}H_0}. 
\end{equation}
\begin{proof}[Proof of Lemma \ref{lemma: isom of forms}]
Since $\mathcal{G}$ and $\mathcal{H}$ are \'{e}tale, both sides are canonically isomorphic to $ \wedge^{i}\overline{T^{*}P} $.
\end{proof}
\end{lemma}

Now for $(E^{\bullet}, A^{E^{\bullet}\prime\prime})\in B(\mathcal{H})$, we define $P_*(E^{\bullet}, A^{E^{\bullet}\prime\prime})$ as follows: the graded bundles are given by $P\times_{\mathcal{H}} E^{\bullet}$. As for the superconnection, we know that
$$
v_i\colon P\times_{\mathcal{H}} E^{\bullet}\to P\times_{\mathcal{H}} \wedge^{i}\overline{T^{*}H_0}  \times_{\mathcal{H}} E^{\bullet+1-i}.
$$
Using the isomorphism \eqref{eq: isom of forms}, we can consider $v_i$ as 
\begin{equation}\label{eq: F on superconnections}
P_*(v_i)\colon P\times_{\mathcal{H}} E^{\bullet}\to \wedge^{i}\overline{T^{*}G_0}  \times_{\mathcal{G}} P \times_{\mathcal{H}} E^{\bullet+1-i}.
\end{equation}
The $P_*(\nabla^{E^{\bullet}\prime\prime})$ is defined in a similar way.

Lemma \ref{lemma: bibundle fiber product} ensures that the map $E\mapsto P\times_{\mathcal{H}} E$ gives an equivalence of categories $\Vect(\mathcal{H})\to \Vect(\mathcal{G})$ with inverse $F\mapsto  P\times_{\mathcal{G}} F$, where we consider $P$ as a right $\mathcal{G}$-space in the second functor. Hence we define $P_*^{-1}$ on bundles as 
$$
F\mapsto  P\times_{\mathcal{G}} F
$$
and define $P_*^{-1}$ on superconnections in the same way as in \eqref{eq: F on superconnections}. By Lemma \ref{lemma: bibundle fiber product} and Lemma \ref{lemma: isom of forms}, it is clear that $P_*^{-1}$ gives an inverse of $P_*$. We finish the proof of Proposition \ref{prop: Morita equivalent and quasi-equivalent}.
\end{proof}

\begin{coro}\label{coro: functoriality of superconnections}
Let $f=(Z,\rho, \sigma)\colon\mathcal{G}\to \mathcal{H}$ be a generalized morphism between complex orbifold groupoids. Then $f$ induces a dg-functor 
\begin{equation}
f^*_b\colon B(\mathcal{H})\to B(\mathcal{G}).
\end{equation}
\end{coro}
\begin{proof}
As shown in Remark \ref{rmk: groupoid from generalized morphism}, from $(Z,\rho, \sigma)$ we can define a groupoid $\mathcal{Z}$ together with morphisms $\phi_{\rho}\colon \mathcal{Z}\to \mathcal{G}$ and $\phi_{\sigma}\colon \mathcal{Z}\to \mathcal{H}$ where $\phi_{\rho}$ is a Morita equivalence. Hence by Proposition \ref{prop: Morita equivalent and quasi-equivalent}, $(\phi_{\rho})^*_b$ is a dg-equivalence. We define $f^*_b$ as $f^*_b:=((\phi_{\rho})^*_b)^{-1}\circ (\phi_{\sigma})^*_b$.
\end{proof}

The dg-category of flat antiholomorphic superconnections on the complex orbifold $X$ is denoted by 
$$B(X).$$

We also need the following local version of antiholomorphic flat superconnections.

\begin{lemma}\label{lemma: local antiholomorphic flat superconnecction}
When we restrict $\mathcal{G}$ to a small open set $U$, $\mathcal{G}|_U$ becomes the transformation groupoid associated with a finite group $G$ acting on a complex manifold $U$. Then an antiholomorphic flat superconnection restricts to a $G$-equivariant antiholomorphic flat superconnection on $U$.  
\end{lemma}
\begin{proof}
Observe that $U\rtimes G$ is an open subgroupoid of $\mathcal{G}$. The Lemma follows from the property that the restriction of an equivariant antiholomorphic flat superconnection to an open subgroupoid is again an equivariant antiholomorphic flat superconnection. 
\end{proof}

We define mapping cones and shifts in $B(X)$. For a degree zero closed map $\phi\colon (E^{\bullet}, A^{E^{\bullet}\prime\prime})\to (F^{\bullet}, A^{F^{\bullet}\prime\prime})$, its mapping cone $(C^{\bullet}, A^{C^{\bullet}\prime\prime}_{\phi})$ is defined by
\begin{equation}
C^{n}=E^{n+1}\oplus F^n,\quad \text{and } A^{C^{\bullet}\prime\prime}=
\begin{bmatrix}
A^{E^{\bullet}\prime\prime}&0\\ 
\phi(-1)^{\deg(\cdot)}& A^{F^{\bullet}\prime\prime}
\end{bmatrix}.
\end{equation}
The shift $(E^{\bullet}, A^{E^{\bullet}\prime\prime})[1]$ is defined by
\begin{equation}
E[1]^{n}=E^{n+1},\quad \text{and } A^{E^{\bullet}\prime\prime}[1]=A^{E^{\bullet}\prime\prime}(-1)^{\deg(\cdot)}.
\end{equation}
It is clear that they give $B(X)$ a pre-triangulated structure, hence its homotopy category $$\underline{B}(X)$$ is a triangulated category.

\subsection{Tensor product}
Let $\mathcal{G}$ be a complex orbifold groupoid and $\mathcal{E}=(E^{\bullet}, A^{E^{\bullet}\prime\prime})$ and $\mathcal{F}=(F^{\bullet}, A^{F^{\bullet}\prime\prime})$ be two objects in $B(\mathcal{G})$. 

\begin{defi}\label{defi: tensor product of antiholo}
We define the tensor product of $\mathcal{E}$ and $\mathcal{F}$, denoted by $\mathcal{E}\widehat{\otimes}_b \mathcal{F}$, to be $$(E^{\bullet}\widehat{\otimes} F^{\bullet}, A^{E^{\bullet}\prime\prime}\widehat{\otimes}_b A^{F^{\bullet}\prime\prime}),$$ where $E^{\bullet}\widehat{\otimes} F^{\bullet}$ is the graded tensor product of graded vector bundles over $G_0$, and $A^{E^{\bullet}\prime\prime}\widehat{\otimes}_b A^{F^{\bullet}\prime\prime}$ is the graded tensor product of superconnections.

In particular, the $\mathcal{G}$-action on $E^{\bullet}\widehat{\otimes} F^{\bullet}$ and $A^{E^{\bullet}\prime\prime}\widehat{\otimes}_b A^{F^{\bullet}\prime\prime}$ is the diagonal action.
\end{defi}

\section{An equivalence of categories}\label{sec:equiv_cat}
The goal here is to explain the equivalence between coherent sheaves and antiholomorphic superconnections on a complex orbifold. This is an extension of the complex manifolds case treated in \cite{bismut2023coherent}.

\subsection{The functor}\label{subsection: the functor}

For a flat antiholomorphic superconnection $(E^{\bullet}, A^{E^{\bullet}\prime\prime})$, let $\mathcal{E}^n$ be the sheafification of 
$$
\bigoplus_{p+q=n} \wedge^p\overline{T^*G_0}\times_G E^q.
$$ 
Note that $\mathcal{E}^n$ is a $\mathcal{G}$-sheaf of $C^{\infty}$-modules hence a sheaf of $\mathcal{O}_{\mathcal{G}}$-modules.

It is clear that $A^{E^{\bullet}\prime\prime}$ gives a map of $\mathcal{O}_{\mathcal{G}}$-modules $\mathcal{E}^n\to \mathcal{E}^{n+1}$.

\begin{prop}\label{prop: sheafification has coherent cohomologies}
The cochain complex of sheaves of $\mathcal{O}_{\mathcal{G}}$-modules $(\mathcal{E}^{\bullet}, A^{E^{\bullet}\prime\prime})$ has coherent cohomologies.
\end{prop}
\begin{proof}
By \cite[Theorem 5.3]{bismut2023coherent}, for each $x$, there exists a local chart $(U,H)$ such that $(\mathcal{E}^{\bullet}|_U, A^{E^{\bullet}\prime\prime})$ is (not necessarily $H$-equivariantly) quasi-isomorphic to a bounded complex of locally free sheaves. By Proposition \ref{prop: weak condition on coherent}, we can make this quasi-isomorphism $H$-equivariant.
\end{proof}

\begin{defi}\label{defi: the functor}
The assignment $(E^{\bullet}, A^{E^{\bullet}\prime\prime})\mapsto (\mathcal{E}^{\bullet}, A^{E^{\bullet}\prime\prime})$ defines a functor
\begin{equation}
F_X\colon \underline{B}(X)\to D^b_{\coh}(X)
\end{equation}
where $\underline{B}(X)$ is the homotopy category of the dg-category of flat antiholomorphic superconnections on $X$.
\end{defi}

\subsection{Equivalence}
For an $\mathcal{F}\in D^b_{\coh}(X)$ considered as a complex of $\mathcal{G}$-sheaves on a groupoid representative $\mathcal{G}=(G_0, G_1)$. Let $\mathcal{F}^{\infty}$ be the tensor product $\mathcal{O}_{G_0}^{\infty}\otimes_{\mathcal{O}_{\mathcal{G}}} \mathcal{F}$ and put
\begin{equation}
\overline{\mathcal{F}}^{\infty}=\mathcal{O}^{\infty}_{G_0}(\wedge \overline{T^*G_0})\otimes_{\mathcal{O}_{\mathcal{G}}} \mathcal{F}.
\end{equation}
We equip $\overline{\mathcal{F}}^{\infty}$ with the Dolbeault differential $\dbar$ and it is clear that the canonical morphism $\mathcal{F}\to \overline{\mathcal{F}}^{\infty}$ is a quasi-isomorphism of complexes of $\mathcal{O}_{\mathcal{G}}$-modules.

\begin{thm}\label{thm: ess surj}
The functor $F_X$ is essentially surjective when $X$ is compact. More precisely, for any object $\mathcal{F}\in D^b_{\coh}(X)$ there exists an object $\mathcal{E}=(E^{\bullet}, A^{E^{\bullet}\prime\prime})\in B(X)$ together with a $\mathcal{G}$-morphism of $\mathcal{O}^{\infty}_{G_0}(\wedge \overline{T^*G_0})$-modules $\phi\colon F_X(\mathcal{E})\to \overline{\mathcal{F}}^{\infty}$ which
is a quasi-isomorphism of complexes of $\mathcal{O}_{G_0}$-modules. Hence $F_X(\mathcal{E})$ and $\mathcal{F}$ are isomorphic in $D^b_{\coh}(X)$.
\end{thm}
\begin{proof}[Proof of Theorem \ref{thm: ess surj}]
By Proposition \ref{prop: global resolution}, there exists a bounded complex of finite-dimensional complex $C^{\infty}$-vector bundles $E^{\bullet}$ on $X$ and a quasi-isomorphism
\begin{equation}\label{eq: quasi-isom of the base complex}
\mathcal{O}_X^{\infty}E^{\bullet}\overset{\sim}{\to} \mathcal{F}^{\infty}.
\end{equation}
Notice that $\mathcal{F}^{\infty}$ has a flat antiholomorphic superconnection
$$
A^{\mathcal{F}}=d^{\mathcal{F}}+\dbar.
$$

We want to use the quasi-isomorphism \eqref{eq: quasi-isom of the base complex} to lift the superconnection $A^{\mathcal{F}}$ to $E^{\bullet}$. For this purpose, we need the following lemma.

\begin{lemma}\label{lemma: hom-proj}
Let $X$ be a complex orbifold and $E^{\bullet}$ be a bounded complex of finite dimensional $C^{\infty}$-vector bundles on $X$. Then for any acyclic complex of sheaves of $\mathcal{O}^{\infty}_X$-modules $N^{\bullet}$, the complex of morphisms $\text{Hom}(E^{\bullet},N^{\bullet})$ is still acyclic.
\end{lemma}
\begin{proof}[Proof of Lemma \ref{lemma: hom-proj}] The claim is standard and we prove bit y induction on the amplitude of $E^{\bullet}$, which is due to \cite[Tag 013R]{stacks-project}. First, if the amplitude is zero, then $E^{\bullet}$ is a single $C^{\infty}$-vector bundle $E$ on $X$. Using a similar argument as in \cite[Proposition 2.4]{segal1968equivariant}, we can find a finite-dimensional trivial vector bundle $T$ on $X$ such that $E$ is a direct summand of $T$. Then the claim is obvious.

Suppose that the claim holds if the amplitude of $E^{\bullet}$ is no more than $l$. Now consider $E^{\bullet}$ with amplitude $l+1$. Let $\phi\colon E^{\bullet}\to N^{\bullet}$ be a closed morphism of degree $k$ with components $\phi^s\colon E^s\to N^{s+k}$. Closedness means
$$
d_{N^{\bullet}}\circ \phi^{s}-(-1)^k\phi^{s+1}\circ d_{E^{\bullet}}=0.
$$
Let $E^m$ be the highest non-zero component of $E^{\bullet}$. Then we have
$$
d_{N^{\bullet}}\circ \phi^{m}=0.
$$
Since $N^{\bullet}$ is acyclic, there exists a map $\psi^m\colon E^m\to N^{m+k-1}$ such that
$$
d_{N^{\bullet}}\circ \psi^m=\phi^m.
$$
Now we consider the ``naive" truncation $\sigma_{\leq m-1}E^{\bullet}$, which is of amplitude $l$. Define 
$$
\tilde{\phi}\colon \sigma_{\leq m-1}E^{\bullet}\to N^{\bullet},\quad
\tilde{\phi}^s=\begin{cases} \phi^s, & \mbox{if } s\leq m-1, \\ \phi^m-(-1)^k\psi^m\circ d_{E^{\bullet}}, & \mbox{if } s=m.\end{cases}
$$
Then it is easy to see that $\tilde{\phi}$ is closed, hence there exists a map $\tilde{\theta}\colon \sigma_{\leq m-1}E^{\bullet}\to N^{\bullet}$ of degree $k-1$ such that
$$
d_{N^{\bullet}}\circ \tilde{\theta}^{s}-(-1)^{k-1}\tilde{\theta}^{s+1}\circ d_{E^{\bullet}}=\tilde{\phi}.
$$
In particular, we have
$$
d_{N^{\bullet}}\circ \tilde{\theta}^{m-1}=\tilde{\phi}^{m-1}=\phi^m-(-1)^k\psi^m\circ d_{E^{\bullet}}.
$$
We now define $\theta\colon E^{\bullet}\to N^{\bullet}$ as
$$
\theta^s=\begin{cases}  \tilde{\theta}^s, & \mbox{if } s\leq m-1, \\  \psi^m, & \mbox{if } s=m. \end{cases}
$$
It is clear that 
$$
d_{N^{\bullet}}\circ\theta^{s}-(-1)^{k-1}\theta^{s+1}\circ d_{E^{\bullet}}=\phi.
$$
\end{proof}

The rest of the proof of Theorem \ref{thm: ess surj} is exactly the same as the proof of \cite[Theorem 6.3.6]{bismut2023coherent}, by Lemma \ref{lemma: hom-proj}.
\end{proof}

The following proposition can be proved in the same way as the proofs of \cite[Theorem 5.7, Theorem 5.10, and Proposition 6.8]{bismut2023coherent}.
\begin{prop}\label{prop: quasi-isom is homo inver}
Let $X$ be a complex orbifold, $(E^{\bullet}, A^{E^{\bullet}\prime\prime})$ and $(F^{\bullet}, A^{F^{\bullet}\prime\prime})$ be two flat antiholomorphic superconnections on $X$. More explicitly, we have
$$
A^{E^{\bullet}\prime\prime}=v_0+\nabla^{E^{\bullet}\prime\prime}+v_2+\ldots,\quad
A^{F^{\bullet}\prime\prime}=u_0+\nabla^{F^{\bullet}\prime\prime}+u_2+\ldots.
$$ 

Let $\phi=\phi_0+\phi_1+\ldots$ be a closed degree $0$ morphism from $(E^{\bullet}, A^{E^{\bullet}\prime\prime})$ to $(F^{\bullet}, A^{F^{\bullet}\prime\prime})$. Then the following are equivalent.
\begin{enumerate}
\item $\phi_0$ gives a homotopy equivalence between complexes of $C^{\infty}$-vector bundles $(E^{\bullet}, v_0)$ and $(F^{\bullet}, u_0)$;
\item $\phi$ is a homotopy equivalence, i.e. $\phi$ has an inverse in the homotopy category $\underline{B}(X)$;
\item $\phi$ induces a quasi-isomorphism between $F_X(E^{\bullet}, A^{E^{\bullet}\prime\prime})$ and $F_X(F^{\bullet}, A^{F^{\bullet}\prime\prime})$, where $F_X$ is the functor introduced in Definition \ref{defi: the functor}.
\end{enumerate}
\end{prop}

\begin{defi}\label{defi: quasi-isomorphism}
We call a closed degree $0$ morphism $\phi$ from $ (E^{\bullet}, A^{E^{\bullet}\prime\prime})$ to $(F^{\bullet}, A^{F^{\bullet}\prime\prime})$ a quasi-isomorphism if it satisfies the equivalent conditions in Proposition \ref{prop: quasi-isom is homo inver}.
\end{defi}

\begin{thm}\label{thm: fully faithful}
The functor $F_X\colon \underline{B}(X)\to D^b_{\coh}(X)$ is fully faithful.
\end{thm}
\begin{proof}
The proof is the same as the proof of \cite[Theorem 6.5.1]{bismut2023coherent}.
\end{proof}

We arrive at the main result of this section.
\begin{coro}\label{coro: equiv of cats}
 $F_X\colon \underline{B}(X)\to D^b_{\coh}(X)$ is an equivalence of triangulated categories.
\end{coro}
\begin{proof}
It is clear that $F_X$ preserves the triangulated structure. Then we conclude by Theorem \ref{thm: ess surj} and Theorem \ref{thm: fully faithful}.
\end{proof}

Next, we discuss some compatibility results of the equivalence $F_X$.
\begin{prop}[Compatibility with pull-backs]\label{prop: compatibility with pull-backs}
Let $f\colon X\to Y$ be a morphism. Let $f^*_b\colon B(Y)\to B(X)$ be as in Proposition \ref{prop: pull back functor} and $Lf^*\colon D^b_{\coh}(Y)\to D^b_{\coh}(X)$ be as in Definition \ref{defi: derived pull back}. Then $f^*_b$ and $Lf^*$ are compatible with $F_X$.
\end{prop}
\begin{proof}
It is a direct check with definitions. And we leave the details to the reader. 
\end{proof}

\begin{prop}[Compatibility with tensor products]\label{prop: compatibility with tensor products}
The tensor product in $B(X)$ as in Definition \ref{defi: tensor product of antiholo} and the derived tensor product in Definition \ref{defi: derived tensor product} are compatible under $F_X$.
\end{prop}
\begin{proof}
For an antiholomorphic flat superconnection $(E^{\bullet}, A^{E^{\bullet}\prime\prime})$, its sheafification $F_X(E^{\bullet}, A^{E^{\bullet}\prime\prime})$ is a complex of flat sheaves, where $F_X$ is given in Definition \ref{defi: the functor}. The claim is then clear.
\end{proof}

\section{Generalized Metrics and Curvature}\label{sec:gen_metrics}
In this section, let $X$ be an $n$-dimensional complex orbifold with a proper complex \'{e}tale effective groupoid representation $\mathcal{G}=(G_0, G_1)$. We introduce the concepts of generalized metric and curvature on a complex orbifold. 

\subsection{Generalized Metrics}
For $\alpha\in  \wedge^{p}T^{*}_{\mathbb{C}}G_0$, let 
\begin{equation*}
\tilde{\alpha}=(-1)^{\frac{p(p+1)}{2}}\alpha, \quad \alpha^*=\bar{\tilde{\alpha}}.
\end{equation*}

Let $E^{\bullet}$ be a graded complex vector bundle over $G_0$. For $A\in \Hom(E^{\bullet},(\bar{E}^{\bullet})^*)$, let $A^*\in \Hom(E^{\bullet},(\bar{E}^{\bullet})^*)$ be its adjoint. For
$$
h=\alpha\widehat{\otimes}A\in \wedge^{\bullet}T^{*}_{\mathbb{C}}G_0\,\widehat{\otimes}\, \Hom(E^{\bullet},(\bar{E}^{\bullet})^*),
$$
we define its adjoint
\begin{equation*}
h^*\colon=\alpha^*\widehat{\otimes}A^*.
\end{equation*}
If $h\in \wedge^{\bullet}T^{*}_{\mathbb{C}}G_0\,\widehat{\otimes}\, \Hom(E^{\bullet},(\bar{E}^{\bullet})^*)$, we can write $h$ as
\begin{equation*}
h=\sum_{i=0}^{2n}h_i,\quad \text{where}\quad h_i\in \wedge^{i}T^{*}_{\mathbb{C}}G_0\,\widehat{\otimes}\, \Hom(E^{\bullet},(\bar{E}^{\bullet})^*).
\end{equation*}

We will often count the degrees in $\wedge^{p}T^{*}G_0$ and $\wedge^{q}\overline{T^{*}G_0}$ as $-p$ and $q$ respectively, and we introduce the corresponding degree $\deg_{-}$. Therefore, for $h\in \wedge^{p,q}T^{*}G_0\,\widehat{\otimes}\, \Hom^r(E^{\bullet},(\bar{E}^{\bullet})^*)$, we define
\begin{equation}\label{deg_}
\deg_{-}h=q-p+r.
\end{equation}

Now we are ready to define generalized metrics.

\begin{defi}\label{defi: generalized metric}
Let $X$ and $\mathcal{G}=(G_0, G_1)$ be as above. Let $E^{\bullet}$ be a $\mathcal{G}$-equivariant graded complex vector bundle on $G_0$. An element $h\in \wedge^{i}T^{*}_{\mathbb{C}}G_0\,\widehat{\otimes}\, \Hom(E^{\bullet},(\bar{E}^{\bullet})^*)$ is called a generalized metric on $E^{\bullet}$ if it satisfies the following conditions:
\begin{enumerate}
\item $h$ is $\mathcal{G}$-equivariant;
\item $h$ is of degree $0$, i.e $\deg_{-}h=0$;
\item $h$ is self-adjoint, i.e. $h^*=h$;
\item the $0$-th component $h_0$ defines a graded Hermitian metric on $E^{\bullet}$.
\end{enumerate}
Let $\mathscr{M}^E$ denote the set of generalized metrics on $E^{\bullet}$.

A generalized metric is said to be pure if $h=h_0$.
\end{defi}

\subsection{Superconnections and generalized metrics}
Recall that a cut-off function $\chi\colon G_0\to \mathbb{R}_{\geq 0}$ is a smooth function on $G_0$ such that for every $x\in G_0$, 
\[
\int_{t(g)=x} \chi(s(g))dg=1. 
\]
According to \cite{tu1999conjecture}, a cut-off function exists on all proper \'etale groupoids $\mathcal{G}$. Furthermore, when $G_0/{G_1}$ is compact, we can choose a cut-off function $\chi$ that has compact support. 

Put $\Omega_\mathcal{G}(G_0, \mathbb{C})=C^{\infty}(G_0,\wedge^{\bullet}T^{*}_{\mathbb{C}}G_0)^{\mathcal{G}}$. If $\alpha$, $\alpha^{\prime}\in \Omega_\mathcal{G}(G_0, \mathbb{C})$, then we put
\begin{equation*}
\theta(\alpha, \alpha^{\prime})=\Big(\frac{\sqrt{-1}}{2\pi}\Big)^n\int_{G_0}\chi \tilde{\alpha}\wedge \overline{\alpha^{\prime}},\quad \text{where } \chi \text{ is the cut-off function}.
\end{equation*}

Let $E^{\bullet}$ be a $\mathcal{G}$-equivariant graded complex vector bundle on $G_0$. Put
$$
\Omega_\mathcal{G}(G_0, E^{\bullet})=C^{\infty}(G_0,\wedge^{\bullet}T^{*}_{\mathbb{C}}G_0\,\widehat{\otimes}\,E^{\bullet})^{\mathcal{G}}, 
$$
equipped with the total degree associated with $\deg_{-}$ on $\Omega_\mathcal{G}(G_0, \mathbb{C})$ and the degree on $E^{\bullet}$.

\begin{defi}\label{defi: metric on forms}
Let $h$ be a generalized metric. For $s$, $s^{\prime}\in \Omega_c(G_0, E^{\bullet})$, put
\begin{equation*}
\theta_h(s,s^{\prime})=\theta(s,hs^{\prime}).
\end{equation*}
\end{defi}

\begin{defi}\label{defi: adjoint superconnection}
Let $\mathcal{E}=(E^{\bullet}, A^{E^{\bullet}\prime\prime})$ be a flat antiholomorphic superconnection on $\mathcal{G}=(G_0, G_1)$ equipped with a generalized metric $h$. Denote the formal adjoint of $A^{E^{\bullet}\prime\prime}$ with respect to $\theta_h$ by $$A^{E^{\bullet}\prime}.$$
\end{defi}

\begin{prop}\label{prop: existence of adjoint superconnection}
Under the above condition, $A^{E^{\bullet}\prime}$ exists, is unique, of degree $-1$, and is $\mathcal{G}$-equivariant.
\end{prop}
\begin{proof}
The existence, uniqueness, and degree condition follow from \cite[Section 7.1]{bismut2023coherent} by forgetting the $\mathcal{G}$-action. Since $A^{E^{\bullet}\prime\prime}$ and $h$ are $\mathcal{G}$-equivariant, the adjoint $A^{E^{\bullet}\prime}$ is also $\mathcal{G}$-equivariant by uniqueness.
\end{proof}

Let $\nabla^{E^{\bullet}\prime}$ be the adjoint of $\nabla^{E^{\bullet}\prime\prime}$ with respect to $\theta_h$ and $v_i^*$ be the adjoint of $v_i$ with respect to $\theta_h$, $i=0$ or $i\geq 2$. We have an analogy of \eqref{eq: decomposition of anti super conn}
\begin{equation}\label{eq: decomposition of adjoint super conn}
A^{E^{\bullet}\prime}=v_0^*+\nabla^{E^{\bullet}\prime}+v_2^*+\ldots.
\end{equation}

\begin{rmk}
If $h$ is a pure metric, then we have
$$
v_i^*\colon E^{\bullet}\to \wedge^{i}T^{*}G_0  \times_G E^{\bullet+i-1}, ~i=0, ~2,\ldots
$$
and $\nabla^{E^{\bullet}\prime}$ is an ordinary $\partial$-connection. This is not the case if $h$ is not pure.
\end{rmk}

\subsection{Curvature}
Let $\mathcal{E}=(E^{\bullet}, A^{E^{\bullet}\prime\prime})$ be a flat antiholomorphic superconnection on $\mathcal{G}=(G_0, G_1)$. Recall that $A^{E^{\bullet}\prime}$ is the formal adjoint of $A^{E^{\bullet}\prime\prime}$ in Definition \ref{defi: adjoint superconnection}. It is clear that
\begin{equation*}
(A^{E^{\bullet}\prime})^2=0.
\end{equation*}

Consider the superconnection on $E^{\bullet}$,
\begin{equation*}
A^{E^{\bullet}}=A^{E^{\bullet}\prime\prime}+A^{E^{\bullet}\prime}.
\end{equation*}
The curvature of $A^{E^{\bullet}}$ is given by
\begin{equation*}
A^{E^{\bullet},2}=[A^{E^{\bullet}\prime\prime},A^{E^{\bullet}\prime}].
\end{equation*}
Then $A^{E^{\bullet},2}$ is a smooth section of $ \wedge^{\bullet}T^{*}_{\mathbb{C}}G_0  \times_G \End(E^{\bullet})$ with total degree $0$ with respect to $\deg_{-}$. 

We also have the Bianchi identities
\begin{equation}\label{eq: Bianchi id}
[A^{E^{\bullet}\prime\prime}, A^{E^{\bullet},2}]=0,~[A^{E^{\bullet}\prime}, A^{E^{\bullet},2}]=0.
\end{equation}
In addition, it is easy to see that
\begin{equation*}
( A^{E^{\bullet},2})^*= A^{E^{\bullet},2}.
\end{equation*}

\section{Chern Character}\label{sec:chern}
In this section, we define the Chern character for a flat antiholomorphic superconnection on an orbifold. 
\subsection{Supertrace and Chern character}
Let $\mathcal{E}=(E^{\bullet}, A^{E^{\bullet}\prime\prime})$ be a flat antiholomorphic superconnection on $\mathcal{G}=(G_0, G_1)$ equipped with a generalized metric $h$. In Definition \ref{defi: beta_G} we have the morphism $\beta_{\mathcal{G}}\colon I\mathcal{G}\to \mathcal{G}$ mapping the inertia groupoid $I\mathcal{G}$ to $\mathcal{G}$. Consider the pullback of $\mathcal{E}$ via $\beta_{\mathcal{G}}$, as in Proposition \ref{prop: pull back functor},
$$
\beta^*_{\mathcal{G},b}\mathcal{E}=(\beta^*_{\mathcal{G}}E^{\bullet}, \beta^*_{\mathcal{G},b}A^{E^{\bullet}\prime\prime}).
$$
In addition, we can also pull back the metric $h$ to $\beta^*_{\mathcal{G},b} h$, which is a generalized metric on $\beta^*_{\mathcal{G},b}\mathcal{E}$. We denote the adjoint of $\beta^*_{\mathcal{G},b}A^{E^{\bullet}\prime\prime}$ under $\beta^*_{\mathcal{G},b} h$ by $\beta^*_{\mathcal{G},b}A^{E^{\bullet}\prime}$. Let 
\begin{equation}\label{eq:connection}
\beta^*_{\mathcal{G},b}A^{E^{\bullet}}=\beta^*_{\mathcal{G},b}A^{E^{\bullet}\prime}+\beta^*_{\mathcal{G},b}A^{E^{\bullet}\prime\prime}.
\end{equation}
We can therefore form the curvature 
\begin{equation}
\beta^*_{\mathcal{G},b}A^{E^{\bullet},2}=[\beta^*_{\mathcal{G},b}A^{E^{\bullet}\prime\prime},\beta^*_{\mathcal{G},b}A^{E^{\bullet}\prime}],
\end{equation}
which is a section of $\wedge^{\bullet}T^{*}_{\mathbb{C}}(I\mathcal{G})_0  \times_{I\mathcal{G}} \End(E^{\bullet})$ with total degree $0$ with respect to $\deg_{-}$

On $I\mathcal{G}$ we also have the tautological section $\tau_{\mathcal{G}}$ with value in $G_1$ as in Definition \ref{defi: tau_G}.
For any $g\in (I\mathcal{G})_0$, let $x$ be the point $s(g)=t(g)\in G_0$. By definition we know that $\beta^*E^{\bullet}_{g}=E^{\bullet}_x$. Since $\mathcal{E}=(E^{\bullet}, A^{E^{\bullet}\prime\prime})$ is $\mathcal{G}$-equivariant, we know that $\tau_{\mathcal{G}}(g)$ acts on $\beta_{\mathcal{G}}^*E^{\bullet}_{g}$.

We have a $\mathcal{G}$-equivariant supertrace map
\begin{equation}
\str\colon  \wedge^{\bullet}T^{*}_{\mathbb{C}}(I\mathcal{G})_0  \times \End(E^{\bullet})\to  \wedge^{\bullet}T^{*}_{\mathbb{C}}(I\mathcal{G})_0 
\end{equation}
that vanishes on supercommutators.

Let $\varphi$ be the following morphism of $\wedge^{\bullet}T^{*}_{\mathbb{C}}(I\mathcal{G})_0$,
$$\varphi\colon \wedge^{\bullet}T^{*}_{\mathbb{C}}(I\mathcal{G})_0\to \wedge^{\bullet}T^{*}_{\mathbb{C}}(I\mathcal{G})_0,\quad \alpha \mapsto (2\pi i)^{\frac{\deg\alpha}{2}} \alpha.$$ 
Here we use the ordinary degree instead of $\deg_{-}$ defined in \eqref{deg_}.

\begin{defi}\label{defi: Chern character}
We set
\begin{equation}\label{eq: Chern character}
\ch(A^{E^{\bullet}\prime\prime},h)=\varphi\str[\tau_{\mathcal{G}}\exp(-\beta^*_{\mathcal{G},b}A^{E^{\bullet},2})]\in \Omega^{(=)}_{I\mathcal{G}}((I\mathcal{G})_0,\mathbb{C}).
\end{equation}
\end{defi}

For $g\in (I\mathcal{G})_0$, let $1, e^{i\theta_1}, \cdots, e^{i\theta_q}$ be the eigenvalues of the 
$\tau_{\mathcal{G}}(g)$ action on $\beta_{\mathcal{G}}^*E^{\bullet}_{g}$. With respect to the eigenvalues,  $\beta_{\mathcal{G}}^*E^{\bullet}_{g}$ decomposes into a direct sum
\begin{equation}\label{eq:eigen-dec}
\beta_{\mathcal{G}}^*E^{\bullet}_{g}=\bigoplus_{j=0}^q \beta_{\mathcal{G}}^*E^{\bullet}_{g, j},
\end{equation}
where $\beta_{\mathcal{G}}^*E^{\bullet}_{g, 0}$ is the eigenspace associated to the eigenvalue 1.  By the $\mathcal{G}$ equivariance, the flat antiholomorphic superconnection $A^{E^\bullet\prime\prime}$  naturally decomposes with respect to the above decomposition, i.e. 
\[
\beta_{\mathcal{G},b}^* A^{E^\bullet\prime\prime}=\sum_j\beta_{\mathcal{G},b}^* A^{E^\bullet\prime\prime}_j,\ \beta^*_{\mathcal{G},b}A^{E^{\bullet}\prime}=\sum_j \beta^*_{\mathcal{G},b}A^{E^{\bullet}\prime}_j,\ \text{and}\ \beta^*_{\mathcal{G},b}A^{E^{\bullet},2}=\sum_j \beta^*_{\mathcal{G},b}A^{E^{\bullet},2}_j. 
\]
Then the Chern character has the following decomposition, 
\[
\ch(A^{E^{\bullet}\prime\prime},h)=\sum_j \varphi\str[e^{i\theta_j}\exp(-\beta^*_{\mathcal{G},b}A^{E^{\bullet},2}_j)].
\]

\begin{rmk}
In the local case, we have $\mathcal{G}=G\ltimes X$ and $I\mathcal{G}=\coprod_{(g)\in \text{Conj}(G)}Z_G(g)\ltimes X^g$ as in \eqref{eq: inertia groupoid local 2}. Then we have
\begin{equation}\label{eq: Chern char defi local}
\varphi\str[\tau_{\mathcal{G}}\exp(-\beta^*_{\mathcal{G},b}A^{E^{\bullet},2})]=\varphi\str[g\exp(-A^{E^{\bullet},2})|_{X^g}],
\end{equation}
which is the definition of equivariant Chern character in \cite[Equation (1.6)]{ma2005orbifolds} and \cite[Equation (4.3.10)]{bismut2011hypoelliptic}.
\end{rmk}

\begin{prop}\label{prop: Chern character is closed}
The form $\ch(A^{E^{\bullet}\prime\prime},h)$ is de Rham closed.
\end{prop}
\begin{proof}
Since $\beta^*_{\mathcal{G},b}A^{E^{\bullet}\prime\prime}$, $\beta^*_{\mathcal{G},b}A^{E^{\bullet}\prime}$, and the de Rham differential operator are $I\mathcal{G}$-equivariant, the claim is a simple consequence of the Bianchi identities \eqref{eq: Bianchi id}.
\end{proof}

Recall that $\mathscr{M}^E$ is the set of generalized metrics on $E^{\bullet}$. Let $\diff^{\mathscr{M}^E}$ denote the de Rham operator on $\mathscr{M}^E$. Then $h^{-1}\diff^{\mathscr{M}^E}h$ is a $1$-form on $\mathscr{M}^E$ with values in degree $0$ morphisms in $\wedge^{\bullet}T^{*}_{\mathbb{C}}I\mathcal{G}_0  \times_{I\mathcal{G}} \End(E^{\bullet})$ that are self-adjoint with respect to $h$.

\begin{thm}\label{thm: Chern character is independent of the metric}
The Bott-Chern cohomology class of $\ch(A^{E^{\bullet}\prime\prime},h)$ is independent of the metric $h$. More precisely we have
\begin{equation}
\diff^{\mathscr{M}^E}\ch(A^{E^{\bullet}\prime\prime},h)=-\frac{\dbar\dpar}{2\pi i}\varphi\str[h^{-1}\diff^{\mathscr{M}^E}h\exp(-A^{E^{\bullet},2})].
\end{equation}
\end{thm}
\begin{proof}
The proof is essentially the same as \cite[Theorem 4.7.1]{bismut2011hypoelliptic}.
\end{proof}

Theorem \ref{thm: Chern character is independent of the metric} implies that the following definition is independent of the choice of the generalized metric $h$.
\begin{defi}\label{defi: Chern character in Bott-Chern}
We denote by $\chBC(A^{E^{\bullet}\prime\prime})$ the Bott-Chern cohomology class of $\ch(A^{E^{\bullet}\prime\prime},h)$ in $\HBC^{(=)}(I\mathcal{G},\mathbb{C})$.
\end{defi}

For a holomorphic vector bundle $E$, the operator $\bar{\partial}$ defines a flat antiholomorphic superconnection on $E$, and the corresponding Chern connection $\nabla$ agrees with the connection on $A^E$ as (in \ref{eq:connection}). The associated $\chBC(A^{E\prime\prime})$ agrees with the classical Chern character for a holomorphic vector bundle on $\mathcal{G}$. 

\begin{rmk}\label{Todd class} Following Definition \ref{defi: Chern character in Bott-Chern}, we similarly define the Todd class using the following function on square matrices, 
\[
\Td(A)=\det\left(\frac{A}{1-e^{-A}}\right),\ e(A)=\det(A).
\]

More explicitly, in the case of a holomorphic vector bundle $E$, using the Chern connection $\nabla$ on $E$
we define $\TdBC(E, \nabla)$ by
\[
\TdBC(E, \nabla)= \Td\left(-\frac{R_0}{2\pi i}\right) \prod_{j=1}^q \frac{\Td(-\frac{R_j}{2\pi i}+i\theta_j)}{e(-\frac{R_j}{2\pi i}+i\theta_j)},
\]
where $R_0, \cdots, R_q$ are the curvature forms of the corresponding eigenbundle $E_j$ associated with the eigenspace decomposition
$E=\oplus_j E_j$ following (\ref{eq:eigen-dec}).  
\end{rmk}
\subsection{Chern Character Form and Scaling of the Metric}
In this subsection, we assume that the metric $h$ is a pure metric. Let $N^E$ be the number operator on $E^{\bullet}$, i.e. for $e\in E^k$ we have $N^E(e)=ke$.

\begin{defi}
For a parameter $T>0$, let $h_T$ be the deformed pure metric defined by
\begin{equation}
h_T=hT^{N^E}.
\end{equation}
In other words, for $e_1, e_2\in E^k$, we have
$$
h_T(e_1,e_2)=T^kh(e_1,e_2).
$$
\end{defi}

Let $(E^{\bullet}, A^{E^{\bullet}\prime\prime})$ be a flat antiholomorphic superconnection with a pure metric $h$. Now we assume that the cohomologies $HE^{\bullet}$ of the complex $(E^{\bullet}, v_0)$ are locally of constant ranks, hence they are $C^{\infty}$-vector bundles. In this case, the $\dbar$-connection $\nabla^{E^{\bullet}\prime\prime}$ induces a flat $\dbar$-connection $\nabla^{HE^{\bullet}\prime\prime}$ on the graded $C^{\infty}$-vector bundle $HE^{\bullet}$.

Put
\begin{equation}
\mathcal{H}E^{\bullet}=\{e\in E^{\bullet}|v_0e=0\text{ and }v_0^*e=0\}.
\end{equation}
By finite-dimensional Hodge theory, we know that $HE^{\bullet}\cong \mathcal{H}E^{\bullet}$. Hence $\mathcal{H}E^{\bullet}$ is a graded $C^{\infty}$-vector subbundle of $E^{\bullet}$. As a subbundle,  $\mathcal{H}E^{\bullet}$ inherits from $E^{\bullet}$ a graded metric which we denote by $h^{\mathcal{H}E}$. Let $\nabla^{\mathcal{H}E}$ be the corresponding Chern connection on $\mathcal{H}E^{\bullet}$.

\begin{lemma}\label{lemma: Chern connection coincide with projection}
We have 
\begin{equation}
\nabla^{\mathcal{H}E}=P\nabla^{E},
\end{equation}
where $\nabla^{E}=\nabla^{E^{\bullet}\prime}+\nabla^{E^{\bullet}\prime\prime}$ as in \eqref{eq: decomposition of anti super conn} and  \eqref{eq: decomposition of adjoint super conn}, and $P\colon E^{\bullet}\to \mathcal{H}E^{\bullet}$ is the orthogonal projection.
\end{lemma}
\begin{proof}
The proof is the same as that of \cite[Theorem 4.10.4]{bismut2011hypoelliptic}.
\end{proof}

Under the isomorphism $HE^{\bullet}\cong \mathcal{H}E^{\bullet}$, we obtain the corresponding graded metric $h^{HE}$ and connection $\nabla^{HE}$ on $HE^{\bullet}$.

\begin{thm}\label{thm: limit T to infinity}
If the cohomologies $HE^{\bullet}$ of the complex $(E^{\bullet}, v_0)$ are locally of constant ranks, then 
\begin{equation}
\ch(A^{E^{\bullet}\prime\prime},h_T)=\ch(\nabla^{HE},h^{HE})+\mathcal{O}(1/\sqrt{T}) \text{ as } T\to \infty. 
\end{equation}
Hence, under the same condition, we have
\begin{equation}
\chBC(A^{E^{\bullet}\prime\prime})=\chBC(HE) \text{ in } \HBC^{(=)}(\mathcal{IG},\mathbb{C}).
\end{equation}
In particular, if the complex $(E^{\bullet}, v_0)$ is acyclic, then 
\begin{equation}
\chBC(A^{E^{\bullet}\prime\prime})=0 \text{ in } \HBC^{(=)}(\mathcal{IG},\mathbb{C}).
\end{equation}
\end{thm}
\begin{proof}
This is a finite-dimensional analogue of \cite[Theorem 4.10.4]{bismut2011hypoelliptic}.
\end{proof}

\subsection{Chern character of mapping cones}
Recall the mapping cone $(C^{\bullet}, A^{C^{\bullet}\prime\prime}_{\phi})$ of a degree $0$ closed map $\phi\colon (E^{\bullet}, A^{E^{\bullet}\prime\prime})\to (F^{\bullet}, A^{F^{\bullet}\prime\prime})$, defined in Section \ref{Section: ahfs}.

\begin{thm}\label{thm: Chern character and mapping cone}
The following identity holds:
\begin{equation}
\chBC( A^{C^{\bullet}\prime\prime}_{\phi})=\chBC(A^{E^{\bullet}\prime\prime})-\chBC(A^{F^{\bullet}\prime\prime}) \text{ in } \HBC^{(=)}(I\mathcal{G},\mathbb{C}).
\end{equation}
In particular, if $\phi$ is a quasi-isomorphism as in Definition \ref{defi: quasi-isomorphism}, then we have
\begin{equation}
\chBC(A^{E^{\bullet}\prime\prime})=\chBC(A^{F^{\bullet}\prime\prime}) \text{ in }  \HBC^{(=)}(I\mathcal{G},\mathbb{C}).
\end{equation}
\end{thm}
\begin{proof}
This is an equivariant version of \cite[Proposition 8.7.1]{bismut2023coherent}. The proof is the same and is omitted. 
\end{proof}

The following is an immediate consequence of Theorem \ref{thm: Chern character and mapping cone}.
\begin{coro}\label{coro: Chern character isomorphism class}
Let $\mathcal{F}\in D^b_{\coh}(\mathcal{G})$ and $(E^{\bullet}, A^{E^{\bullet}\prime\prime})$ be as in Theorem \ref{thm: ess surj}. Then the Bott-Chern class of $\chBC(A^{E^{\bullet}\prime\prime})$ depends only on the isomorphism class of $\mathcal{F}$ in $D^b_{\coh}(\mathcal{G})$.
\end{coro}

The following definition thus makes sense.
\begin{defi}\label{defi: Chern character of derived category}
We denote by $\chBC(\mathcal{F})$ the Bott-Chern cohomology class of $\chBC(A^{E^{\bullet}\prime\prime})$ in Corollary \ref{coro: Chern character isomorphism class}.
\end{defi}

The following is a direct consequence of Corollary \ref{coro: equiv of cats} and Theorem \ref{thm: Chern character and mapping cone}. 
\begin{coro}
The Chern character $\chBC$ can be
viewed as a homomorphism from $K(\mathcal{G})$ into $\HBC^{(=)}(I\mathcal{G},\mathbb{C})$.
\end{coro}
 
Next, we consider two compatibility properties of the Chern character.
\begin{prop}\label{prop: chern character pull back}
For $f\colon \mathcal{G}\to \mathcal{H}$ and $\mathcal{F}\in D^b_{\coh}(\mathcal{H})$, we have
\begin{equation}
\chBC(Lf^*\mathcal{F})=If^*\chBC(\mathcal{F}) \text{ in }\HBC^{(=)}(I\mathcal{G},\mathbb{C}).
\end{equation}
\end{prop}
\begin{proof}
Let $\mathcal{E}$ be an object in $B(X)$ which represents $\mathcal{F}$. The claim is a direct consequence of Proposition \ref{prop: compatibility with pull-backs} and the construction of Chern characters.
\end{proof}

\begin{prop}\label{prop: chern character tensor product}
For $\mathcal{F}_1$ and $\mathcal{F}_2\in D^b_{\coh}(\mathcal{G})$, we have
\begin{equation}
\chBC(\mathcal{F}_1\widehat{\otimes}^L\mathcal{F}_2)=\chBC(\mathcal{F}_1)\chBC(\mathcal{F}_2) \text{ in }\HBC^{(=)}(I\mathcal{G},\mathbb{C}).
\end{equation}
\end{prop}
\begin{proof}
Let $\mathcal{E}_1$ and $\mathcal{E}_2$ be objects in $B(X)$ which represent $\mathcal{F}_1$ and $F_2$ respectively. By Proposition \ref{prop: compatibility with tensor products}, $\mathcal{E}_1\widehat{\otimes}_b \mathcal{E}_2$ represents $\mathcal{F}_1\widehat{\otimes}^L\mathcal{F}_2$. Recall Definition \ref{defi: tensor product of antiholo}, the $\mathcal{G}$-action on $\mathcal{E}_1\widehat{\otimes}_b \mathcal{E}_2$ is given by the diagonal action. The claim then follows from the construction of orbifold Chern characters. We omit the details.
\end{proof}

\section{Riemann-Roch-Grothendieck Theorem for embeddings}\label{sec:RRG_embeddings}

The goal of this section is to prove a Riemann-Roch-Grothendieck theorem for embeddings between complex orbifolds, namely Theorem \ref{thm:RRG_embeddings}, which we reproduce here:
\begin{thm*}[=Theorem \ref{thm:RRG_embeddings}]
Let $i_{X,Y}\colon X\hookrightarrow Y$ be an embedding of compact complex orbifold groupoids. 
Let $\mathcal{F}\in D^b_{\coh}(X)$ and $i_{X,Y,*}\mathcal{F}\in D^b_{\coh}(Y)$ be its direct image. We have
\begin{equation*}
\chBC(i_{X,Y,*}\mathcal{F})=Ii_{X,Y,*}\left(\frac{\chBC(\mathcal{F})}{\TdBC(N_{X/Y})}\right) \text{ in }\HBC^{(=)}(IY,\mathbb{C}),
\end{equation*}
where $Ii_{X,Y}$ is the induced morphism between inertia groupoids.
\end{thm*}

In Proposition \ref{prop: decompose an embedding into two types}, we show that an orbifold embedding can be decomposed into a composition of an iso-spatial embedding followed by a stabilizer-preserving embedding. This splits our proof into the corresponding two cases.

\subsection{Iso-spatial case}
Recall that \eqref{eq: definitionofpushforward} defines the pushforward of sheaves of $\mathcal{O}_X$-modules under a generalized morphism. We define the pushforward of sheaves of $C^{\infty}$-modules in the same way.

The following proposition is important in subsequent constructions.

\begin{prop}\label{prop: push forward of bundle under iso-spatial}
Let $f= (Z, \rho, \sigma)\colon \mathcal{G}\to \mathcal{H}$ be an iso-spatial embedding of complex orbifold groupoids. For any finite-dimensional $C^{\infty}$-vector bundle $E$ on $\mathcal{G}$, its pushforward $f_*(E)$ is a $C^{\infty}$-vector bundle on $\mathcal{H}$.
\end{prop}
\begin{proof}[Proof of Proposition \ref{prop: push forward of bundle under iso-spatial}]
By \eqref{eq: definitionofpushforward}, we know $f_*(E)=\sigma_*\rho^*(E)$. Recall the definition of $\tilde{\mathcal{G}}$ in Section \ref{subsec: orbifold embeddings}. 
Since $f= (Z, \rho, \sigma)$ is iso-spatial (in the sense of Definition \ref{defi: iso-spatial embedding}), $\sigma\colon Z\to H_0$ is a submersion. So $\sigma_*$ is simply taking  $\tilde{\mathcal{G}}$-invariant sections. Note that $\tilde{\mathcal{H}}$ only acts on $Z$ while $\tilde{\mathcal{G}}$ acts on $Z$ as well as on the fibers of $E$.

Since $\rho\colon Z\to G_0$ and $\sigma\colon Z\to H_0$ are local diffeomorphisms, for a point $x\in H_0$, we can pick a neighborhood $U$ of $x$ in $H_0$ such that $\sigma^{-1}(U)$ and $\rho(\sigma^{-1}(U))$  are finite disjoint unions of open subsets which are diffeomorphic to $U$. More precisely,  we know $\sigma^{-1}\{x\}$ is a discrete subset of $Z$ and $\rho(\sigma^{-1}\{x\})$ is a discrete subset of $G_0$. Furthermore, we have diffeomorphisms
\begin{equation}
\sigma^{-1}(U)\cong U\times \sigma^{-1}\{x\},\quad \text{and } \rho(\sigma^{-1}(U))\cong U\times \rho(\sigma^{-1}\{x\}).
\end{equation}

Since the map $\tilde{\mathcal{G}}\hookrightarrow \tilde{\mathcal{H}}$ is saturated,  it is easy to see that $\mathcal{G}$ acts transitively on $\rho(\sigma^{-1}\{x\})$. Pick a point $\tilde{x}\in \rho(\sigma^{-1}\{x\}) \subset G_0$, we have the following result.

\begin{lemma}\label{lemma: rho inverse intersect sigma inverse}
We have the following bijection between sets:
\begin{equation}\label{eq: rho inverse x intersect sigma inverse x}
\rho^{-1}(\tilde{x})\cap \sigma^{-1}(x)\cong \mathcal{H}_x,
\end{equation}
hence a diffeomorphism
\begin{equation}\label{eq: rho inverse U intersect sigma inverse U}
\rho^{-1}(\tilde{U})\cap \sigma^{-1}(U)\cong U\times \mathcal{H}_x,
\end{equation}
where $\mathcal{H}_x$ is the isotropy group of $x$.
\end{lemma}
\begin{proof}[Proof of Lemma \ref{lemma: rho inverse intersect sigma inverse}]
\eqref{eq: rho inverse U intersect sigma inverse U} follows from \eqref{eq: rho inverse x intersect sigma inverse x} so it suffices to prove \eqref{eq: rho inverse x intersect sigma inverse x}.

We pick and fix a $z\in \rho^{-1}(\tilde{x})\cap \sigma^{-1}(x)$. Since the $\mathcal{H}$-action on $Z$ is free, it is clear that $z\mathcal{H}_x $, which is a subset of $\rho^{-1}(\tilde{x})\cap \sigma^{-1}(x)$, is bijective to $\mathcal{H}_x$. 

Moreover, since $Z/\mathcal{H}\cong G_0$, for any $z^{\prime}\in \rho^{-1}(\tilde{x})\cap \sigma^{-1}(x)$, there exists $h\in H_1$ such that $z^{\prime}h=z$. Since $\sigma (z^{\prime})=\sigma (z)=x$, we get $xh=x$, i.e. $h\in \mathcal{H}_x$. Therefore, we know that $z\mathcal{H}_x$ is the whole set $\rho^{-1}(\tilde{x})\cap \sigma^{-1}(x)$, hence the latter is bijective to $\mathcal{H}_x$. This completes the proof of \eqref{eq: rho inverse x intersect sigma inverse x}.
\end{proof}

By definition, we know that
\begin{equation}
\Gamma(U,\sigma_*\rho^*(E))=\Gamma(\sigma^{-1}(U),\rho^*(E))^{\mathcal{G}}.
\end{equation}
Let $\mathcal{G}_{\tilde{x}}$ be the isotropy group at $\tilde{x}\in G_0$. Since the map $\tilde{\mathcal{G}}\hookrightarrow \tilde{\mathcal{H}}$ is saturated, $\mathcal{G}$ acts transitively on $\rho(\sigma^{-1}\{x\})$. Hence we have
\begin{equation}\label{eq: reduce of G to Gx}
\Gamma(\sigma^{-1}(U),\rho^*(E))^{\mathcal{G}}\cong \Gamma(\rho^{-1}(\tilde{U})\cap \sigma^{-1}(U), \rho^*(E))^{\mathcal{G}_{\tilde{x}}}.
\end{equation}

Since the map $\tilde{\mathcal{G}}\hookrightarrow \tilde{\mathcal{H}}$ is an embedding, it is easy to see that $\mathcal{G}_{\tilde{x}}$ acts freely on $\rho^{-1}(\tilde{x})\cap \sigma^{-1}(x)$. Hence $\mathcal{G}_{\tilde{x}}$ can be considered as a subgroup of $\mathcal{H}_x$. Therefore \eqref{eq: rho inverse U intersect sigma inverse U} and \eqref{eq: reduce of G to Gx} imply
\begin{equation}
\Gamma(\sigma^{-1}(U),\rho^*(E))^{\mathcal{G}}\cong \Gamma(U\times (\mathcal{H}_x/\mathcal{G}_{\tilde{x}}), E).
\end{equation}
This completes the proof of Proposition \ref{prop: push forward of bundle under iso-spatial}.
\end{proof}

Recall that the definition of antiholomorphic flat superconnections (Definition \ref{defi: antiholo superconn}) requires the superconnection to be $\mathcal{G}$-equivariant, hence it induces an antiholomorphic flat superconnection on the $\tilde{\mathcal{G}}$-invariant sections. Hence for an antiholomorphic flat superconnection $\mathcal{E}=(E^{\bullet}, A^{E^{\bullet}\prime\prime})$ on $\mathcal{G}$, we can define its pushforward $f_{b,*}(\mathcal{E})$ as an antiholomorphic flat superconnection on $\mathcal{H}$ under an iso-spatial embedding $f= (Z, \rho, \sigma)\colon \mathcal{G}\to \mathcal{H}$. 

\begin{prop}\label{prop: derived push-forward iso-spatial embedding}
Let $f= (Z, \rho, \sigma)\colon \mathcal{G}\to \mathcal{H}$ be an iso-spatial embedding of complex orbifold groupoids. For any $\mathcal{F}\in D^b_{\coh}(\mathcal{G})$ which corresponds to $\mathcal{E}=(E^{\bullet}, A^{E^{\bullet}\prime\prime})$ under the equivalence $F_{\mathcal{G}}$ in Definition \ref{defi: the functor}, the derived pushforward $Rf_*\mathcal{F}$ is given by $f_{b,*}(\mathcal{E})$ under the equivalence $F_{\mathcal{H}}$.
\end{prop}
\begin{proof}
By \cite[Proposition  III.8.12]{gelfand2003methods}, we can choose the soft resolution to compute $Rf_*\mathcal{F}$.
\end{proof}

Now we come to the Riemann-Roch-Grothendieck theorem for iso-spatial embeddings.
\begin{thm}\label{thm: GRR iso-spatial embedding}
Let $f= (Z, \rho, \sigma)\colon \mathcal{G}\to \mathcal{H}$ be an iso-spatial embedding of complex orbifold groupoids. 
Let $\mathcal{F}\in K(\mathcal{G})$ and $f_!\mathcal{F}\in K(\mathcal{H})$ be its direct image. Then we have 
\begin{equation}\label{eq: GRR iso-spatial embedding}
\chBC(f_!\mathcal{F})=If_*\chBC(\mathcal{F}) \text{ in }\HBC^{(=)}(I\mathcal{H},\mathbb{C}),
\end{equation}
where $If\colon  I\mathcal{G}\to I\mathcal{H}$ is the induced generalized morphism as in Definition \ref{defi: generalized morphism for inertial groupoid}.
\end{thm}
\begin{proof}
Recall that, for $\mathcal{G}$ and $\mathcal{H}$, Definition \ref{defi: beta_G} gives natural morphisms $\beta_{\mathcal{G}}\colon I\mathcal{G}\to \mathcal{G}$ and $\beta_{\mathcal{H}}\colon I\mathcal{H}\to \mathcal{H}$. Let $\mathcal{E}$ be an antiholomorphic flat superconnection on $\mathcal{G}$ which corresponds to $\mathcal{F}$. By Definition \ref{defi: Chern character}, we have
\begin{equation}\label{eq: chBC F}
\chBC(\mathcal{F})=\varphi\str[\tau_{\mathcal{G}}\exp(-(\beta_{\mathcal{G},b}^*A^{E^{\bullet}})^2)],
\end{equation}
and
\begin{equation}\label{eq: chBC f!F}
\chBC(f_!\mathcal{F})=\varphi\str[\tau_{\mathcal{H}}\exp(-(\beta_{\mathcal{H},b}^*f_{b,*}A^{E^{\bullet}})^2)],
\end{equation}
where $\tau_{\mathcal{G}}$ and $\tau_{\mathcal{H}}$ are the tautological sections on $I\mathcal{G}$ and $I\mathcal{H}$ respectively, and $\varphi$ is the normalizing operator.

We can prove that, locally, the two sides in \eqref{eq: GRR iso-spatial embedding} are equal at the level of differential forms. Recall that, as in Remark \ref{rmk: two kind embeddings local case}, the iso-spatial embedding is locally given by (up to Morita equivalence) the inclusion of finite groups $G\hookrightarrow H$ and the identity map on the complex manifold $X$.  Moreover, by Remark \ref{rmk: inertia morphism local case}, the induced morphism between the inertia groupoids 
$$
If\colon  \coprod_{(g)\in C(G)}Z_G(g)\ltimes X^g\to  \coprod_{(h)\in C(H)}Z_H(h)\ltimes X^h
$$
consists of the inclusion of finite groups $G\hookrightarrow H$ together with the identity map $X^g\to X^g$.  

Locally we have $\mathcal{E}=(E^{\bullet}, A^{E^{\bullet}\prime\prime})$ where $E^{\bullet}$ is a $\mathbb{Z}$-graded, $G$-equivariant vector bundle on $X$ and $A^{E^{\bullet}\prime\prime}$ is a $G$-equivariant anti-holomorphic flat superconnection on $E^{\bullet}$. Hence \eqref{eq: chBC F} becomes
\begin{equation}
\chBC(\mathcal{F})=\varphi\str[g\exp(-(A^{E^{\bullet}})^2)] \text{ on } X^g.
\end{equation}
In Section \ref{subsec: functorial Bott-Chern cohomology} we use currents to define the pushforward of Bott-Chern cohomology as $If_*=(I\sigma_*)\circ (I\rho_*)^{-1}$. Since $f$ is an iso-spatial embedding, locally we can define $If_*$ using differential forms as shown in Proposition \ref{prop: pushforward map of forms for iso-spatial embedding local case}. For $h\in H$, let $g_1,\ldots, g_k$ be representatives of $G$-conjugate classes of elements in $G$ such that $g_i$ and $h$ are conjugate in $H$. On $X^h$ we have
\begin{equation}
If_*\chBC(\mathcal{F})=\sum_{i=1}^k\frac{1}{Z_G(g_i)}\sum_{\tilde{h}\in Z_H(g_i)}(\tilde{h})^*\Big(\varphi\str[g_i \exp(-(A^{E^{\bullet}})^2)]\Big).
\end{equation}
Moreover, in the local case we have
\begin{equation}
f_{b,*}(\mathcal{E})=(H\times_G E^{\bullet}, 1\times A^{E^{\bullet}\prime\prime}).
\end{equation}
Therefore \eqref{eq: chBC f!F} locally becomes
\begin{equation}
\chBC(f_!\mathcal{F})=\varphi\str[h\exp(-1\times (A^{E^{\bullet}})^2)] \text{ on } X^h.
\end{equation} 
The proof then follows from \cite[Exercise 3.19]{fulton1991representation}.
\end{proof}

\subsection{Transversality and direct image}
In this section, we consider two stabilizer-preserving embeddings of complex orbifold groupoids 
\begin{equation}\label{eqn:two_embeddings}
i_{X,Z}\colon X\hookrightarrow Z, \quad i_{Y,Z}\colon Y\hookrightarrow Z.
\end{equation}
\begin{defi}\label{defi: transverse intersect}
Let $X\hookrightarrow Z$ and $Y\hookrightarrow Z$ be as in (\ref{eqn:two_embeddings}). Let $\mathcal{G}_1$, $\mathcal{G}_2$, and $\mathcal{H}$ be the groupoid representations of $X$, $Y$, and $Z$, respectively, and let $(T_1,\rho_1,\sigma_1)\colon \mathcal{G}_1\to \mathcal{H}$ and $(T_2,\rho_2,\sigma_2)\colon \mathcal{G}_2\to \mathcal{H}$ be two generalized morphisms representing the embeddings. We say that $X$ and $Y$ {\em intersect transversely} if the images of $\sigma_1\colon T_1\to H_0$ and $\sigma_2\colon T_2\to H_0$ intersect transversely. 
\end{defi}

\begin{defi}\label{defi: intersection}
Let $X\hookrightarrow Z$ and $Y\hookrightarrow Z$ be as in (\ref{eqn:two_embeddings}) which intersect transversely. We define $X\cap Y$ to be the fiber product $X\times_Z Y$.
\end{defi}

The following proposition is an orbifold version of \cite[Proposition 9.1.1]{bismut2023coherent}.

\begin{prop}\label{prop: cartesian diagram coherent sheaves}
Let $i_{X,Z}\colon X\hookrightarrow Z$ and $i_{Y,Z}\colon Y\hookrightarrow Z$ be as in (\ref{eqn:two_embeddings}) which intersect transversely. Let $U=X\times_Z Y$ be the fiber product and $i_{U,X}\colon U\to X$, $i_{U,Y}\colon U\to Y$ be the natural maps, summarized in the following diagram:
{
\begin{equation*}
\xymatrixcolsep{3pc}\xymatrix{X\times_Z Y\ar@{=}[r] &U\ar@{^{(}->}[d]_{i_{U,Y}}\ar@{^{(}->}[r]^{i_{U,X}} & X\ar@{^{(}->}[d]^{i_{X,Z}}\\
\, & Y\ar@{^{(}->}[r]_{i_{Y,Z}} & Z.}     
\end{equation*}}
Then for any $\mathcal{F}\in D^b_{\coh}(X)$ there exists an isomorphism in $D^b_{\coh}(Y)$
\begin{equation}
Li_{Y,Z}^*i_{X,Z,*}\mathcal{F}\simeq i_{U,Y,*} Li_{U,X}^*\mathcal{F}.
\end{equation}
Notice that for closed embeddings, the derived pushforward coincides with the pushforward.
\end{prop}
\begin{proof}
Since the statement is local, for any $x\in U$, we can focus on a small neighborhood of $x$ in $Z$. Therefore, we can assume that $X$, $Y$, $Z$, and $U$ are global quotient orbifolds and furthermore, they are suborbifolds of a quotient orbifold associated to $\mathbb{C}^m$. More precisely, we can assume that there exists a finite group $G$ acting holomorphically on a complex manifold $M$ such that
$$
Z=[M/ G].
$$
Moreover, we can assume $M$ is an open subset of $\mathbb{C}^m$ and $x$ corresponds to $0\in \mathbb{C}^m$.

Then we can further assume that there exist $L_1$ and $L_2$ which are open neighborhood of $\mathbb{C}^{l_1}\subset \mathbb{C}^m$ and $\mathbb{C}^{l_2}\subset \mathbb{C}^m$, respectively, and subgroups $H_1$, $H_2$ of $G$ acting on $L_1$ and $L_2$, respectively, such that
$$
X=[L_1/ H_1] \text{ and } Y=[L_2/ H_2.]
$$
In this case, the embeddings of $X$ in $Z$ and $Y$ in $Z$ are given by
\begin{equation}
X\cong [(L_1\times_{H_1} G)/G]\hookrightarrow [M/G],\quad \text{and}\quad 
Y\cong [(L_2\times_{H_2} G)/G]\hookrightarrow [M/G],
\end{equation}
respectively. Then $U$ is given by 
\begin{equation}
U=\left[\left(\coprod_{[g_1]\in G/H_1,~[g_2]\in G/H_2}g_1L_1\cap g_2L_2\right)\Bigg/G\right]. 
\end{equation}

By Corollary \ref{coro: equiv of cats}, there exists a flat antiholomorphic superconnection $(E^{\bullet}, A^{E^{\bullet}\prime\prime})$ on $X=[L_1/ H_1]$ which represents $\mathcal{F}$. By definition, $(E^{\bullet}, A^{E^{\bullet}\prime\prime})$ is actually an $H_1$-equivariant flat superconnection on $L_1$.

The normal bundle $N_{X/Z}$ is then a trivial bundle of dimension $m-l_1$ on $X$. Since $X$ and $Y$ intersect transversely, the normal bundle $N_{U/Y}$ is also a trivial bundle of dimension $m-l_1$ on $U$. We then use the same Koszul resolution as in the proof of \cite[Proposition 9.1.1]{bismut2023coherent} to prove the proposition.
\end{proof}

\subsection{Deformation to the normal cone}
In this subsection we focus on stabilizer-preserving embeddings as in Definition \ref{defi: stabilizer-preserving embedding}.

Let $i_{X,Y}\colon X\hookrightarrow Y$ be a stabilizer-preserving embedding of compact complex orbifold groupoids. Let
$N_{X/Y}$ be the normal bundle to $X$ in $Y$. We construct the deformation to the normal cone of $X$ to $Y$, which generalizes the construction in \cite[Section 4]{bismut1990complex} and \cite[Section 9.2]{bismut2023coherent}.

Let $W$ be the blow-up\footnote{In the case of orbifold groupoids, we can apply the blow up construction introduced by \cite{debord2021blowup, posthumatangwang}.} of $Y\times \mathbb{P}^1$ along $X\times \infty$. Then we have the embedding $i_{X\times \mathbb{P}^1,W}\colon X\times \mathbb{P}^1\hookrightarrow W$. Let $P$ be the exceptional divisor of the blow-up, i.e.
\begin{equation}\label{eq: definition of P}
P\colon=\mathbb{P}(N_{X\times \infty/Y\times \mathbb{P}^1}).
\end{equation}
So we have the natural embedding $i_{X\times \infty, P}\colon X\times \infty \hookrightarrow P$ as the $0$-section.

Let $p_X\colon X\times \infty\to X$ and $p_{\infty}\colon X\times \infty\to \infty$ be the projections. Let 
\begin{equation}
A:=p_X^*N_{X/Y}\otimes p_{\infty}^*N_{\infty/\mathbb{P}^1}^{-1}.
\end{equation}
Then we have 
\begin{equation}
P=\mathbb{P}(A\oplus \underline{\mathbb{C}}), 
\end{equation}
and $i_{\mathbb{P}(N_{X/Y}), P}\colon \mathbb{P}(N_{X/Y})\cong \mathbb{P}(A) \hookrightarrow P$ as the $\infty$-section. 

Let $\widetilde{Y}$ be the blow-up of $Y$ along $X$. The exceptional divisor in $\widetilde{Y}$ of this blow-up may be identified with $\mathbb{P}(N_{X/Y})$. Let $q_{W,Y}\colon W\to Y$ and $q_{W,\mathbb{P}^1}\colon W\to \mathbb{P}^1$ be the obvious maps. For $z\in \mathbb{P}^1$, put
\begin{equation}\label{eq: Yz definition}
Y_z:=q_{W,\mathbb{P}^1}^{-1}z\subset W.
\end{equation}
Then 
\begin{equation}\label{eq: Yz}
Y_z\cong \begin{cases}
    Y,      & ~ \text{if } z\neq\infty,\\
    P\cup \widetilde{Y},  & ~ \text{if } z=\infty.
  \end{cases}
\end{equation}
For $z=\infty$, $P$ and $\widetilde{Y}$ meet transversely along $\mathbb{P}(N_{X/Y})$. The map $q_{W,\mathbb{P}^1}$ is a submersion except on $\mathbb{P}(N_{X/Y})$, where it has ordinary double points as singularities. 

{\tiny\begin{figure}[h]
\begin{tikzpicture}[scale=0.75]

    \draw[thick] (0, 0) -- (11.5, 0) node[right, font=\large] {$\mathbb{P}^1$};

    \draw[thick, orange] (1.1, 0.15) -- (1.1, -0.15) node[below=4pt, text=orange, font=\large] {$\infty$};
    \draw[thick, orange] (9.7, 0.15) -- (9.7, -0.15) node[below=4pt, text=orange, font=\large] {$z$};

    \draw[->, thick, >=stealth] (5.8, 1.5) -- (5.8, 0.6);

    \node[font=\Large] at (5.8, 5.0) {$\dots\dots\dots$};

    
    \coordinate (P_TL) at (0.3, 7.5);
    \coordinate (P_TR) at (2.5, 8.1);
    \coordinate (P_BL) at (0.4, 4.6);
    \coordinate (P_BR) at (2.5, 5.4);

    \draw[thick, purple!70!black] (P_BL) -- (P_TL) -- (P_TR) -- (P_BR);

    \draw[thick, red!80!black] (P_BL) -- (P_BR);

    \node[left, font=\large] at (0.2, 7.4) {$P$};

    \coordinate (X1_L) at (0.33, 6.63);
    \coordinate (X1_R) at (2.5, 7.29);
    \draw[thick, blue!90!black] (X1_L) -- (X1_R) node[right=2pt, font=\large] {$X$};

    \coordinate (Y_BL) at (0.65, 1.8);
    \coordinate (Y_BR) at (2.58, 3.2);

    \draw[thick, teal!80!black] (P_BL) to[bend right=12] (Y_BL);
    \draw[thick, teal!80!black] (P_BR) to[bend right=12] (Y_BR);
    \draw[thick, teal!80!black] (Y_BL) to (Y_BR);

    \node[left, text=red!80!black, font=\large] at (0.3, 4.5) {$\mathbb{P}(A)$};

    \node[left, text=orange!80!black, font=\large] at (0.35, 3.1) {$\tilde{Y}$};

    
    \coordinate (Z_TL) at (8.6, 7.9);
    \coordinate (Z_TR) at (10.7, 8.8);
    \coordinate (Z_BL) at (8.8, 2.3);
    \coordinate (Z_BR) at (10.9, 3.8);

    \draw[thick] (Z_TL) -- (Z_TR);
    \draw[thick] (Z_TL) to[bend left=12] (Z_BL);
    \draw[thick] (Z_TR) to[bend left=12] (Z_BR);
    \draw[thick] (Z_BL) to (Z_BR);

    \coordinate (X2_L) at (8.8, 7.0);
    \coordinate (X2_R) at (10.9, 7.9);
    \draw[thick, blue!90!black] (X2_L) -- (X2_R) node[right=2pt, font=\large] {$X$};

    \node[right, font=\large] at (11.3, 5.2) {$Y_z$};

\end{tikzpicture}
    \caption{The total space $W$}
    \label{fig:enter-label}
\end{figure}}

Let $U=\mathcal{O}_P(-1)$ be the universal line bundle on $P$. We have the following exact sequence of holomorphic vector bundles on $P$:
\begin{equation}\label{eq: universal exact sequence}
0\to U\to A\oplus \underline{\mathbb{C}} \to (A\oplus \underline{\mathbb{C}})/U\to 0.
\end{equation} 
Let $\sigma$ be the image of $1\in \underline{\mathbb{C}}$ in $(A\oplus \underline{\mathbb{C}})/U$. Then $\sigma$ is a holomorphic section of $(A\oplus \underline{\mathbb{C}})/U$ which vanishes exactly on $X\times \infty$. On $\mathbb{P}(A)$, $U$ restricts to the corresponding universal line bundle, the exact sequence \eqref{eq: universal exact sequence} restricts to 
\begin{equation}
0\to U\to A\oplus \underline{\mathbb{C}} \to  (A/U)\oplus  \underline{\mathbb{C}}\to 0,
\end{equation}
and $\sigma$ restricts to the section $1$ of $\underline{\mathbb{C}}$ in $(A/U)\oplus  \underline{\mathbb{C}}$.

We consider the Koszul complex 
$$(\wedge^{\bullet}((A\oplus \underline{\mathbb{C}})/U)^*,i_{\sigma}),$$ 
where $i_{\sigma}$ is the contraction by $\sigma$. This complex provides a resolution of $i_{X\times \infty, P,*}\mathcal{O}_{X\times \infty}$.  The restriction of the Koszul complex to $\mathbb{P}(A)\simeq \mathbb{P}(N_{X/Y})$ is just the split complex $$\wedge^{\bullet}(A/U)^*\widehat{\otimes}(\wedge^{\bullet}(\underline{\mathbb{C}}),i_1).$$

\subsection{Stabilizer-preserving case}
As in \cite[Equation (1.7) and (1.8)]{ma2005orbifolds}, we define the Bott-Chern Todd class $\TdBC(N_{X/Y})\in \HBC^{(=)}(IX,\mathbb{C})$ where $IX$ is the inertia groupoid of $X$.

The following is the Riemann-Roch-Grothendieck theorem for stabilizer-preserving embeddings.
\begin{thm}\label{thm: GRR stablizer-preserving embedding}
Let $i_{X,Y}\colon X\hookrightarrow Y$ be a stabilizer-preserving embedding of compact complex orbifold groupoids. 
Let $\mathcal{F}\in D^b_{\coh}(X)$ and $i_{X,Y,*}\mathcal{F}\in D^b_{\coh}(Y)$ be its direct image. We have
\begin{equation}
\chBC(i_{X,Y,*}\mathcal{F})=Ii_{X,Y,*}\left(\frac{\chBC(\mathcal{F})}{\TdBC(N_{X/Y})}\right) \text{ in }\HBC^{(=)}(IY,\mathbb{C}).
\end{equation}
\end{thm}
\begin{proof}
Let $p_{X\times \mathbb{P}^1, X}\colon X\times \mathbb{P}^1\to X$ be the projection. Since  $p_{X\times \mathbb{P}^1, X}$ is flat, $p_{X\times \mathbb{P}^1, X}^*=Lp_{X\times \mathbb{P}^1, X}^*$.

Recall that we have the natural embedding $i_{X\times \mathbb{P}^1,W}\colon X\times \mathbb{P}^1\hookrightarrow W$. For $\mathcal{F}\in D^b_{\coh}(X)$ we have 
\begin{equation}
i_{X\times \mathbb{P}^1,W,*}p_{X\times \mathbb{P}^1, X}^* \mathcal{F}\in D^b_{\coh}(W).
\end{equation}

Moreover, let 
\begin{equation}
\begin{split}
i_{Y,W}\colon  Y\hookrightarrow W, ~i_{P,W}\colon P\hookrightarrow W,\quad 
i_{\widetilde{Y},W}\colon  \widetilde{Y}\hookrightarrow W, ~i_{X,X\times \mathbb{P}^1}\colon X\times 0=X\hookrightarrow X\times \mathbb{P}^1
\end{split}
\end{equation}
be the natural embeddings, where the last embedding is the inclusion of $0\in \mathbb{P}^1$ into $\mathbb{P}^1$. We summarize the various maps in the following diagram:
{
\begin{equation*}
\xymatrixcolsep{5pc}\xymatrix{
&Y\ar@{^{(}->}[r]_{i_{Y,W}}& W & Y\ar@{_{(}->}[l]^{i_{Y,W}}& P\ar@/_10pt/@{_{(}->}[ll]_{i_{P,W}}\\
X\ar@{=}[r]&X\times \infty\ar@{^{(}->}[r]^{i_{X\times\infty,X\times\mathbb{P}^1}}\ar@{^{(}->}[u]^{i_{X\times\infty,Y}}& X\times\mathbb{P}^1\ar@{^{(}->}[u]_{i_{X\times\mathbb{P}^1,W}}\ar@/_15pt/[rr]|{p_{X\times\mathbb{P}^1,X}}\ar@/^15pt/[ll]|{p_{X\times\mathbb{P}^1,X}}& X\times 0\ar@{_{(}->}[l]_{i_{X,X\times\mathbb{P}^1}}\ar@{^{(}->}[u]_{i_{X,Y}}& X\ar@{=}[l]\ar@{^{(}->}[u]_{i_{X,P}}.
}    
\end{equation*}}

Notice that the composition $p_{X\times \mathbb{P}^1, X}\circ i_{X, X\times \mathbb{P}^1}$ is the identity map on $X$, so we have
\begin{equation}
Li_{X, X\times \mathbb{P}^1}^*~p_{X\times \mathbb{P}^1, X}^*\mathcal{F}\simeq \mathcal{F} \in D^b_{\coh}(X).
\end{equation}
By the same reason, if we identify $X\times \infty$ with $X$, then we get
\begin{equation}
Li_{X\times \infty, X\times \mathbb{P}^1}^*~p_{X\times \mathbb{P}^1, X}^*\mathcal{F}\simeq \mathcal{F} \in D^b_{\coh}(X).
\end{equation}

Recall that $W$ is the blow-up of $Y\times \mathbb{P}^1$ along $X\times \infty$ with $P$ the exceptional divisor. Then $Y$ and $P$ are both transverse to $X\times \mathbb{P}^1$ in $W$. By Proposition \ref{prop: cartesian diagram coherent sheaves} we have
\begin{equation}
Li_{Y,W}^*\circ i_{X\times \mathbb{P}^1,W,*}\circ p_{X\times \mathbb{P}^1,X}^*\mathcal{F}\simeq i_{X,Y,*} \mathcal{F} \text{ in }  D^b_{\coh}(Y),
\end{equation}
and
\begin{equation}
Li_{P,W}^*\circ i_{X\times \mathbb{P}^1,W,*}\circ p_{X\times \mathbb{P}^1,X}^*\mathcal{F}\simeq i_{X,P,*} \mathcal{F} \text{ in }  D^b_{\coh}(P).
\end{equation}

Then by Proposition \ref{prop: chern character pull back} we have
\begin{equation}\label{eq: chern character pull back 1}
Ii_{Y,W}^*\chBC(i_{X\times \mathbb{P}^1,W,*}\circ p_{X\times \mathbb{P}^1,X}^*\mathcal{F})=\chBC( i_{X,Y,*} \mathcal{F}) \text{ in } \HBC^{(=)}(IY,\mathbb{C})
\end{equation}
and
\begin{equation}\label{eq: chern character pull back 2}
Ii_{P,W}^*\chBC(i_{X\times \mathbb{P}^1,W,*}\circ p_{X\times \mathbb{P}^1,X}^*\mathcal{F})=\chBC( i_{X,P,*} \mathcal{F}) \text{ in } \HBC^{(=)}(IP,\mathbb{C}).
\end{equation}

Let $z$ be the canonical meromorphic function on $\mathbb{P}^1$ that vanishes at $0$ and has a pole at $\infty$. We have the Poincar\'{e}-Lelong equation
\begin{equation}\label{eq: Poincare-Lelong equation}
\frac{\dbar^{\mathbb{P}^1}\dpar^{\mathbb{P}^1}}{2\pi i}\log(|z|^2)=\delta_0-\delta_{\infty}.
\end{equation}

Recall \eqref{eq: Yz definition} that $Y_z=q_{W,\mathbb{P}^1}^{-1}z\subset W$ and $IY_z$ is the inertia groupoid of $Y_z$.

Let $\delta_{IY_0}$ and $\delta_{IY_{\infty}}$ be the currents on $IW$ defined by integration along $IY_0$ and $IY_{\infty}$ respectively. Since $q_{W,\mathbb{P}^1}$ has ordinary double points as singularities near $\mathbb{P}(N_{X/Y})$ in the sense of orbifolds, we have an integrable current $q_{IW,\mathbb{P}^1}^*\log(|z|^2)$ in $IW$.  
Then \eqref{eq: Poincare-Lelong equation} gives
\begin{equation}\label{eq: pull back of Poincare-Lelong equation}
\frac{\dbar^{IW}\dpar^{IW}}{2\pi i}q_{IW,\mathbb{P}^1}^*\log(|z|^2)=\delta_{IY_0}-\delta_{IY_{\infty}}.
\end{equation}

Let $\alpha\in  \Omega^{(=)}(IW,\mathbb{C})$ denote the smooth form representing $\chBC(i_{X\times \mathbb{P}^1,W,*}\circ p_{X\times \mathbb{P}^1,X}^*\mathcal{F})$ in $\HBC^{(=)}(IW,\mathbb{C})$. Then \eqref{eq: pull back of Poincare-Lelong equation} gives 
\begin{equation}\label{eq: pull back of Poincare-Lelong equation with alpha}
\frac{\dbar^{IW}\dpar^{IW}}{2\pi i}\big(\alpha q_{IW,\mathbb{P}^1}^*\log(|z|^2)\big)=\alpha\delta_{IY_0}-\alpha\delta_{IY_{\infty}}.
\end{equation}
Let $q_{Y_{\infty},Y}$ be the restriction of $q_{W,Y}$ to $Y_{\infty}$. For $z\in \mathbb{P}^1$, let $Ii_z\colon IY_z\hookrightarrow IW$ be the embedding. It is clear that $Ii_0=Ii_{Y,W}$, which is the induced morphism of $i_{Y,W}$ on inertia defined before. Then \eqref{eq: pull back of Poincare-Lelong equation with alpha} gives 
\begin{equation}\label{eq: pull back to Y of Poincare-Lelong equation with alpha}
\frac{\dbar^{IY}\dpar^{IY}}{2\pi i}q_{IW,IY,*}\big(\alpha q_{W,\mathbb{P}^1}^*\log(|z|^2)\big)=Ii_0^*\alpha-q_{IY_{\infty},IY,*}Ii_{\infty}^*\alpha.
\end{equation}
By \eqref{eq: chern character pull back 1} and \eqref{eq: chern character pull back 2}, we know that $Ii_0^*(\alpha)=Ii_{Y,W}^*(\alpha)$ represents $\chBC( i_{X,Y,*} \mathcal{F})$ in $\HBC^{(=)}(IY,\mathbb{C})$, and $Ii_{P,W}^*(\alpha)$ represents $\chBC( i_{X,P,*} \mathcal{F})$ in $\HBC^{(=)}(IP,\mathbb{C})$. Then \eqref{eq: pull back to Y of Poincare-Lelong equation with alpha} gives 
\begin{equation}\label{eq: chern character of push forward deform}
\chBC( i_{X,Y,*} \mathcal{F})=q_{IY_{\infty},IY,*}Ii_{\infty}^*\alpha  \text{ in } \HBC^{(=)}(IY,\mathbb{C}).
\end{equation}

Recall \eqref{eq: Yz} that $Y_{\infty}=P\cup \widetilde{Y}$. Let $q_{P,Y}$ and $q_{\widetilde{Y},Y}$ be the restriction of $q_{Y_{\infty},Y}$ to $P$ and $\widetilde{Y}$ respectively.  Notice that
\begin{equation}\label{eq: chern character of push forward decompose}
q_{IY_{\infty},IY,*}Ii_{\infty}^*\alpha=q_{IP,IY,*}i_{IP,IW}^*\alpha+q_{I\widetilde{Y},IY,*}Ii_{\widetilde{Y},W}^*\alpha.
\end{equation}
Since $\widetilde{Y}\cap (X\times \mathbb{P}^1)=\emptyset$, we have
\begin{equation}\label{eq: pull back to Y tilde is zero}
Li_{\widetilde{Y},W}^*\circ i_{X\times \mathbb{P}^1,W,*}\circ p_{X\times \mathbb{P}^1,X}^*\mathcal{F}\simeq 0 \text{ in } D^b_{\coh}(\widetilde{Y}).
\end{equation}
By Proposition \ref{prop: chern character pull back} and the definition of $\alpha$, \eqref{eq: pull back to Y tilde is zero} gives $Ii_{\widetilde{Y},W}^*\alpha=0$ in $ \HBC^{(=)}(I\widetilde{Y},\mathbb{C})$. Hence
\begin{equation}\label{eq: vanish of alpha on Y tilde}
q_{I\widetilde{Y},IY,*}Ii_{\widetilde{Y},W}^*\alpha=0  \text{ in } \HBC^{(=)}(IY,\mathbb{C}).
\end{equation}
Combining \eqref{eq: chern character pull back 2}, \eqref{eq: chern character of push forward deform}, \eqref{eq: chern character of push forward decompose}, and \eqref{eq: vanish of alpha on Y tilde}, we get
\begin{equation}\label{eq: chern character after deform}
\chBC( i_{X,Y,*} \mathcal{F})=q_{IP,IY,*}\chBC(i_{X\times \infty,P,*}\mathcal{F})  \text{ in } \HBC^{(=)}(IY,\mathbb{C}).
\end{equation}

We can compute the right-hand side of \eqref{eq: chern character after deform} explicitly.

Recall that  $(\wedge^{\bullet}((A\oplus \underline{\mathbb{C}})/U)^*,i_{\sigma})$, with  $\sigma$  the image of $1\in \underline{\mathbb{C}}$ in $(A\oplus \underline{\mathbb{C}})/U$, provides
a Koszul resolution of $i_{X\times \infty, P, *}\mathcal{O}_{X\times \infty}$.  Let $q_{P,X\times \infty}$ be the natural projection. By the equivariant version of the  projection formula  \cite[\href{https://stacks.math.columbia.edu/tag/0943}{Tag 0943}]{stacks-project}, we have
\begin{equation}\label{eq: projection formula to change F}
\begin{split}
i_{X\times \infty,P,*}\mathcal{F}\simeq Lq_{P,X\times \infty}^*\mathcal{F}\widehat{\otimes}^L_{\mathcal{O}_P}i_{X\times \infty, P, *}\mathcal{O}_{X\times \infty}\simeq Lq_{P,X\times \infty}^*\mathcal{F}\widehat{\otimes}_{\mathcal{O}_P}(\wedge^{\bullet}((A\oplus \underline{\mathbb{C}})/U)^*,i_{\sigma}).
\end{split}
\end{equation}
Notice that since $(\wedge^{\bullet}((A\oplus \underline{\mathbb{C}})/U)^*,i_{\sigma})$ is a complex of locally free sheaves, the derived tensor product coincides with the tensor product.
By Proposition \ref{prop: chern character pull back} and Proposition \ref{prop: chern character tensor product} we have
\begin{equation}
\chBC(i_{X\times \infty,P,*}\mathcal{F})=(Iq_{P,X\times \infty}^*\chBC(\mathcal{F}))\chBC(\wedge^{\bullet}((A\oplus \underline{\mathbb{C}})/U)^*,i_{\sigma}).
\end{equation}
Since $i_{X,Y}: X\to Y$ is a stabilizer-preserving embedding, it follows from Proposition \ref{prop: inertia morphism of stablizer-preserving embedding is  a stablizer-preserving embedding} that the map $Ii_{X,Y}$ is a stabilizer-preserving embedding. Furthermore, we can directly check 
\[
I_{q_{P, Y}}=Ii_{X,Y}\circ I_{q_{P, X\times\infty}}
\]
by working with the local models
\[
H\ltimes M\rightarrow H\ltimes
W,
\]
where $H$ is a finite group and $M\hookrightarrow W$ is an $H$-equivariant embedding.

By the projection formula for Bott-Chern cohomology\footnote{This follows from a suitable generalization of \cite[Eq. (1.15)]{berline2004heat} to the Bott-Chern cohomology.}, 
we get
\begin{equation}\label{eq: chern character to Koszul}
\begin{split}
Iq_{P,Y,*}\chBC(i_{X\times \infty,P,*}\mathcal{F}) &=Ii_{X,Y,*}\circ Iq_{P,X\times \infty,*}\Big((Iq_{P,X\times \infty}^*\chBC(\mathcal{F}))\chBC(\wedge^{\bullet}((A\oplus \underline{\mathbb{C}})/U)^*,i_{\sigma})\Big)\\
& =Ii_{X,Y,*}\Big(\chBC(\mathcal{F})Iq_{P,X\times \infty,*}\chBC(\wedge^{\bullet}((A\oplus \underline{\mathbb{C}})/U)^*,i_{\sigma})\Big).
\end{split}
\end{equation}

By \cite[Theorem 6.7]{bismut1995equivariant}, we get
\begin{equation}\label{eq: chern character of Koszul}
\chBC(\wedge^{\bullet}((A\oplus \underline{\mathbb{C}})/U)^*,i_{\sigma})=\TdBC(N_{X\times \infty,P})^{-1}\delta_{IX\times \infty}  \text{ in } \HBC^{(=)}(IP,\mathbb{C}).
\end{equation}
By \eqref{eq: definition of P} it is also clear that $N_{X\times \infty,P}=N_{X/Y}$ via the identification $X\times \infty\cong X$. Hence \eqref{eq: chern character of Koszul} gives
\begin{equation}\label{eq: Iq Chern character}
    Iq_{P,X\times \infty,*}\chBC(\wedge^{\bullet}((A\oplus \underline{\mathbb{C}})/U)^*,i_{\sigma})=\TdBC(N_{X/Y})^{-1}\text{ in }\HBC^{(=)}(IX,\mathbb{C}).
\end{equation}

Now \eqref{eq: chern character to Koszul} and \eqref{eq: Iq Chern character} together give 
\begin{equation}
\chBC(i_{X,Y,*}\mathcal{F})=Ii_{X,Y,*}\left(\frac{\chBC(\mathcal{F})}{\TdBC(N_{X/Y})}\right) \text{ in }\HBC^{(=)}(IY,\mathbb{C}).
\end{equation}
\end{proof}

\subsection{General case}\label{sec:pf_RRG_emb}

Combining Theorem \ref{thm: GRR iso-spatial embedding} and Theorem \ref{thm: GRR stablizer-preserving embedding} we get Theorem \ref{thm:RRG_embeddings}, the Riemann-Roch-Grothendieck theorem for embeddings of complex orbifolds.

\begin{proof}[Proof of Theorem \ref{thm:RRG_embeddings}]
By Proposition \ref{prop: decompose an embedding into two types}, $i_{X,Y}$ is the composition of two maps:  
\begin{equation*}
\begin{split}    
&i_1\colon X\hookrightarrow \bar{Y}, \quad \text{an iso-spatial embedding},\\
&i_2\colon \bar{Y}\hookrightarrow Y, \quad \text{a stabilizer-preserving embedding}.
\end{split}
\end{equation*}

By Theorem \ref{thm: GRR iso-spatial embedding} and Theorem \ref{thm: GRR stablizer-preserving embedding} we get
\begin{equation}
\chBC(i_{X,Y,*}\mathcal{F})=\chBC(i_{2,*}i_{1,*}\mathcal{F})=Ii_{2,*}\left(\frac{\chBC(i_{1,*}\mathcal{F})}{\TdBC(N_{\bar{Y}/Y})}\right)=Ii_{2,*}\left(\frac{i_{1,*}\chBC(\mathcal{F})}{\TdBC(N_{\bar{Y}/Y})}\right).
\end{equation}
By the projection formula for Bott-Chern cohomology, we have
\begin{equation}
\frac{Ii_{1,*}\chBC(\mathcal{F})}{\TdBC(N_{\bar{Y}/Y})}=Ii_{1,*}\left(\frac{\chBC(\mathcal{F})}{Ii_1^*\TdBC(N_{\bar{Y}/Y})}\right).
\end{equation}
Therefore 
\begin{equation}\label{eq: chBC next to final}
\chBC(i_{X,Y,*}\mathcal{F})=Ii_{2,*}Ii_{1,*}\left(\frac{\chBC(\mathcal{F})}{Ii_1^*\TdBC(N_{\bar{Y}/Y})}\right)=Ii_{X,Y,*}\left(\frac{\chBC(\mathcal{F})}{Ii_1^*\TdBC(N_{\bar{Y}/Y})}\right).
\end{equation}

We want to show that 
\begin{equation}\label{eq: TdBc pullback}
Ii_1^*\TdBC(N_{\bar{Y}/Y})=\TdBC(N_{X/Y}).
\end{equation}
Since $\TdBC$ is defined at the level of differential forms, the identity (\ref{eq: TdBc pullback}) can be proved by working locally on $Y$. 

Locally on $Y$, the factorization $i_{X,Y}=i_2\circ i_1$ can be described as
\begin{equation}\label{eqn:embedding_local}
i_{X,Y}\colon G\ltimes M\overset{i_1}{\longrightarrow} H\ltimes M\overset{i_2}{\longrightarrow} H\ltimes \mathsf{W},    
\end{equation}
where $M\subset \mathsf{W}$ is an embedding of manifolds, $H$ is a finite group and $G\subset H$ is a subgroup. The normal bundle $N_{i_2}$ is the normal bundle $N_{M/\mathsf{W}}$ together with the $H$-action. The pullback $i_1^*N_{i_2}$ is the bundle $N_{M/\mathsf{W}}$ with the $G$-action induced from the $H$-action and $G\subset H$. Therefore the equality 
\begin{equation}
i_1^*N_{i_2}=N_{i_{X,Y}}    
\end{equation}
is valid locally and globally. Therefore 
\begin{equation}\label{eqn:Todd1}
\TdBC(N_{i_{X,Y}})=\TdBC(i_1^*N_{i_2}).
\end{equation}

It remains to show 
\begin{equation}\label{eqn:Todd2}
\TdBC(i_1^*N_{i_2})=Ii_1^*\TdBC(N_{i_2}),
\end{equation}
for which we can work locally as in (\ref{eqn:embedding_local}). The induced map $Ii_1$ between inertia orbifolds can be described locally as
\begin{equation}\label{inertiamap}
\coprod_{(g)\in \text{Conj}(G)} Z_G(g)\ltimes M^g \longrightarrow \coprod_{(h)\in\text{Conj}(H)} Z_H(h)\ltimes M^h
\end{equation}
where $Z_G(g)\to Z_H(h)$ is induced from $G\subset H$. Therefore, a component $Z_H(h)\ltimes M^h$ is in the image of $Ii_1$ if and only if $h$ is conjugate in $H$ to some $g\in G\subset H$. In this situation, consider the restriction of $Ii_1$:
\begin{equation}
i_g\colon Z_G(g)\ltimes M^g\to Z_H(h)\ltimes M^h=Z_H(g)\ltimes M^g.   \end{equation}
Then we have the decomposition into eigenbundles of $g$:
\begin{equation}
N_{M/\mathsf{W}}|_{M^g}=\bigoplus_k N_k,
\end{equation}
which induces isomorphic eigenbundle decomposition of $N_{i_2}|_{Z_H(g)\ltimes M^g}$ and $(i_1^*N_{i_2})|_{Z_G(g)\ltimes M^g}$. The definition of $\TdBC$ as a differential form then implies that 
\begin{equation}
\TdBC(i_1^*(N_{i_2}))|_{Z_G(g)\ltimes M^g}=Ii_g^*(\TdBC(N_{i_2})|_{Z_H(g)\ltimes M^g}). 
\end{equation}
This proves (\ref{eqn:Todd2}). Equation (\ref{eq: TdBc pullback}) follows from (\ref{eqn:Todd1}) and (\ref{eqn:Todd2}).

Finally, \eqref{eq: chBC next to final} and \eqref{eq: TdBc pullback} together give \eqref{eq: GRR for embedding}.
\end{proof}

\subsection{Uniqueness of orbifold Chern character}
Here we present a proof of Theorem \ref{thm:unique_Chern}, which we reproduce below:
\begin{thm*}[=Theorem \ref{thm:unique_Chern}]
The orbifold Chern character $\chBC\colon K(X)\to \HBC^{(=)}(IX,\mathbb{C})$ in (\ref{eq:intro-ch}) is the unique group homomorphism satisfying 
\begin{enumerate}
\item[($\star 1$)]     
For holomorphic vector bundles $E$ on complex orbifolds, our definition of $\chBC(E)$ agrees with the one in \cite[Section 1.2]{ma2005orbifolds}.
\item [($\star 2$)] 
$\chBC$ is functorial under pullbacks.

\item [($\star 3$)]
$\chBC$ satisfies the Riemann-Roch-Grothendieck formula for orbifold embeddings, Equation (\ref{eq: GRR for embedding}).
\end{enumerate}
\end{thm*}

\begin{proof}
Certainly $\chBC$ defined in this paper satisfies the conditions ($\star 1$)-($\star 3$). To show that any such group homomorphism must coincide with $\chBC$, we may adopt the argument in the proof of \cite[Theorem 8]{Grivaux}. This requires two general results:
\begin{itemize}
\item[(A)] {\em For a coherent sheaf $\mathcal{F}$ on $X$, there exists a bimeromorphic map $\pi\colon \tilde{X}\to X$ which is the composition of a sequence of blow-ups along smooth centers, such that the pullback $\pi^*\mathcal{F}$ admits a locally free quotient of maximal rank.}

 As indicated in \cite[Proposition 7]{Grivaux}, this is an immediate consequence of Hironaka's flattening result \cite{Hironaka_flattening}. To make this work for orbifolds, we simply observe that Hironaka's flattening result is valid for complex orbifolds, because the desired property (flatness) and the construction (blow-ups along smooth centers) commute with \'etale base changes.    

\item[(B)] {\em There exists an isomorphism between $G$-theory $G(D)$ of coherent sheaves on a smooth divisor $D\subset \tilde{X}$ and the $G$-theory $G_D(\tilde{X})$ of coherent sheaves on $\tilde{X}$ supported on $D$.}

By the argument of Proposition 7.2 in the arXiv version of \cite{Grivaux}, this follows from the d\'evissage theorem (see e.g. \cite[Devissage Theorem 6.3]{Weibel_K-book}) applied to abelian categories of coherent sheaves on orbifolds.
\end{itemize}

With these in hand, we proceed to prove the Theorem. Suppose $\ch': K(-)\to \HBC^{(=)}(-,\mathbb{C})$ is another group homomorphism satisfying ($\star 1$)-($\star 3$). We prove $\chBC=\ch'$ by induction on the dimension of the orbifold $X$. The base case, i.e. the $0$-dimensional case, is obvious because ($\star 1$) fixes the Chern character in terms of representation theory of the stabilizer groups.

Suppose the dimension of $X$ is positive. Given a coherent sheaf $\mathcal{F}$ on $X$, we would like to show that $\chBC(\mathcal{F})=\ch'(\mathcal{F})$. By (A), there exists a bimeromorphic map $\pi\colon \tilde{X}\to X$ which is the composition of a sequence of blow-ups along smooth centers, such that the pullback $\pi^*\mathcal{F}$ admits a locally free quotient $\mathcal{E}$ of maximal rank. Furthermore, Hironaka's result allows us to assume that the support of the kernel of the quotient map $\pi^*\mathcal{F}\to \mathcal{E}$ and the exceptional locus of $\pi$ are both contained in a simple normal crossing divisor $D\subset \tilde{X}$. Therefore, the difference $\pi^![\mathcal{F}]-[\mathcal{E}]$ is in the $G$-theory $G_D(\tilde{X})$ of coherent sheaves on $\tilde{X}$ supported on $D$. 

Let $D_i\subset D$ be the irreducible components of $D$. By assumption, $D_i$ are smooth of dimension one less than that of $\tilde{X}$. The natural map $\oplus_i G_{D_i}(\tilde{X})\to G_D(\tilde{X})$ is surjective. By (B), $G_{D_i}(\tilde{X})=G(D_i)$.

By induction, $\chBC=\ch'$ on complex orbifolds of dimensions less than that of $\tilde{X}$. Therefore $\chBC$ and $\ch'$ define the same Todd class on $D_i$. By ($\star 3$), $\chBC$ and $\ch'$ are equal on $G_{D_i}(\tilde{X})$. By ($\star 1$), $\chBC(\mathcal{E})=\ch'(\mathcal{E})$. Therefore, $\chBC(\pi^![\mathcal{F}])=\ch'(\pi^![\mathcal{F}])$. By ($\star 2$), it follows that $\pi^*(\chBC(\mathcal{F})-\ch'(\mathcal{F}))=0$. Since $\pi$ is a composition of blow-ups along smooth centers, the pull-back map $\pi^*$ on cohomology is injective. Thus $\chBC(\mathcal{F})=\ch'(\mathcal{F})$.
\end{proof}

\section{An Example}\label{sec:example}

The goal of this section is to present an example to which Theorem \ref{thm:RRG_embeddings} applies, although the pushforward sheaf does not admit a finite locally free resolution. The example also shows that the use of Bott-Chern cohomology is essential, since the relevant Chern class is nonzero in Bott-Chern cohomology but vanishes in Dolbeault cohomology.

By Voisin \cite[Corollary A.5]{MR1902630}, we may choose a general complex torus
\begin{equation}\label{torus}
T=\mathbb C^3/\Lambda
\end{equation}
such that the skyscraper sheaf $\mathbb C_0$ at the origin $0\in T$ does not admit a finite locally free resolution. Let $\mathbb Z_2=\{1,\gamma\}$ act on $T$ by
$\gamma\cdot a=-a.$ Its fixed point set is the \(2\)-torsion subgroup
$T^\gamma=\{a\in T\mid 2a=0\}.$

Now consider the Hopf threefold
\begin{equation}
H=
\big(\mathbb C^3-\{0\}\big)\big/
(z_1,z_2,z_3)\sim(2z_1,2z_2,2z_3).
\end{equation}
Let $\mathbb Z_2$ act on $H$ by
\begin{equation}
\gamma\cdot (z_1,z_2,z_3)=(-z_1,z_2,z_3).
\end{equation}
The fixed point set is
\begin{equation}
H^\gamma=\{(0,z_2,z_3)\mid(z_2,z_3)\neq(0,0)\},
\end{equation}
which is a Hopf surface.

Let $\mathbb Z_2$ act diagonally on $T\times H$ by
\begin{equation}\label{10.7}
\gamma(a,h)=(-a,\gamma h).
\end{equation}
The fixed point set is
\begin{equation}
T^\gamma\times H^\gamma.
\end{equation}

We set
\begin{equation}
X= \mathbb{Z}_2\ltimes H,
\quad
Y=\mathbb{Z}_2\ltimes(T\times H).
\end{equation}
with an embedding
\begin{equation}\label{10.10}
i\colon h\in X\mapsto (0,h)\in Y.
\end{equation}

We next define an orbifold holomorphic line bundle on $X$. Let $L$ be the flat holomorphic line bundle on $H$ given by
\begin{equation}\label{10.11}
L=
\big((\mathbb C^3-\{0\})\times\mathbb C\big)\big/
\big((z_1,z_2,z_3),\xi\big)
\sim
\big((2z_1,2z_2,2z_3),2\xi\big).
\end{equation}
We equip $L$ with the Hermitian metric
\begin{equation}
h^L(z,\xi)
=
\big(|z_1|^2+|z_2|^2+|z_3|^2\big)^{-1}|\xi|^2.
\end{equation}
This gives its first Chern class
\begin{equation}\label{chernbc}
c_{1,\mathrm{BC}}(L,h^L)
=\omega=
\frac{\sqrt{-1}}{2\pi}
\partial\bar\partial
\log\big(|z_1|^2+|z_2|^2+|z_3|^2\big).
\end{equation}

We claim that
\begin{equation}
	[\omega]_{\bar\partial}=0,\quad [\omega]_{BC}\neq 0.
\end{equation}
First, $\omega$ is Dolbeault-exact. Indeed,
\begin{equation}
\omega
=
-\bar\partial\left(
\frac{\sqrt{-1}}{2\pi}\partial
\log\big(|z_1|^2+|z_2|^2+|z_3|^2\big)
\right).
\end{equation}
Hence $[\omega]_{\bar\partial}=0$.

On the other hand, $\omega$ is not zero in Bott-Chern cohomology. Suppose, for contradiction, that there exists a smooth function $f\in C^\infty(H)$ such that
\begin{equation}
\omega=
\frac{\sqrt{-1}}{2\pi}\partial\bar\partial f.
\end{equation}
Pulling back to $\mathbb C^3-\{0\}$, and writing again $f$ for the pullback, we get
\begin{equation}
\partial\bar\partial
\left(
\log\big(|z_1|^2+|z_2|^2+|z_3|^2\big)-f
\right)=0.
\end{equation}
Thus the difference
\begin{equation}
g(z)=
\log\big(|z_1|^2+|z_2|^2+|z_3|^2\big)-f(z)
\end{equation}
is pluriharmonic on $\mathbb C^3-\{0\}$. Since $\mathbb C^3-\{0\}$ is simply connected, $g=\mathrm{Re}\,G$ for some holomorphic function $G$ on $\mathbb C^3-\{0\}$. By Hartogs' theorem, $G$ extends holomorphically to $\mathbb C^3$. Since $f$ descends to $H$, we have
\begin{equation}
f(2z)=f(z).
\end{equation}
Therefore,
\begin{equation}
\mathrm{Re}\big(G(2z)-G(z)\big)=g(2z)-g(z)=\log 4
\end{equation}
Letting $z\to 0$, we get a contradiction. Hence $[\omega]_{BC}\neq 0$.

We now show that $i_*L$ admits no finite locally free resolution on $Y$. Choose the point $(0,1,1)\in H^\gamma$ and consider the map
\begin{equation}
j\colon a\in \mathbb{Z}_2\ltimes T\mapsto (a,(0,1,1))\in  Y.
\end{equation}
By \eqref{10.10}, the image of $j$ intersects $i(X)$ transversely at the single point $(0,(0,1,1))$. By definition \eqref{10.11}, $L$ is locally free on $H$. Then by the base change property of Proposition \ref{prop: cartesian diagram coherent sheaves}, we have
\begin{equation}
Lj^*i_*L\cong L_{(0,1,1)}\otimes \mathbb C_0\cong \mathbb{C}_0,
\end{equation}
 the skyscraper sheaf at $0\in T$. If $i_*L$ admitted a finite locally free resolution on $Y$, then pulling it back via $j$, would give a finite locally free resolution of $\mathbb C_0$ on $T$. This contradicts the choice of $T$ in \eqref{torus}.

We now apply the RRG formula. By Remark \ref{rmk:inertia_local} and \eqref{inertiamap}, we have
\begin{equation}
	\begin{split}
	&IX=\mathbb{Z}_2\ltimes(H\sqcup H^\gamma),\quad IY=\mathbb{Z}_2\ltimes\big((T\times H)\sqcup (T^\gamma\times H^\gamma)\big),\\
	&Ii\colon h\sqcup h'\in IX\mapsto (0,h)\sqcup (0,h')\in IY.
	\end{split}
\end{equation}
By \eqref{chernbc}, let 
\begin{equation}
	\omega_{H^\gamma}=\frac{i}{2\pi}\partial\overline{\partial}\log \big(\vert z_2\vert^2+\vert z_3\vert^2\big),
\end{equation}
the restriction of $\omega$ to $H^\gamma$, then
\begin{equation}
	\chBC(L)=e^{\omega}\sqcup e^{\omega_{H^\gamma}}=(1+\omega+\tfrac{\omega^2}{2}+\tfrac{\omega^3}{6})\sqcup(1+\omega_{H^\gamma}+\tfrac{\omega_{H^\gamma}^2}{2}).
\end{equation}

By \eqref{10.10}, the normal bundle of $i$ is
\begin{equation}
	N_{X/Y}\cong T_0T,
\end{equation}
a trivial holomorphic vector bundle of dimension $3$. By \eqref{10.7}, we see that $\gamma$ acts on $T_0T$ by $-1$. By Remark \ref{Todd class}, the Todd class is given by
\begin{equation}
	\TdBC(N_{X/Y})=1\sqcup 2^{-3}.
\end{equation}

Applying Theorem \ref{thm:RRG_embeddings}, we get
\begin{equation}
	\chBC(i_*L)=Ii_*\big(e^{\omega}\sqcup 8e^{\omega_{H^\gamma}}\big)=Ii_*\big((1+\omega+\tfrac{\omega^2}{2}+\tfrac{\omega^3}{6})\sqcup 8(1+\omega_{H^\gamma}+\tfrac{\omega_{H^\gamma}^2}{2})\big).
\end{equation}
In general, the pushforward of a differential form is represented by a current, see \eqref{eq: pushforward current inertia}. In the present example, however, we can find a smooth Bott-Chern representative. Let $(z_1',z_2',z_3')$ be holomorphic coordinates of $T$. Then we form the following $\gamma$-invariant form
\begin{equation}
\Omega=Cdz_1'\wedge d\overline{z_1}'\wedge dz_2'\wedge d\overline{z_2}' \wedge dz_3'\wedge d\overline{z_3}'\in \HBC^{3,3}(T,\mathbb{C}),
\end{equation}
where $C$ is a normalized factor such that $\int_T\Omega=1$. Now we can write
\begin{equation}\label{eq: chern character of iL}
	\chBC(i_*L)=\big(\Omega\wedge(1+\omega+\tfrac{\omega^2}{2}+\tfrac{\omega^3}{6})\big)\sqcup 8(1+\omega_{H^\gamma}+\tfrac{\omega_{H^\gamma}^2}{2}),
\end{equation}
where ${\omega_{H^\gamma}}$ is supported on $\{0\}\times H^\gamma\subset T^\gamma\times H^\gamma$.

Actually we can compute $\chBC(i_*L)$ directly via the antiholomorphic flat superconnection associated with $i_*L$ and Chern-Weil theory.

First recall that we have
\[
\varphi\colon
\wedge^{\bullet}T^{*}_{\mathbb C}(I\mathcal G)_0
\longrightarrow
\wedge^{\bullet}T^{*}_{\mathbb C}(I\mathcal G)_0,
\qquad
\alpha\longmapsto
(2\pi i)^{-\frac{\deg\alpha}{2}}\alpha.
\]




We equip $L$ with the $\mathbb Z_2$-linearization
\[
\gamma[z,\xi]=[\gamma z,\xi].
\]
Thus $\gamma$ acts trivially on the fibers of $L$ over $H^\gamma$.


\begin{prop}\label{prop: explicit superconnection iL}
There is an explicit, noncanonical $\mathbb Z_2$-equivariant
antiholomorphic flat superconnection
$(E_Y^{\bullet},A^{E_Y^{\bullet}\prime\prime})$ on $T\times H$ which
represents $i_*L$ under the equivalence of Corollary
\ref{coro: equiv of cats}.
\end{prop}

\begin{proof}
Choose an orientation-preserving smooth group isomorphism
$T\simeq T_1\times T_2\times T_3$, where the $T_j$ are real two-tori,
and write $p_j\colon T\to T_j$ for the projections.  Give each $T_j$
an auxiliary elliptic curve structure.  Let
$M_j=\mathcal O_{T_j}(0)$ be the inversion-linearized line bundle of
degree one, and let $\sigma_j$ be its equivariant section with a simple
zero at the origin.  We retain only the underlying smooth line bundles
when returning to the given complex structure on $T$.  Put
\[
V=\bigoplus_{j=1}^3p_j^*M_j,
\qquad
s=(p_1^*\sigma_1,p_2^*\sigma_2,p_3^*\sigma_3).
\]
Then $s$ is $\mathbb Z_2$-equivariant and $s^{-1}(0)=\{0\}$.  After an
equivariant modification in a small invariant ball about $0$, we may
suppose that, in holomorphic coordinates $w=(w_1,w_2,w_3)$ and an
equivariant trivialization of $V$, one has
\[
s(w)=w,
\qquad
\gamma|_{V_0}=-1.
\]

Choose a $\mathbb Z_2$-invariant Hermitian metric on $V$ which is
standard in this trivialization near $0$.  Choose also a
$\mathbb Z_2$-equivariant $(0,1)$-connection
$D^{V,\prime\prime}$ such that
\[
D^{V,\prime\prime}s=0
\]
and such that $D^{V,\prime\prime}=\dbar$ near $0$.  To construct it,
use the splitting $V=\mathbb Cs\oplus Q$ away from $0$, patch it with
the trivial connection near $0$, and average over $\mathbb Z_2$.

Put
\[
E_T^{-q}=\wedge^qV^*,\quad 0\leq q\leq3,
\qquad
v_0=\iota_s,
\]
and let $D^{\prime\prime}$ be the induced connection on
$E_T^\bullet$.  Then
\[
[D^{\prime\prime},v_0]=0.
\]
Write $R^{0,2}=(D^{\prime\prime})^2$.  Since
$D^{\prime\prime}=\dbar$ near $0$, the form $R^{0,2}$ vanishes there.

Let $s^\dagger$ be the Hermitian dual of $s$.  On $T-\{0\}$, let
$\varepsilon_\eta$ denote exterior multiplication by $\eta$ and set
\[
\kappa=\varepsilon_{s^\dagger/\lVert s\rVert^2}.
\]
The Koszul identity is
\[
v_0\kappa+\kappa v_0=1.
\]
The expressions
\[
v_2=-\kappa R^{0,2},
\qquad
v_3=-\kappa[D^{\prime\prime},v_2]
\]
extend smoothly by zero across $0$.  Define
\[
A_T^{\prime\prime}
=v_0+D^{\prime\prime}+v_2+v_3.
\]

We verify flatness.  Set $\delta=[v_0,\,\cdot\,]$.  The Koszul identity,
$[v_0,R^{0,2}]=0$, and the Bianchi identity give
\[
\delta v_2=-R^{0,2},
\qquad
\delta[D^{\prime\prime},v_2]=0,
\qquad
\delta v_3=-[D^{\prime\prime},v_2].
\]
Consequently, all components of $(A_T^{\prime\prime})^2$ of
antiholomorphic form degree at most $3$ vanish.  Every remaining
component has antiholomorphic form degree at least $4$, and hence
vanishes because $\dim_{\mathbb C}T=3$.  Thus
$(A_T^{\prime\prime})^2=0$.

Near $0$, the operator $A_T^{\prime\prime}$ is the
Dolbeault--Koszul operator $\iota_w+\dbar$.  Its holomorphic Koszul
complex resolves $\mathbb C_0$, and the augmentation to $\mathbb C_0$
is $\mathbb Z_2$-equivariant with the trivial action on the target.
Away from $0$, the Koszul complex is contractible by $\kappa$.
Filtering by antiholomorphic form degree therefore shows that
$(E_T^\bullet,A_T^{\prime\prime})$ represents the trivial-linearized
skyscraper sheaf $\mathbb C_0$.

Let $p_T\colon T\times H\to T$ and $p_H\colon T\times H\to H$ be the
projections.  Finally, set
\[
E_Y^{-q}=p_T^*\wedge^qV^*\otimes p_H^*L,
\qquad
A^{E_Y^{\bullet}\prime\prime}
=p_T^*A_T^{\prime\prime}\widehat{\otimes}1
+1\widehat{\otimes}p_H^*\dbar_L.
\]
This is the graded tensor product from Definition
\ref{defi: tensor product of antiholo}; it is $\mathbb Z_2$-equivariant
and flat.  Proposition \ref{prop: compatibility with tensor products}
shows that it represents
\[
\mathbb C_0\boxtimes L\simeq i_*L.
\]
We complete the proof.
\end{proof}

\begin{rmk}\label{rmk: Chern-Weil explicit model}
The superconnection in Proposition
\ref{prop: explicit superconnection iL} gives a direct Chern--Weil
verification of \eqref{eq: chern character of iL}.  Let $h^{E_T}$ be
the pure metric on $E_T^\bullet$ induced by the metric on $V$, and give
$E_Y^\bullet$ the pure product metric
\[
h^{E_Y}|_{E_Y^{-q}}
=p_T^*h^{\wedge^qV^*}\otimes p_H^*h^L.
\]
For $u>0$, put
\[
h_u^{E_Y}=h^{E_Y}u^{N^{E_Y}},
\qquad
h_u^{E_Y}|_{E_Y^{-q}}=u^{-q}h^{E_Y}|_{E_Y^{-q}}.
\]
If $D'$ and $v_j^*$ denote the adjoints for the unscaled metric, then
the adjoint of $A_T^{\prime\prime}$ for the scaled metric is
\[
A_{T,u}'=uv_0^*+D'+u^{-1}v_2^*+u^{-2}v_3^*.
\]
Set
\[
F_{T,u}=[A_T^{\prime\prime},A_{T,u}'],
\qquad
\Theta_u=\varphi\str\left[\exp(-F_{T,u})\right].
\]
If $R^L=[\dbar_L,\nabla^{L,\prime}]$, the curvature of the product
superconnection splits as
\[
F_{Y,u}=p_T^*F_{T,u}\widehat{\otimes}1
+1\widehat{\otimes}p_H^*R^L.
\]
The multiplicativity of the supertrace and
$\varphi\exp(-R^L)=e^\omega$ therefore show that the Chern--Weil form
on the identity component of $IY$ is exactly
\[
\Theta_u\wedge e^\omega.
\]

We next identify the Bott--Chern class of $\Theta_u$.  If
$x_1,x_2,x_3$ are the Chern roots of the smooth complex vector bundle
$V$, ordinary Chern--Weil deformation gives
\[
[\Theta_u]_{\mathrm{dR}}
=\sum_{q=0}^3(-1)^q\ch(\wedge^qV^*)
=\prod_{j=1}^3(1-e^{-x_j})
=c_3(V).
\]
The last equality holds because $\dim_{\mathbb C}T=3$.  The section
$s$ has one transverse complex zero of positive multiplicity, so
$\int_Tc_3(V)=1$.  Since $T$ satisfies the
$\dbar\dpar$-lemma, it follows that
\[
[\Theta_u]_{\mathrm{BC}}=[\Omega]_{\mathrm{BC}}.
\]

This equality can also be seen from the local Chern--Weil form.  Near
$0$, where $s(w)=w$, $v_2=v_3=0$, and the connection and metric are
standard, one has
\[
F_{T,u}
=u\lVert w\rVert^2
+\sum_{j=1}^3dw_j\,\iota_{e_j}
+u\sum_{j=1}^3d\overline{w_j}\,\varepsilon_{e^j},
\]
where $(e_j)$ is the local frame of $V$ and $(e^j)$ is its dual frame.
Taking the supertrace gives the normalized Mathai--Quillen form
\[
\Theta_u
=\left(\frac{\sqrt{-1}u}{2\pi}\right)^3
e^{-u\lVert w\rVert^2}
\bigwedge_{j=1}^3
\left(dw_j\wedge d\overline{w_j}\right).
\]
The Gaussian on the right has total mass $1$ on $\mathbb C^3$.  The
standard rescaling argument,
as in the localization formula used in
\cite[Theorem 6.7]{bismut1995equivariant}, shows that $\Theta_u$
converges to the current of integration $\delta_0$ as $u\to+\infty$.
Moreover, Theorem
\ref{thm: Chern character is independent of the metric} gives
\[
\frac{\partial}{\partial u}\Theta_u
=-\frac{\dbar\dpar}{2\pi i}
\varphi\str\left[
\frac{N^{E_T}}{u}\exp(-F_{T,u})
\right].
\]
Thus the smooth forms $\Theta_u$ all represent the class of
$\delta_0$, which is represented by $\Omega$.  Notice that Theorem
\ref{thm: limit T to infinity} does not apply here, since the ranks of
the cohomologies of $(E_T^\bullet,v_0)$ jump at $0$.

It remains to compute the $\gamma$-component.  Since $T^\gamma$ is a
finite set, all positive-degree forms in the $T$-direction vanish on
it.  For $a\in T^\gamma$, let
\[
\Delta_a=v_0(a)v_0(a)^*+v_0(a)^*v_0(a).
\]
The restriction of the torus curvature is $u\Delta_a$, and equivariant
finite-dimensional Hodge theory gives
\[
\str\left[\gamma\exp(-u\Delta_a)\right]
=\str\left[\gamma\mid
H^\bullet(\wedge^\bullet V_a^*,\iota_{s(a)})\right].
\]
For $a\neq0$, the Koszul complex is acyclic, and this supertrace is
zero.  At the origin, $v_0(0)=0$ and $\gamma|_{V_0^*}=-1$, so
\[
\str\left(\gamma\mid\wedge^\bullet V_0^*\right)
=\det(1-\gamma|_{V_0^*})=2^3=8.
\]
Because the chosen linearization of $L$ is trivial on the fibers over
$H^\gamma$, the full finite-parameter Chern--Weil form is therefore
\[
\ch(A^{E_Y^{\bullet}\prime\prime},h_u^{E_Y})
=\left(\Theta_u\wedge e^\omega\right)
\sqcup
\left(8\mathbf 1_{\{0\}}e^{\omega_{H^\gamma}}\right).
\]
Passing to Bott--Chern cohomology recovers
\eqref{eq: chern character of iL}.  For finite $u$, the form
$\Theta_u$ is in general cohomologous to, but not equal to, $\Omega$.
\end{rmk}

\appendix

\section{Generalized groupoid morphism}

In this section, we discuss some properties of generalized morphisms between groupoids.
\subsection{Generalized morphisms}\label{subsect: generalized morphisms}

Let $(Z, \rho, \sigma)$ be a generalized morphism from $\mathcal{G}$ to $\mathcal{H}$. We can define a groupoid $\mathcal{Z}=(Z_0, Z_1)$ from a generalized morphism $(Z,\rho, \sigma)$ where $Z_0=Z$ and 
\begin{equation}\label{eq: groupoid from generalized morphism}
Z_1=\{(z,g,h)\in Z\times G_1\times H_1| s(g)=\rho(z),~ t(h)=\sigma(z)\}.
\end{equation}
We define $s(z,g,h)=z$ and $t(z,g,h)=gzh$. In addition, we define
$$
(z_1,g_1,h_1)\cdot (z_2,g_2,h_2)=(z_2, g_1g_2, h_2h_1)
$$
when $z_1=g_2z_2h_2$.

We have natural groupoid homomorphisms $\phi_{\rho}\colon \mathcal{Z}\to \mathcal{G}$ and $\phi_{\sigma}\colon \mathcal{Z}\to \mathcal{H}$. Actually, on the $Z_0$ level, we define
$$
\phi_{\rho}(z)=\rho(z) \text{ and } \phi_{\sigma}(z)=\sigma(z).
$$
On the $Z_1$ level, we define 
$$
\phi_{\rho}(z,g,h)=g \text{ and } \phi_{\sigma}(z,g,h)=h^{-1}.
$$
It is easy to check that $\phi_{\rho}$ and $\phi_{\sigma}$ are groupoid homomorphisms and $\phi_{\rho}$ is a Morita equivalence morphism. If $(Z,\rho, \sigma)$ is a principal bibundle, then $\phi_{\sigma}$ is also a Morita equivalence morphism.

\begin{eg}\label{ex:morphism-generalized morphism}
Let $f\colon \mathcal{G}\to \mathcal{H}$ be a morphism between \'etale Lie groupoids. We consider the \emph{comma generalized morphism} $(C, \rho, \sigma)$ as follows. 

Define 
$$
C:=\{(g_0,h)\in G_0\times H_1| f_0(g_0)=t(h)\}.
$$
Since the map $t$ is \'etale, the space $C$ is a smooth manifold. 

Define the left $\mathcal{G}$-action on $C$ as
$$
g(g_0,h):=(t(g), f_1(g)h),
$$
and the right $\mathcal{H}$-action on $C$ as
$$
(g_0,h)\tilde{h}=(g_0,h\tilde{h}).
$$

The above left $\mathcal{G}$-action can be depicted in the diagram below:

{ \tiny\begin{tikzcd}
g_0 \arrow[dd, "f_0"'] &                                                   & t(g) \arrow[dd, "f_0"] &  &                                           \\
                               & {} \arrow[r, maps to]                             & {}                             &             &                                                                 \\
f_0 (g_0)=t( h)        & s(h) \arrow[l, "h"', bend right]   & t(f_1(g))=f_0(t(g))            & f_0 (g_0)=t(h) \arrow[l, "f_1(g)"', bend right, dashed] &s(h) \arrow[l, "h"', bend right, dashed] \arrow[ll, "f_1(g)h"', bend right=50].
\end{tikzcd}}

The above right $\mathcal{H}$-action can be depicted in the diagram below:

{\tiny\begin{tikzcd}
g_0 \arrow[dd, "f_0"'] &                                                   & g_0 \arrow[dd, "f_0"]          &                                          &                                  \\
                               & {} \arrow[r, maps to]                             & {}                             &                                               &                                          &                                  \\
f_0 (g_0)=t( h)        & s(h) \arrow[l, "h"', bend right]         &f_0 (g_0)= t(h) & s(h)=t(\tilde{h}) \arrow[l, "h"', bend right, dashed] & s(\tilde{h}) \arrow[l, "\tilde{h}"', bend right, dashed] \arrow[ll, "h\tilde{h}"', bend right=50].
\end{tikzcd}}

We define the map $\rho\colon C\to G_0$ as 
$$
\rho(g_0,h)=g_0,
$$
and $\sigma\colon C\to H_0$ as
$$
\sigma(g_0,h)=s(h).
$$

It is straightforward to check that $(C,\rho, \sigma)$ gives a generalized morphism from $\mathcal{G}$ to $\mathcal{H}$ in the sense of Definition \ref{defi: generalized morphism}.
\end{eg}

For another generalized morphism $(W,\tau,\chi)$ from $\mathcal{H}$ to $\mathcal{K}$, we define the \emph{composition} of $(Z,\rho, \sigma)$ and $(W,\tau,\chi)$ as follows. We first consider the submanifold
$$
\tilde{Y}=\{(z,w)\in Z\times W|\sigma(z)=\tau(w)\in H_0\}.
$$
There is a right $\mathcal{H}$-action on $\tilde{Y}$ given by
$$
h(z,w)=(zh,h^{-1}w).
$$
Since the right $\mathcal{H}$-action on $Z$ is free and proper, so is the right $\mathcal{H}$-action on $\tilde{Y}$. Therefore the quotient
$Y:=\tilde{Y}/\mathcal{H}$ is a manifold.

Define $$\theta\colon Y\to G_0,\quad \theta(z,w)=\rho(z)$$
and $$\eta: Y\to K_0,\quad \eta(z,w)=\chi(w).$$
Moreover, we define a left $\mathcal{G}$-action on $Y$ as
$$
g(z,w)=(gz,w)
$$
and a right $\mathcal{K}$-action on $Y$ as
$$
(z,w)k=(z,wk).
$$
It is clear that the triple $(Y,\theta,\eta)$ gives a generalized morphism from $\mathcal{G}$ to $\mathcal{K}$, which we define to be the composition of $(Z,\rho, \sigma)$ and $(W,\tau,\chi)$. 

Moreover, if $f\colon \mathcal{G}\to \mathcal{H}$ and $g\colon \mathcal{H}\to \mathcal{K}$ are two ordinary morphisms, then it is easy to check that the comma generalized morphism of $f$ composed with the comma generalized morphism of $g$ is naturally isomorphic to the comma generalized morphism of $g\circ f$.

\begin{rmk}
Locally, an orbifold groupoid is a transformation groupoid associated with a finite group $G$ action on a manifold $X$. In this way, a generalized morphism can be represented as a groupoid homomorphism $G\ltimes X\to H\ltimes Y$  defined by a map $\underline{f}\colon X\to Y$ equivariant with respect to the group homomorphism $\bar{f}: G\to H$.
\end{rmk}

\subsection{Proper generalized morphisms}\label{subsect: appendix proper generalized morphisms}

In this subsection, we present a characterization of proper generalized morphisms, Definition \ref{defi: proper morphism}, using the induced map between orbit spaces.

\begin{prop}\label{prop: proper equivalent}
Let $\mathcal{G}$ and $\mathcal{H}$ be proper Lie groupoids. A generalized morphism $(Z,\rho, \sigma)$ from $\mathcal{G}$ to $\mathcal{H}$ is proper if and only if the induced map $\tilde{\theta}\colon \mathcal{G}\backslash G_0 \to H_0/\mathcal{H} $ is a proper map between topological spaces.
\end{prop}
\begin{proof}[Proof of Proposition \ref{prop: proper equivalent}]
Only the "if" part is non-trivial.

We assume $\tilde{\theta}$ is a proper map.
Suppose $C$ is a compact subset of $H_0$, then $\pi_{\mathcal{H}}(C)$ is a compact subset of $H_0/\mathcal{H}$. Since $\tilde{\theta}$ is a proper map, we know that $\tilde{\theta}^{-1}(\pi_{\mathcal{H}}(C))$ is a proper subset of $\mathcal{G}\backslash G_0$.

Recall that we have the commutative diagram:
{\tiny\begin{equation}
\begin{tikzcd}
Z/\mathcal{H}\cong G_0 \arrow[d, "\pi_{\mathcal{G}}"'] \arrow[r, "\theta"]                       & H_0/\mathcal{H} \\
\mathcal{G}\backslash Z/\mathcal{H}\cong \mathcal{G}\backslash G_0. \arrow[ru, "\tilde{\theta}"'] &                
\end{tikzcd}
\end{equation}}
We need the following
\begin{lemma}\label{lemma: tech in proper}
There exists a compact subset $D$ in $G_0$ such that $\tilde{\theta}^{-1}(\pi_{\mathcal{H}}(C))$ is contained in $\pi_{\mathcal{G}}(D)$, and a compact subset $E$ in $Z$ such that $D$ is contained in $\rho(E)$.
\end{lemma}
\begin{proof}[Proof of Lemma \ref{lemma: tech in proper}]
We can choose such a $D$ by covering $\tilde{\theta}^{-1}(\pi_{\mathcal{H}}(C))$ by sufficiently small open subsets whose closures are also compact. Since the groupoid $\mathcal{G}$ is proper, we can find local sections of the map $\pi_{\mathcal{G}}\colon G_0\to \mathcal{G}\backslash G_0$ and choose $D$ to be the union of the local sections.

Similarly, we can cover $D$ by sufficiently small open subsets whose closures are also compact. Since $\mathcal{H}$ acts on $Z$ freely and properly, we can find local sections of the map $Z\to Z/\mathcal{H}$ and choose $E$ to be the union of the local sections.
\end{proof}

We resume the proof of Proposition \ref{prop: proper equivalent}. Let $D$ and $E$ be as in Lemma \ref{lemma: tech in proper}. We consider $\pi_{\mathcal{G}\backslash Z}(E)$, which is a compact subset of $\mathcal{G}\backslash Z$.

We can also consider $\tilde{\sigma}^{-1}(C)\subset \mathcal{G}\backslash Z$. It is clear that $\tilde{\rho}$ maps $\tilde{\sigma}^{-1}(C)$ to $\tilde{\theta}^{-1}(\pi_{\mathcal{H}}(C))$. In particular, for any $x\in \tilde{\sigma}^{-1}(C)$, we know that $\tilde{\rho}(x)\in \tilde{\theta}^{-1}(\pi_{\mathcal{H}}(C))$. By Lemma \ref{lemma: tech in proper}, we can find $y\in D$ and $z\in E$ such that 
$$
\pi_{\mathcal{G}}(y)=\tilde{\rho}(x) \text{ and } \pi_{Z/\mathcal{H}}(z)=y.
$$

Since we have the commutative diagram 
{\tiny\begin{equation}
\begin{tikzcd}
Z \arrow[r, "\pi_{Z/\mathcal{H}}"] \arrow[d, "\pi_{\mathcal{G}\backslash Z}"']  & G_0\cong Z/\mathcal{H} \arrow[d, "\pi_{\mathcal{G}}"]              \\
\mathcal{G}\backslash Z \arrow[r, "\tilde{\rho}"'] & \mathcal{G}\backslash G_0\cong \mathcal{G}\backslash Z/\mathcal{H},
\end{tikzcd}
\end{equation}}
we get
$$
\tilde{\rho}(\pi_{\mathcal{G}\backslash Z}(z))=\tilde{\rho}(x).
$$
In other words, $\pi_{\mathcal{G}\backslash Z}(z)$ and $x$ are in the same orbit of the $\mathcal{H}$-action on $\mathcal{G}\backslash Z$. Hence there exists an $h$ such that 
$$
\pi_{\mathcal{G}\backslash Z}(z)h=x.
$$
As  $x\in \tilde{\sigma}^{-1}(C)$, we get
$$
t(h)=\tilde{\sigma}(\pi_{\mathcal{G}\backslash Z}(z)h)=\tilde{\sigma}(x)\in C.
$$
On the other hand,
$$
s(h)=\tilde{\sigma}(\pi_{\mathcal{G}\backslash Z}(z))=\sigma(z)\in \sigma(E).
$$
In conclusion, for any $x\in \tilde{\sigma}^{-1}(C)$, we have
\begin{equation}\label{eq: x subset}
x\in \pi_{\mathcal{G}\backslash Z}(E)h, ~s(h)\in \sigma(E)\text{ and }t(h)\in C.
\end{equation}

Let
$$
H_1'=\{h\in H_1|s(h)\in \sigma(E)\text{ and }t(h)\in C\}.
$$
Since  $\mathcal{H}$ is proper and both $E$ and $C$ are compact sets, we know that $H_1'$ is compact. Moreover \eqref{eq: x subset} implies that $\tilde{\sigma}^{-1}(C)\subset \pi_{\mathcal{G}\backslash Z}(E)H_1'$, hence $\tilde{\sigma}^{-1}(C)$ is also compact.
\end{proof}

\subsection{Embedding generalized morphism}
\begin{prop}\label{prop: embedding of quotient spaces}
If a generalized morphism $(Z,\rho, \sigma)$ from $\mathcal{G}$ to $\mathcal{H}$ is an embedding, then it induces an embedding of the quotient space $G_0/\mathcal{G}$ into $H_0/\mathcal{H}$.
\end{prop}
\begin{proof}
Let $(Z,\rho, \sigma)$ be an embedding generalized morphism. We can pass to consider the groupoid morphism $\tilde{\mathcal{G}} \to \mathcal{H}$. 

We observe that for any $z\in Z$ and $n=\sigma(z)\in H_0$, and any $h\in H_1$ such that $
t(h)=n$, we have $s(h)=\sigma(z\cdot h)$. Therefore, $L:=\sigma(Z)$ is an $\mathcal{H}$-saturated closed subset of $H_0$. 

Consider the diagram
{\tiny$$
\begin{tikzcd}
H_1|_L \arrow[r, hook] \arrow[d, shift right] \arrow[d, shift left] & H_1 \arrow[d, shift left] \arrow[d, shift right] \\
L \arrow[r, "i", hook] \arrow[d, "\pi_L"']                          & H_0 \arrow[d, "\pi_\mathcal{H}"]                 \\
L/(H_1|_L) \arrow[r, "\bar{i}"]                                     & H_0/H_1  .                                       
\end{tikzcd}
$$}
The map $\bar{i}\colon L/(H_1|_L)\to H_0/H_1$ is obviously continuous. Since $L$ is $\mathcal{H}$-saturated, $\bar{i}$ is injective. We can prove that $\bar{i}$ is a closed map.

Let $[l_k]\in L/(H_1|_L)$ and $[h]\in H_0/H_1$ be such that $[l_k]\to [h]$ in $ H_0/H_1$. We choose a lift $h\in H_0$ of $[h]$ and a precompact neighborhood $U$ of $h$ in $ H_0$ which is sufficiently small so that $U\to [U]\subset H_0/H_1$ is a quotient of $U$ by a finite group. Choose a lift $l_k\in U$ of $[l_k]$.

Since $U$ is precompact and the restriction $\mathcal{H}|_{U}$ is a transformation groupoid associated with a finite group action,  $\{l_k\}$ has a subsequence converging to $l_0\in U$. Without loss of generality, we can assume $l_k\to l_0$. Since $L$ is a closed subset of $H_0$, we know $l_0\in L$. Hence $[l_0]\in  L/(H_1|_L)$. We prove that $\bar{i}\colon L/(H_1|_L)\to H_0/H_1$ is a closed map and an embedding.

We consider the groupoid morphism
{\tiny$$
\begin{tikzcd}
\widetilde{H}_1 \arrow[d, shift right] \arrow[d, shift left] \arrow[r] & H_1|_L \arrow[d, shift right] \arrow[d, shift left] \\
Z \arrow[r]                                                        & L.                                                  
\end{tikzcd}
$$}
As $\widetilde{H}_1$ is the pullback of $H_1$ to $Z$ via $\sigma\colon Z\to L$, the above morphism is a Morita equivalence. Hence the quotient $Z/\widetilde{H}_1$ is homeomorphic to $L/( H_1|_L )$.

We also consider 
{\tiny$$
\begin{tikzcd}
\widetilde{G}_1 \arrow[d, shift left] \arrow[d, shift right] \arrow[r, "\phi"] & \widetilde{H}_1 \arrow[d, shift left] \arrow[d, shift right] \\
Z \arrow[r, equal]                                           & Z.                                                       
\end{tikzcd}
$$}
Since $\phi$ is saturated, the induced map $Z/\widetilde{G}_1 \to Z/\widetilde{H}_1 $ is the identity map.

Combining the above results, we conclude that the map $G_0/G_1\to H_0/H_1$ is an embedding.
\end{proof}

The following proposition shows that being stabilizer-preserving is preserved under taking induced (generalized) morphisms between inertia groupoids.
\begin{prop}\label{prop: inertia morphism of stablizer-preserving embedding is  a stablizer-preserving embedding1}
Let $(Z,\rho, \sigma)\colon \mathcal{G}\to \mathcal{H}$ be a generalized morphism between complex orbifold groupoids. If $(Z,\rho, \sigma)$ is a stabilizer-preserving embedding, then $(IZ,I\rho, I\sigma)$ is also a stabilizer-preserving embedding.
\end{prop}
\begin{proof}[Proof of Proposition \ref{prop: inertia morphism of stablizer-preserving embedding is  a stablizer-preserving embedding1}]
First we need to check that $(IZ,I\rho, I\sigma)$ satisfies the conditions in Definition \ref{defi: orbifold embedding}.

\begin{lemma}\label{lemma: inertia morphism embedded subvariety}
The image of $I\sigma$ is a closed subset of $(I\mathcal{H})_0$.
\end{lemma}
\begin{proof}[Proof of Lemma \ref{lemma: inertia morphism embedded subvariety}]
First, we show that $I\sigma$ is a local embedding. For $(g,z)\in IZ$, choose a neighborhood $U$ of $z$ in $Z$ such that $\rho|_U\colon U\to G_0$ is a diffeomorphism and $\sigma|_U\colon U\to H_0$ is an embedding. Choose a neighborhood $W$ of $g$ in $(IG)_0$ and $V$ of $I\sigma(g,z)$ in $(I\mathcal{H})_0$ such that $W$ and $V$ are embedded in $\rho(U)$ and $\sigma(U)$, respectively.

We observe that $U \leftindex_{\rho} \times_t W\leftindex_s\times_{\rho} U$ forms a neighborhood of $(z,g,z')$ in $Z \leftindex_{\rho} \times_t (I\mathcal{G})_1 \leftindex_s\times_{\rho} Z$, and $U \leftindex_{\sigma} \times_t V\leftindex_s\times_{\sigma} U$ forms a neighborhood of $(z,I\sigma(g,z), z')$ in $Z \leftindex_{\rho} \times_t (I\mathcal{H})_1 \leftindex_s\times_{\rho} Z$. Since 
$$
\phi\colon Z \leftindex_{\rho} \times_t G_1 \leftindex_s\times_{\rho} Z\to Z \leftindex_{\rho} \times_t H_1 \leftindex_s\times_{\rho} Z
$$
is an embedding, the restriction
$$
\phi\colon U \leftindex_{\rho} \times_t W\leftindex_s\times_{\rho} U\to U \leftindex_{\sigma} \times_t V\leftindex_s\times_{\sigma} U
$$
is also an embedding. Notice that $\rho|_U$ is a diffeomorphism, hence $U \leftindex_{\rho} \times_t W\leftindex_s\times_{\rho} U$ is diffeomorphic to $W\leftindex_s\times_{\rho} U$. On the other hand, $U \leftindex_{\sigma} \times_t V\leftindex_s\times_{\sigma} U$ is naturally embedded into $V$. This shows that the map $W\leftindex_s\times_{\rho} U\to V$ is an embedding. Hence $IZ\to (I\mathcal{H})_0$ is a local embedding.

We need to show that $I\sigma$ is a closed map. Let $h_n\in (I\mathcal{H})_0$ be in the image of $I\sigma$. In other words, there exist $g_n\in G_1$ and $z_n\in Z$ such that
$$
g_nz_nh_n=z_n.
$$
If $h_n\to h_0\in (I\mathcal{H})_0$, then we need to show that $h_0$ is also in the image of $I\sigma$.

Notice that $h_0\in H_1$ and $s(h_0)=t(h_0)\in H_0$. Since $(Z,\rho, \sigma)$ is an embedding, by Condition (2) of Definition \ref{defi: orbifold embedding} there exists an open neighborhood $\tilde{W}$ of $s(h_0)=t(h_0)$ such that $\sigma(Z)\cap \tilde{W}$ has a finite resolution. Since $\sigma(z_n)\in H_0$ converges to $s(h_0)=t(h_0)$, there exists a subsequence of $\{z_n\}$ whose images under $\sigma$ are contained in one $\sigma(V_i)$. By abuse of notation, we still denote this subsequence by $\{z_n\}$.

Recall that by Definition \ref{defi: orbifold embedding} Condition (2) $\sigma|_{V_i}$ is a diffeomorphism and $\rho|_{V_i}$ is a diffeomorphism to an open subset of $G_0$ if $V_i$ is sufficiently small.

We can also choose an open neighborhood $W$ of $h_0$ in $H_1$ such that $s|_W$ and $t|_W$ are both diffeomorphisms to $\tilde{W}$. Similarly, we choose an open neighborhood $V\subset G_1$ such that $s|_V$ is a diffeomorphism to $\rho(U)\subset G_0$ and $t|_V$ is a diffeomorphism to its image. 

Since $\sigma|_U$ is a diffeomorphism to a branch containing $\{\sigma(z_n)\}$, there exists a unique $z'_n\in U$ such that $\sigma(z'_n)=\sigma(z_n)$ for each $n\geq 1$. Therefore 
$$
(z'_n,\id,z_n)\in \widetilde{H}_1
$$
in the sense of \eqref{eq: tilde H}. Since $(Z,\rho, \sigma)\colon \mathcal{G}\to \mathcal{H}$ is an embedding, the induced map $\phi\colon \tilde{\mathcal{G}}\to \tilde{\mathcal{H}}$ is saturated. Hence there exists an $g'_n\in G_1$ such that
$$
(z'_n,g'_n,z_n)\in \widetilde{G}_1.
$$
Let $\phi(z'_n,g'_n,z_n)=(z'_n,h'_n,z_n)$. Then for each $n\geq 1$ we found $g'_n\in G_1$ and $h'_n\in H_1$ such that
\begin{equation}
g'_nz_nh'_n=z'_n.
\end{equation}
Since $g_nz_nh_n=z_n$, we get
\begin{equation}
g'_ng_n(g'_n)^{-1}z'_n(h'_n)^{-1}h_nh'_n=z'_n.
\end{equation}
We notice that 
\begin{equation}\label{eq: h'_n}
s(h'_n)=t(h'_n)=\sigma(z'_n).
\end{equation}

Since there are only finitely many isotropy elements for every point in $H_0$, \eqref{eq: h'_n} implies that there exists a subsequence of $\{h'_n\}$ that converges to some $h'_0\in H_1$. Notice that both $h'_0$ and $h_0$ are in the isotropy group of $\sigma(z_0)$. On the other hand, we know that $h_n$ converges to $h_0\in H_1$, therefore $(h'_n)^{-1}h_nh'_n$ converges to $(h'_0)^{-1}h_0h'_0$.  Again, since $\phi\colon \tilde{\mathcal{G}}\to \tilde{\mathcal{H}}$ is an embedding, $g'_ng_n(g'_n)^{-1}$ also converges to some $g'_0\in G_1$. Since $z'_n\in U$, there is also a subsequence of $z'_n$ which converges to $z'_0\in Z$.

Therefore, we can assume that $(g'_ng_n(g'_n)^{-1},z'_n)$ converges to $(g'_0,z'_0)\in IZ$. Hence,  $(h'_0)^{-1}h_0h'_0=I\sigma(g'_0,z'_0)$ is in the image of $I\sigma$. Therefore, $h_0$ is also in the image of $I\sigma$ as $h_0=I\sigma(g'_0,z'_0(h'_0)^{-1})$.
\end{proof}

\begin{lemma}\label{lemma: inertia finitely many branches nbh}
For each point in the image of $I\sigma$ in $(I\mathcal{H})_0$, there exists a neighborhood such that it has a finite resolution.
\end{lemma}
\begin{proof}[Proof of Lemma \ref{lemma: inertia finitely many branches nbh}]
It follows from the fact that every point in $H_0$ has only finitely many isotropy elements.
\end{proof}

Conditions (1) and (3) of Definition \ref{defi: orbifold embedding} are clear, as they are local properties and $(Z,\rho, \sigma)$ is a stabilizer-preserving embedding. This completes the proof of Proposition \ref{prop: inertia morphism of stablizer-preserving embedding is  a stablizer-preserving embedding1}.
\end{proof}

\begin{rmk}
In general $(IZ,I\rho, I\sigma)$ is not an iso-spatial embedding even if $(Z,\rho, \sigma)$ is.
\end{rmk}

\subsection{Graph generalized morphisms}\label{subsect: graph generalized morphism}
For a generalized morphism $(Z,\rho, \sigma)$ from $\mathcal{G}$ to $\mathcal{H}$, we can define its \emph{graph} generalized morphism from $\mathcal{G}$ to $\mathcal{G}\times \mathcal{H}$ as the triple $$(\Gr(Z), \Gr(\rho), \Gr(\sigma)),$$ where
\begin{equation}\label{eq: Gr(Z)}
\Gr(Z):=\{(g,z)\in G_1\times Z|s(g)=\rho(z)\}.
\end{equation}
The map $\Gr(\rho)\colon \Gr(Z)\to G_0$ is defined to be
$$
\Gr(\rho)(g,z):=t(g),
$$
and the map $\Gr(\sigma)\colon \Gr(Z)\to G_0\times H_0$ is defined to be
$$
\Gr(\sigma)(g,z):=(s(g),\sigma(z)).
$$
The left $\mathcal{G}$-action on $\Gr(Z)$ is defined to be
\begin{equation}
\tilde{g}(g,z):=(\tilde{g}g,z),
\end{equation}
and the right $\mathcal{G}\times \mathcal{H}$-action on $\Gr(Z)$ is defined to be
\begin{equation}
(g,z)(\tilde{g},h):=(g\tilde{g},\tilde{g}^{-1}zh).
\end{equation}

It is easy to check that $(\Gr(Z), \Gr(\rho), \Gr(\sigma))$ is indeed a generalized morphism from $\mathcal{G}$ to $\mathcal{G}\times \mathcal{H}$. 

Moreover, we consider the natural projection $\mathcal{G}\times \mathcal{H}\to \mathcal{H}$. It is easy to see that the associated comma generalized morphism is given by the triple $(G_0\times H_1, \pr_1\times t, s\circ \pr_2)$. It is easy to check that the composition of $(\Gr(Z), \Gr(\rho), \Gr(\sigma))$ with $(G_0\times H_1, \pr_1\times t, s\circ \pr_2)$ is canonically isomorphic to $(Z,\rho, \sigma)$. This justifies the term ``graph generalized morphism''.

\begin{prop}\label{prop: graph is embedding}
For a generalized morphism $(Z,\rho, \sigma)$ from $\mathcal{G}$ to $\mathcal{H}$, the graph generalized morphism $(\Gr(Z), \Gr(\rho), \Gr(\sigma))$ is an embedding from $\mathcal{G}$ to $\mathcal{G}\times \mathcal{H}$.
\end{prop}
\begin{proof}
The proof consists of the following lemmas.

First, we show that the image of $\Gr(\sigma)\colon \Gr(Z)\to G_0\times H_0$ is a closed subset. 

\begin{lemma}\label{lemma: Gr is an immersion}
The map $\Gr(\sigma)$ is an immersion.
\end{lemma}
\begin{proof}[Proof of Lemma \ref{lemma: Gr is an immersion}]
Let $(g,z)$ be a point in $\Gr(Z)$. We know $s(g)=\rho(z)$. Notice that $\rho\colon Z\to G_0$ is a surjective submersion and $\mathcal{H}$ acts on $Z$ properly and freely. Since $\mathcal{H}$ is a proper \'etale groupoid, we know that $\rho$ is a local diffeomorphism.

Let $U$ be a neighborhood of $s(g)$, $V$ be a neighborhood of $g$, and $W$ be a neighborhood of $z$, such that $s\colon V\to U$ and $\rho\colon W\to U$  are diffeomorphisms. Therefore, $V\times_{G_0}W\cong U$ is a neighborhood of $(g,z)$ in $\Gr(Z)$, and the map 
$$
\Gr(\sigma)\colon V\times_{G_0}W\to G_0\times H_0
$$
 is given by 
$$
(g',w)\mapsto (g', \sigma\circ \rho^{-1}\circ s(g')).
$$
On $V\times_{G_0}W$, $\Gr(\sigma)$ is the graph of the map $\sigma\circ \rho^{-1}\circ s\colon V\to W$. 
Hence $\Gr(\sigma)$ is an immersion.
\end{proof}

\begin{lemma}\label{lemma: Gr is a closed map}
The map $\Gr(\sigma)$ is a closed map. 
\end{lemma}
\begin{proof}[Proof of Lemma \ref{lemma: Gr is a closed map}]
Let
$$
\{(\tilde{g}_k,z_k)\}\in \Gr(Z)= G_1\times_{G_0} Z
$$
be a sequence in $G_1\times_{G_0} Z$ such that
$$
\lim_{k\to \infty}\Gr(\sigma)(\tilde{g}_k,z_k)=(g,h)\in G_0\times H_0.
$$
We will show that there exists a $(\tilde{g},z)\in G_1\times_{G_0}Z$ such that $\Gr(\sigma)(\tilde{g},z)=(g,h)$.

Recall that 
$$
\Gr(\sigma)(\tilde{g}_k,z_k):=(s(\tilde{g}_k), \sigma(z_k))=(\rho(z_k),\sigma(z_k)).
$$
We denote $s(\tilde{g}_k)$ by $g_k$. It is clear that $\lim_{k\to \infty}g_k=g$. We choose $z_0\in Z$ such that $\rho(z_0)=g$. Since the $\mathcal{H}$-action on $Z$ is free and proper and $\mathcal{H}$ is an \'{e}tale groupoid, the map $\rho\colon Z\to Z/\mathcal{H}\cong G_0$ is \'{e}tale. Hence we can choose neighborhoods $W$ and $U$ of $z_0$ and $g$ respectively such that $\rho|_W\colon W\to U$ is a diffeomorphism.

Since the $\mathcal{H}$-action on $Z$ is free and proper, for every $z_k$, there exists a unique $\tilde{h}_k\in H_1$ such that 
$$
z_k=(\rho|_W)^{-1}(\rho(z_k))\cdot \tilde{h}_k.
$$
Hence $\sigma(z_k)=s(\tilde{h}_k)\to h\in H_0$.

Let $V$ be a neighborhood of $h$ in $H_0$. Then for sufficiently large $k$ we have $s(\tilde{h}_k)\in V$ and 
$$
t(\tilde{h}_k)=\sigma((\rho|_W)^{-1}(\rho(z_k)))\in \sigma(W).
$$

We can choose $V$ and $W$ to be precompact. By the properness of $\mathcal{H}$, we know that there is a subsequence of $\tilde{h}_k$ which has a limit $\tilde{h}_0\in H_1$. Without loss of generality, we can work with this subsequence. Hence 
$$
(u(\rho(z_k)), (\rho|_W)^{-1}(\rho(z_k))\cdot \tilde{h}_k)\in G_1\times_{G_0} Z
$$
 converge to 
$$
(u(\rho(z_0)), (\rho|_W)^{-1}(\rho(z_0))\cdot \tilde{h}_0)=(u(g), (\rho|_W)^{-1}(\rho(z_0))\cdot \tilde{h}_0)
$$ as $k$ goes to $\infty$. Here, $u\colon G_0\to G_1$ is the unit map. Let 
$$
\tilde{g}=u(g) \text{ and } z=(\rho|_W)^{-1}(\rho(z_0))\cdot \tilde{h}_0.
$$ 
We have $\Gr(\sigma)(\tilde{g},z)=(g, h)$ as expected. So $\Gr(\sigma)$ is a closed map.
\end{proof}

\begin{lemma}\label{lemma: image of Gr is an embedded subvariety}
The image of $\Gr(\sigma)$ is a closed subset of $G_0\times H_0$.
\end{lemma}
\begin{proof}[Proof of Lemma \ref{lemma: image of Gr is an embedded subvariety}]
By Lemma \ref{lemma: Gr is an immersion} and Lemma \ref{lemma: Gr is a closed map}, $\Gr(\sigma)\colon G_1\times_{G_0}Z \to G_0\times H_0$ is a closed immersion, hence its image is a closed subset of $G_0\times H_0$. 
\end{proof}

\begin{lemma}\label{lemma: image of Gr finite branches}
For any point $x$ in the image of $\Gr(\sigma)$, there exists an open neighborhood $U$ of $x$ such that $\text{Image}(\Gr(\sigma))\cap U$ has a finite resolution. 
\end{lemma}
\begin{proof}[Proof of Lemma \ref{lemma: image of Gr finite branches}]
Let $(g,z)$ and $(\tilde{g},\tilde{z})$ be another point in $\Gr(Z)$ such that $\Gr(\sigma)(g,z)=\Gr(\sigma)(\tilde{g},\tilde{z})$, then $s(g)=s(\tilde{g})=\rho(z)=\rho(\tilde{z})$. So there exist $g_1\in G_1$ and $h_1\in H_1$ such that $g_1V$ and $Wh_1$ are neighborhoods of $\tilde{g}\in G_1$ and $\tilde{z}\in Z$, respectively. It is clear that $s(h_1)=t(h_1)=\sigma(z)=\sigma(\tilde{z})$, i.e. $h_1$ is in the isotropy group of $\sigma(z)$, which is a finite group, since the groupoid $\mathcal{H}$ is proper.

The image of $g_1V\times_{G_0}Wh_1$ will be the same as that of $V\times_{G_0}W$ on the $G_0$ component but twisted by $h_1$ on the $H_0$ component. Hence $\text{Image}(\Gr(\sigma))$ restricted to a sufficiently small open neighborhood of $\Gr(\sigma)(g,z)$ is the union of the image of $V\times_{G_0}Wh_1$ over all $h_1\in H_{\sigma(z)}$, which is the isotropy group of $\sigma(z)$.  Recall from Lemma \ref{lemma: Gr is an immersion} that $\Gr(\sigma)$ is a diffeomorphism on $V\times{G_0} Wh_1$. We have obtained the desired local finite resolution property.  
\end{proof}

We can define the pullback groupoids $\tilde{\mathcal{G}}$, $\widetilde{\mathcal{G}\times \mathcal{H}}$ and the map $\phi\colon \tilde{\mathcal{G}}\to\widetilde{\mathcal{G}\times \mathcal{H}}$ as in Section \ref{subsec: orbifold embeddings}. In more detail, $\tilde{\mathcal{G}}=(\widetilde{G}_0,\widetilde{G}_1)$ where $\widetilde{G}_0=\Gr(Z)$ as in \eqref{eq: Gr(Z)}, and
$$
\widetilde{G}_1=\{((g_1,z_1),g, (g_2,z_2))\in \Gr(Z)\times G_1\times \Gr(Z)|s(g)=t(g_2),~t(g)=t(g_1)\}.
$$ 
Similarly, $\widetilde{\mathcal{G}\times \mathcal{H}}=(\widetilde{K}_0, \widetilde{K}_1)$ where $\widetilde{K}_0=\Gr(Z)$ and 
\begin{equation*}
\begin{split}
\widetilde{K}_1=\{&((g_1,z_1),g,h, (g_2,z_2))\in \Gr(Z)\times G_1\times H_1 \times \Gr(Z)|\\&s(g_1)=\rho(z_1)=t(g),~\sigma(z_1)=t(h),s(g_2)=\rho(z_2)=s(g),~\sigma(z_2)=s(h)\}.
\end{split}
\end{equation*}
Moreover, the map $\phi$ is defined by $\phi_0\colon \widetilde{G}_0\to \widetilde{K}_0$  being the identity map and $\phi_1\colon \widetilde{G}_1\to \widetilde{K}_1$ being
$$
\phi_1((g_1,z_1),g, (g_2,z_2))=((g_1,z_1),g_1^{-1}gg_2^{-1},h(z_1,g_1^{-1}gg_2^{-1},z_2), (g_2,z_2)),
$$
where $h(z_1,g_1^{-1}gg_2^{-1},z_2)$ is the unique element in $H_1$ such that 
$$
g_1^{-1}gg_2^{-1}z_2=z_1h(z_1,g_1^{-1}gg_2^{-1},z_2).
$$

By definition, it is clear that $\phi$ is saturated. 

\begin{lemma}\label{lemma: pullback groupoid injective immersion}
$\phi_1$ is an injective immersion.
\end{lemma}
\begin{proof}[Proof of Lemma \ref{lemma: pullback groupoid injective immersion}]
By definition, it is clear that $\phi$ is injective. 

We show that $\phi$ is an immersion. Let $U$ be an open subset of $G_1$ such that $s$ and $t$ are diffeomorphisms when restricted to $U$. Let $U_1$ and $U_2$ be open neighborhoods of $g_1$ and $g_2$, respectively, so that
$$
s\colon U_1\to s(U_1)\subset s(U),\quad t\colon U_2\to t(U_2)\subset t(U)
$$
are diffeomorphisms.

Since $\rho\colon Z\to G_0$ is a local diffeomorphism, we can choose open neighborhoods $V_1$ and $V_2$ of $z_1$ and $z_2$, respectively, so that
$$
\rho\colon V_1\to t(U_1),\quad \rho\colon V_2\to t(U_2)
$$
are also diffeomorphisms. Therefore we have the diffeomorphism
\begin{equation}\label{eq: diffeo U}
(U_1\times_{G_0}V_1)\times_{G_0}U\times_{G_0}(U_2\times_{G_0}V_2)\overset{\simeq}{\longrightarrow} U.
\end{equation}

On the other hand, we can choose an open neighborhood $W$ of $h(z_1,g_1^{-1}gg_2^{-1},z_2)$ in $H_1$ such that $s$ and $t$ map $W$ onto $\sigma(V_1)$ and $\sigma(V_2)$, respectively. Therefore we have another diffeomorphism
\begin{equation}\label{eq: diffeo U W}
(U_1\times_{G_0}V_1)\times_{G_0\times H_0}(U\times W)\times_{G_0\times H_0}(U_2\times_{G_0}V_2)\overset{\simeq}{\longrightarrow} U\times W.
\end{equation}

We consider the restriction of $\phi_1\colon \widetilde{G}_1\to \widetilde{G\times H}_1$ to $(U_1\times_{G_0}V_1)\times_{G_0}U\times_{G_0}(U_2\times_{G_0}V_2)$. Under the diffeomorphisms \eqref{eq: diffeo U} and  \eqref{eq: diffeo U W}, $\phi_1|_{(U_1\times_{G_0}V_1)\times_{G_0}U\times_{G_0}(U_2\times_{G_0}V_2)}$ is nothing but the embedding $U$ to $U\times W$. Therefore $\phi_1$ is an injective immersion.
\end{proof}

It remains to show that $\phi_1$ is an embedding, which reduces to showing the following
\begin{lemma}\label{phi is closed}
$\phi_1$ is a closed map.
\end{lemma}
\begin{proof}[Proof of Lemma \ref{phi is closed}]
Suppose we have a sequence $((g^k_1,z^k_1),g^k, (g^k_2,z^k_2))$ such that
$$
((g^k_1,z^k_1),(g^k_1)^{-1}g^k(g^k_2)^{-1},h(z^k_1,(g^k_1)^{-1}g^k(g^k_2)^{-1},z^k_2), (g^k_2,z^k_2))
$$
approaches to $((g^0_1,z^0_1),g^0,h^0, (g^0_2,z^0_2))$ as $k\to \infty$. This implies that as $k\to \infty$
\begin{equation}\label{eq: h to h0}
h(z^k_1,(g^k_1)^{-1}g^k(g^k_2)^{-1},z^k_2)\to h^0,
\end{equation}
and $g_i^k\to g_i^0$, $z_i^k\to z_i^0$ for $i=1$ and $2$, and 
$$
(g^k_1)^{-1}g(g^k_2)^{-1}\to g^0.
$$
Hence $g^k\to g_1^0g^0g_2^0$ as $k\to \infty$. In conclusion 
$$
((g^k_1,z^k_1),g^k, (g^k_2,z^k_2))\to ((g^0_1,z^0_1),g_1^0g^0g_2^0,  (g^0_2,z^0_2)),
$$
as $k\to \infty$. We need to show that
$$
\phi_1((g^0_1,z^0_1),g_1^0g^0g_2^0, (g^0_2,z^0_2))=((g^0_1,z^0_1),(g^0,h^0), (g^0_2,z^0_2)).
$$

By definition, we know that
$$
(g^k_1)^{-1}g^k(g^k_2)^{-1}z_2^k=z_1^k h(z^k_1,(g^k_1)^{-1}g^k(g^k_2)^{-1},z^k_2).
$$
Let $k\to\infty$, we get
$$
g^0 z_2^0=z_1^0h(z_1^0, (g_1^0)^{-1}g^0(g_2^0)^{-1}, z_2^0).
$$
Since we assume that the $\mathcal{H}$-action on $Z$ is free, we can show that for each $k$, the element $h(z^k_1,(g^k_1)^{-1}g^k(g^k_2)^{-1},z^k_2)\in H_1$ is contained in a precompact neighborhood of $h(z_1^0, (g_1^0)^{-1}g^0(g_2^0)^{-1}, z_2^0)$. Therefore
\begin{equation}\label{eq: h to h1}
h(z^k_1,(g^k_1)^{-1}g^k(g^k_2)^{-1},z^k_2)\to h(z_1^0, (g_1^0)^{-1}g^0(g_2^0)^{-1}, z_2^0)
\end{equation}
as $k\to \infty$. Combining \eqref{eq: h to h0} and \eqref{eq: h to h1}  we know that
$$
 h(z_1^0, (g_1^0)^{-1}g^0(g_2^0)^{-1}, z_2^0)=h^0,
$$
Hence 
$$
\phi_1((g^0_1,z^0_1),g_1^0g^0g_2^0,  (g^0_2,z^0_2))=((g^0_1,z^0_1),(g^0,h^0), (g^0_2,z^0_2)),
$$
as desired.
\end{proof}
We have now completed the proof of Proposition \ref{prop: graph is embedding} with Lemmas \ref{lemma: Gr is an immersion}-\ref{phi is closed}.
\end{proof}

\begin{rmk}\label{rmk:graphembedding}
In general, the embedding $(\Gr(Z), \Gr(\rho), \Gr(\sigma))$ is neither iso-spatial nor stabilizer-preserving. However, by Proposition \ref{prop: decompose an embedding into two types}, $(\Gr(Z), \Gr(\rho), \Gr(\sigma))$ is the composition of an iso-spatial embedding and a stabilizer-preserving embedding.
\end{rmk}

\section{Coherent Sheaves on complex orbifold groupoids}

In this section, we develop some properties of coherent sheaves on complex orbifold groupoids. 

\begin{prop}\label{prop: weak condition on coherent} 
Suppose that $\mathcal{G}$ is Morita equivalent to a transformation groupoid associated with a compact group $K$ action on a manifold $M$. A sheaf $\mathcal{F}$ of $\mathcal{O}_{\mathcal{G}}$-modules is coherent if it satisfies Condition (2) in Definition \ref{defi: coherent sheaves on orbifolds} and the following weakened version of Condition (1):

$({1}^{\prime})$. For every $x\in G_0$, there exists an open neighborhood $(U,G)$ of $x$ such that there exists a (not necessarily $\mathcal{G}$-equivariant) surjective map $\mathcal{O}_U^n\twoheadrightarrow \mathcal{F}|_U$.
\end{prop}
\begin{proof} The proof consists of two steps. 

\noindent{\bf Step I. Local construction.} Let $\phi\colon \mathcal{O}_U^n\twoheadrightarrow \mathcal{F}|_U$ be a surjective map, which is not necessarily $\mathcal{G}$-equivariant. Consider
$$
\mathcal{S}|_U:=\bigoplus_{g\in G} (\mathcal{O}_U^n)_g,
$$
where $(\mathcal{O}|_{U})_g$ is a copy of $ \mathcal{O}_U^n$ indexed by $g$. Since $G$ is a finite group, $\mathcal{S}|_U=\oplus_{g\in G} (\mathcal{O}_U^n)_g$ is still a finite-rank free sheaf on $U$. Define the 
$G$-action on $\mathcal{S}|_U$ as follows: for a section $s\in \Gamma(\mathcal{O}_U^n)_g$ and $h\in G$, we have $h\cdot s=hs\in (\mathcal{O}_U^n)_{hg}$, where the $hs$ on the right-hand side is the action of $h$ on the original $\mathcal{O}_U^n$.

Now we define $\tilde{\phi}\colon \mathcal{S}|_U\to \mathcal{F}|_U$. For $s\in \Gamma(\mathcal{O}_U^n)_g$, define
$$
\tilde{\phi}(s):=g\phi(g^{-1}s)
$$
where $g^{-1}s$ comes from the original $G$-action on $\mathcal{O}_U^n$. We can check that for any $h\in G$,
$$
\tilde{\phi}(h\cdot s)=hg\phi((hg)^{-1}hs)=hg\phi(g^{-1}s)=h(\tilde{\phi}(s)),
$$
hence $\tilde{\phi}$ is $\mathcal{G}|_U$-equivariant. It is clear that $\tilde{\phi}$ is still surjective. 

\noindent{\bf Step II. Global construction.} To construct the global $\mathcal{G}$-sheaf $\mathcal{M}$, by Remark \ref{rmk:global quotient}, we start with passing $\mathcal{G}$ to a Morita equivalent groupoid that is a transformation groupoid $K\ltimes M$ associated with a compact group $K$-action on a manifold $M$ that carries a $K$-invariant transverse\footnote{By a transverse complex structure, we mean a complex structure on the normal bundle to $K$-orbits in $M$.} complex structure. We refer the reader to \cite{moerdijk1997orbifolds} for the detail of the constructions of $K$, $M$, and the $K$-action on $M$. In particular, $G_0$ is an embedded (possibly disconnected) submanifold of $M$ that is transverse to every $K$-orbit. The restriction of the transformation groupoid $K\ltimes M$ to $U$ is exactly the groupoid $\mathcal{G}$.   

Inside $M$, there is a $K$-invariant open subset $\widetilde{U}$ associated with the open subset $U$ such that $\mathcal{G}|_U$ is Morita equivalent to the transformation groupoid associated with the $K$-action on $\widetilde{U}$. Furthermore, the sheaf $\mathcal{S}|_U$ constructed on $U$ in Step I defines a $K$-equivariant vector bundle $\widetilde{\mathcal{S}}|_{\widetilde{U}}$ on $\widetilde{U}$, whose restriction to $U$ coincides with $\mathcal{S}|_U$. 

We apply the construction in the proof of \cite[Proposition 2.4]{segal1968equivariant} to the sheaf $\widetilde{\mathcal{S}}|_{\widetilde{{U}}}$ and find a $K$-module $\mathbf{M}$ and $\widetilde{\mathcal{S}}|_{\widetilde{{U}}}$ is a $K$-invariant direct summand of $\widetilde{U}\times \mathbf{M}$. Now on $M$, consider the sheaf $\underline{\mathbf{M}}:=M\times \mathbf{M}$ on $M$ and $\mathcal{M}$, the restriction of $\underline{\mathbf{M}}$ to $G_0$. It follows from its definition that $\mathcal{M}$ is a finite rank free sheaf over $G_0$ and carries a $\mathcal{G}$-sheaf structure. As $\widetilde{\mathcal{S}}|_{\widetilde{U}}$ is a direct summand of $\widetilde{U}\times \mathbf{M}$, we obtain a $K$-equivariant surjective map from $\underline{\mathbf{M}}_{\widetilde{U}}$ to $\widetilde{\mathcal{S}}|_{\widetilde{U}}$, whose restriction to $U$ is a surjective $\mathcal{G}|_U$-equivariant morphism $\psi$ from $\mathcal{M}|_U$ to $\mathcal{S}|_U$. Its composition with $\tilde{\phi}$ in Step I. defines a $\mathcal{G}|_U$-equivariant surjective map from $\mathcal{M}|_U$ to $\mathcal{F}|_U$. This proves that $\mathcal{F}$ satisfies Condition (1) in Definition \ref{defi: coherent sheaves on orbifolds}.
\end{proof}


The following result is a version of Oka's coherence theorem for complex orbifolds. 

\begin{prop}\label{prop: Oka coherence}
Let $\mathcal{G}$ be a complex orbifold groupoid. Then the sheaf $\mathcal{O}_{\mathcal{G}}$ is coherent.
\end{prop}
\begin{proof}
We only need to prove that $\mathcal{O}_{\mathcal{G}}$ satisfies Condition (2) in Definition \ref{defi: coherent sheaves on orbifolds}. For any $\mathcal{G}|_U$-equivariant  map $\phi\colon \mathcal{M}_U\to \mathcal{O}_U$. Forget the $G$-action, by the ordinary Oka's coherence theorem, $\ker \phi$ is of finite type, i.e. there exists a finite set $S$ in $\ker \phi$ that generates $\ker \phi$. Let $GS$ be the $G$-orbit of $S$. Since $G$ and $S$ are both finite, $GS$ is still a finite subset of $\ker \phi$. Take the free $\mathcal{O}_U$-module with a basis $GS$. Then it has an obvious $G$-action and an $\mathcal{G}$-equivariant surjection to $\ker \phi$. By Proposition \ref{prop: weak condition on coherent}, $\ker{\phi}$ is finite type and therefore $\mathcal{O}_\mathcal{G}$ is coherent.
\end{proof}

The following result is a version of the Syzygy theorem for complex orbifold groupoids.

\begin{prop}\label{prop: Syzygy}
Let $\mathcal{G}$ be a complex orbifold groupoid. Then for any coherent sheaf $\mathcal{F}$ and every $x\in G_0$ there exists an open neighborhood $(U,G)$ of $x$ such that there exists a $\mathcal{G}$-equivariant finite length resolution
\begin{equation}\label{eq: finite resolution of coherent sheaf}
0\to \mathcal{M}^{-k}_U\to \mathcal{M}^{-k+1}_U\to\ldots\to \mathcal{M}^0_U\to \mathcal{F}|_U,
\end{equation}
where $\mathcal{M}^0_U,..., \mathcal{M}^{-k+1}_U$ are $\mathcal{G}$-sheaves of finite rank free\footnote{We call a $\mathcal{G}$-sheaf $\mathcal{M}_U$ free if it is a $\mathcal{G}$-sheaf of finite rank free $\mathcal{O}_U$-modules that is the restriction of a $\mathcal{G}$-sheaf of finite rank free $\mathcal{O}_{G_0}$-modules as in part (1) of Definition \ref{defi: coherent sheaves on orbifolds}.} $\mathcal{O}_{U}$-modules and $\mathcal{M}^{-k}_U$ is a $\mathcal{G}$-sheaf of finite rank locally free $\mathcal{O}_{U}$-modules.
\end{prop}
\begin{proof}
By Definition \ref{defi: coherent sheaves on orbifolds}, for any $x\in G_0$, there exists an open neighborhood $(U,G)$ of $x$, a $\mathcal{G}$-equivariant finite rank free sheaf $\mathcal{M}^0_U$ on $U$ together with a $\mathcal{G}$-equivariant surjective map
$$
d_0: \mathcal{M}^0_U\twoheadrightarrow \mathcal{F}|_U.
$$
Again by Definition  \ref{defi: coherent sheaves on orbifolds}, $\ker(d_0)$ is also coherent hence there exists a $\mathcal{G}$-equivariant  finite rank free sheaf $\mathcal{M}^{-1}_U$ on (possibly smaller) $U$ together with a $\mathcal{G}$-equivariant surjective map
$$
d_{-1}: \mathcal{M}^{-1}_U\twoheadrightarrow \ker(d_0).
$$
We then obtain a resolution
$$
\mathcal{M}^{-1}_U\xrightarrow{{d_{-1}}}\mathcal{M}^0_U\overset{d_0}{\longrightarrow} \mathcal{F}|_U.
$$
We continue this process. Notice that since $G_0$ is non-singular, there exists a $k$ such that $\ker (d_{-k+1})$ is a locally free sheaf on $U$, where $d_{-k+1}$ is the $\mathcal{G}$-equivariant map
$$
d_{-k+1}: \mathcal{M}^{-k+1}_U\to \mathcal{M}^{-k+2}_U.
$$ 
We then take 
$$
\mathcal{M}^{-k}_U:=\ker d_{-k+1}, 
$$
and $d_{-k}: \mathcal{M}^{-k}_U\to \mathcal{M}^{-k+1}_U$ the inclusion map. We obtain the resolution
$$
0\to \mathcal{M}^{-k}_U\xrightarrow{{d_{-k}}} \mathcal{M}^{-k+1}_U\xrightarrow{{d_{-k+1}}} \ldots\longrightarrow \mathcal{M}^0_U\overset{d_0}{\longrightarrow}  \mathcal{F}|_U.
$$
Notice that all $\mathcal{M}^i_U$'s except $\mathcal{M}^{-k}_U$ are free sheaves.
\end{proof}

\begin{rmk}\label{rmk: resolution longer than k}
If the resolution in \eqref{eq: finite resolution of coherent sheaf} has length $k$, then we can find a finite-length resolution with length arbitrarily longer than $k$. Actually, in the proof of Proposition \ref{prop: Syzygy}, the locally free sheaf $\ker d_{-k+1}$ is still coherent. Hence instead of setting $
\mathcal{M}^{-k}_U:=\ker d_{-k+1}
$, we can find a $\mathcal{G}$-equivariant finite rank free sheaf $\mathcal{M}^{-k}_U$ on $U$ together with a $\mathcal{G}$-equivariant surjective map
$$
d_{-k}: \mathcal{M}^{-k}_U\twoheadrightarrow \ker(d_{-k+1}).
$$
Again $\ker(d_{-k})$ is locally free hence coherent. We can continue this resolution and terminate at any degree $-l\leq -k$. In the resolution 
$$
0\to \mathcal{M}^{-l}_U\xrightarrow{{d_{-l}}} \mathcal{M}^{-l+1}_U\xrightarrow{{d_{-l+1}}} \ldots\longrightarrow \mathcal{M}^0_U\overset{d_0}{\longrightarrow}  \mathcal{F}|_U.
$$
We still have the property that all $\mathcal{M}^i_U$'s except $\mathcal{M}^{-l}_U$ are free sheaves.
\end{rmk}

Let $\mathcal{O}_{\mathcal{G}}^{\infty}$ denote the sheaf of $C^{\infty}$-functions on $G_0$. For a sheaf of $\mathcal{O}_{\mathcal{G}}$-modules $\mathcal{F}$, let $\mathcal{F}^{\infty}$ denote the tensor product $\mathcal{F}\otimes_{\mathcal{O}_{\mathcal{G}}} \mathcal{O}_{\mathcal{G}}^{\infty}$. For a $C^{\infty}$-vector bundle $V$ on $G_0$, let $\mathcal{O}_{\mathcal{G}}^{\infty}V$ denote the associated sheaf of $C^{\infty}$-sections of $V$.

We then prove that, after tensoringwith  $\mathcal{O}^{\infty}_{G_0}$, any bounded complex with coherent cohomologies is quasi-isomorphic to a bounded complex of locally free sheaves. We first prove it in an important special case.

\begin{lemma}\label{lemma: global resolution special case}
Let $\mathcal{G}$ be a compact complex orbifold groupoid. For a bounded complex of sheaves  $\mathcal{F}^{\bullet}$ of $\mathcal{O}_{\mathcal{G}}\text{-Mod}$ with coherent cohomologies. If  $\mathcal{F}^i=0$ for $i<0$ and $\mathcal{H}^i(\mathcal{F}^{\bullet})=0$ for $i>0$, then there exists a bounded complex of finite-dimensional $C^{\infty}$-, left $\mathcal{G}$-vector bundles $Q^{\bullet}$ on $G_0$ and a quasi-isomorphism
$$
Q^{\bullet}\overset{\sim}{\longrightarrow} \mathcal{F}^{\bullet,\infty}.
$$
In addition, we can choose $Q^{\bullet}$ such that $Q^i$ is a trivial  $\mathcal{G}$-vector bundle except for the leftmost nonzero $Q^i$.
\end{lemma}
\begin{proof}
This is similar to the first part of the proof of \cite[Proposition 6.3]{Xu2025} and \cite[Proposition 6.3.2]{bismut2023coherent}. 

    Since $\mathcal{H}^0(\mathcal{F}^{\bullet})$ is a coherent sheaf, by Proposition \ref{prop: Syzygy}, for each $x\in G_0$ there exists an open neighborhood $(U,G)$ of $x$ such that there exists a $\mathcal{G}$-equivariant finite length resolution
$$
0\to \mathcal{M}^{-l_U}_U\to\ldots\to \mathcal{M}^0_U\to \mathcal{H}^0(\mathcal{F}^{\bullet})|_U,
$$
where $\mathcal{M}^0_U,..., \mathcal{M}^{-l_U+1}_U$ are $\mathcal{G}$-sheaves of finite rank free $\mathcal{O}_{U}$-modules and $\mathcal{M}^{-l_U}_U$ is a $\mathcal{G}$-sheaf of finite rank locally free $\mathcal{O}_U$-modules.
Since $\mathcal{G}$ is compact, by Proposition \ref{prop: local chart compact}, we can then choose $\mathcal{G}$-invariant neighborhoods $\{U_i\}$ so that they form a finite open cover of $G_0$. 
Since $\mathcal{O}_{U_i}^{\infty}$ is flat over $\mathcal{O}_{U_i}$, we obtain the exact sequence of $\mathcal{O}_{U_i}^{\infty}$-modules
\begin{equation}\label{eq: resolution C infty}
0\to \mathcal{O}_{U_i}^{\infty}\mathcal{M}_i^{-l_{U_i}}\to\ldots\to \mathcal{O}_{U_i}^{\infty}\mathcal{M}_i^0\to \mathcal{H}^0(\mathcal{F}^{\bullet,\infty})|_{U_i}.
\end{equation}

Let
$$
l=\sup_i l_{U_i}.
$$
According to Remark \ref{rmk: resolution longer than k}, We can then assume that on each $U_i$, the complex has length $l$, and in \eqref{eq: resolution C infty} the sheaves $\mathcal{O}_{U_i}^{\infty}\mathcal{M}^0_i,..., \mathcal{O}_{U_i}^{\infty}\mathcal{M}^{-l+1}_i$ are $\mathcal{G}$-sheaves of finite rank free $\mathcal{O}_{U_i}^{\infty}$-modules and $\mathcal{O}_{U_i}^{\infty}\mathcal{M}^{-l}_i$ is a $\mathcal{G}$-sheaf of finite rank locally free $\mathcal{O}_{U_i}^{\infty}$-modules.

We will do induction on $l$. If $l=0$, then this means that $\mathcal{H}^0(\mathcal{F}^{\bullet})$ is locally free. Since $\mathcal{F}^i=0$ for $i<0$,  and $\mathcal{H}^i(\mathcal{F}^{\bullet})=0$ for $i>0$, we obtain the resolution
$$
\mathcal{H}^0(\mathcal{F}^{\bullet,\infty})\overset{\sim}{\longrightarrow} \mathcal{F}^{\bullet,\infty}.
$$

Assume now that $l\geq 1$ and the claim is true for the case  $\leq l-1$. We denote by $\mathcal{O}_{U_i}^{\infty}\mathcal{M}_i^{\bullet}$ the complex $\mathcal{O}_{U_i}^{\infty}\mathcal{M}_i^{-l}\to\ldots\to \mathcal{O}_{U_i}^{\infty}\mathcal{M}_i^0$. Since $\mathcal{F}^i=0$ for $i<0$, \eqref{eq: resolution C infty} gives a morphism of complexes
$$
\phi_i: \mathcal{O}_{U_i}^{\infty}\mathcal{M}_i^{\bullet}\to \mathcal{F}^{\bullet,\infty}|_{U_i}.
$$
Let $\mathcal{C}_{U_i}$ be the mapping cone
$$
\mathcal{C}^{\bullet}_{U_i}:=\text{cone}(\mathcal{O}_{U_i}^{\infty}\mathcal{M}_i^{\bullet}\to \mathcal{F}^{\bullet,\infty}|_{U_i}).
$$
Actually $\mathcal{C}^{\bullet}_{U_i}$ is the complex
$$
0\to \mathcal{O}_{U_i}^{\infty}\mathcal{M}_i^{-l}\to\ldots\to \mathcal{O}_{U_i}^{\infty}\mathcal{M}_i^0\overset{\phi_i}{\longrightarrow} \mathcal{F}^{0,\infty}|_{U_i}\to \mathcal{F}^{1,\infty}|_{U_i}\to \ldots
$$
Since $\mathcal{H}^i(\mathcal{F}^{\bullet})=0$ for $i>0$, we know that $\phi_i$ is a quasi-isomorphism hence $\mathcal{C}^{\bullet}_{U_i}$ is exact. Since $l\geq 1$, each $\mathcal{M}^{0}_i$ is a free sheaf on $U_i$, we can extend $\mathcal{M}^{0}_i$ to $G_0$ and define
\begin{equation}
\mathcal{M}=\bigoplus_i \mathcal{M}_i^0,
\end{equation}
which is a $\mathcal{G}$-equivariant finite rank free  sheaf of $\mathcal{G}_{G_0}$-modules on $G_0$.

Note that each $\phi_i$ is a  $\mathcal{G}$-equivariant morphism. We use the $\mathcal{G}$-invariant\footnote{The properties of $\mathcal{G}$-invariant neighborhood in Proposition \ref{prop: local chart compact} allow us to choose a $\mathcal{G}$-invariant partition of unity on $G_0$.} partition of unity $s_i$ on $G_0$ and define
\begin{equation}\label{eq:phi}
\phi:=\sum_i s_i \phi_i.
\end{equation}
It is clear that $\phi\colon \mathcal{O}_{G_0}^{\infty}\mathcal{M}\to \mathcal{F}^{0,\infty}$ is a $\mathcal{G}$-equivariant morphism of $\mathcal{O}_{G_0}^{\infty}$-modules. Moreover, for each $U_i$, $\phi$ induces a surjection
$\phi|_{U_i}: \mathcal{O}_{U_i}^{\infty}\mathcal{M}|_{U_i}\twoheadrightarrow \mathcal{H}^0(\mathcal{F}^{\bullet,\infty})|_{U_i}$, thus $\phi$ induces a surjection $\phi: \mathcal{O}_{G_0}^{\infty}\mathcal{M}\twoheadrightarrow \mathcal{H}^0(\mathcal{F}^{\bullet,\infty})$.

We then show that for each $U_i$, the morphism $\phi|_{U_i}: \mathcal{O}_{U_i}^{\infty}\mathcal{M}|_{U_i}\twoheadrightarrow \mathcal{H}^0(\mathcal{F}^{\bullet,\infty})|_{U_i}$ lifts to a $\mathcal{G}$-equivariant morphism of $\mathcal{O}_{U_i}^{\infty}$-modules
$$
\psi_i: \mathcal{O}_{U_i}^{\infty}\mathcal{M}|_{U_i}\to \mathcal{O}_{U_i}^{\infty}\mathcal{M}^{\bullet}_i
$$
such that the following diagram commutes:
{
\begin{equation} \label{eq: LocalLift}
\begin{tikzcd}
 & \mathcal{O}_{U_i}^\infty \mathcal{M}^{\bullet}_i\arrow[d, "\phi_i"] \\
\mathcal{O}_{U_i}^\infty \mathcal{M}|_{U_i} \arrow[r, "\phi|_{U_i}"] \arrow[ru, "\psi_i", dotted] & \mathcal{F}^{\bullet,\infty}|_{U_i}                
\end{tikzcd}	
\end{equation}}
Actually, a non $\mathcal{G}$-equivariant lift
$$
\tilde{\psi_i}: \mathcal{O}_{U_i}^\infty \mathcal{M}|_{U_i}\to \mathcal{O}_{U_i}^\infty \mathcal{M}^{\bullet}_i
$$
exists as in the proof of \cite[Proposition 6.3.2]{bismut2023coherent}. Notice that $U_i$ is a $\mathcal{G}$-invariant open subset of $G_0$. We can choose $U_i$ small enough so that $\mathcal{G}|_{U_i}$ is Morita equivalent to the action groupoid of a finite group $G_i$ acting on $U_i$. We then define $\psi_i$ as the average
\begin{equation}
\psi_i:=\frac{1}{|G_i|}\sum_{g\in G_i}g\cdot \tilde{\psi_i}.
\end{equation}
It is clear that $\psi_i$ is $\mathcal{G}$-equivariant and fits into \eqref{eq: LocalLift}.

We then take the mapping cones of 
$$
\psi_i:  \mathcal{O}_{U_i}^\infty \mathcal{M}^{\bullet}|_{U_i}\to  \mathcal{O}_{U_i}^\infty \mathcal{M}^{\bullet}_i, \quad \text{and}\quad \phi|_{U_i}: \mathcal{O}_{U_i}^\infty \mathcal{M}^{\bullet}|_{U_i}\to \mathcal{F}^{\bullet,\infty}|_{U_i}.
$$
The quasi-isomorphism $\phi_i: \mathcal{O}_{U_i}^\infty \mathcal{M}^{\bullet}_i\to \mathcal{F}^{\bullet,\infty}|_{U_i}$ give the following quasi-isomorphism of  complexes:
{\tiny
\begin{equation} \label{eq: QuasiComplex}
\begin{tikzcd}
0 \arrow[r] & \mathcal{O}_{U_i}^\infty \mathcal{M}^{-l}_i \arrow[r] \arrow[d] & \cdots \arrow[r] \arrow[d] & \mathcal{O}_{U_i}^\infty \mathcal{M}^{-1}_i \oplus \mathcal{O}_{U_i}^\infty \mathcal{M}|_{U_i}\arrow[r, "\gamma_i"] \arrow[d] & \mathcal{O}_{U_i}^\infty \mathcal{M}^{0}_i \arrow[r] \arrow[d, "\phi_i"]       & 0 \arrow[r] \arrow[d]                     & \cdots \\
0 \arrow[r] & 0 \arrow[r]                              & \ldots \arrow[r]                & \mathcal{O}_{U_i}^\infty \mathcal{M}|_{U_i} \arrow[r, "\phi|_{U_i}"]                                            & {\mathcal{F}^{0,\infty}|_{U_i}} \arrow[r] & {\mathcal{F}^{1,\infty}|_{U_i}} \arrow[r] & \cdots 
\end{tikzcd}	
\end{equation}
}
Notice that the two rows of \eqref{eq: QuasiComplex} are the mapping cones of $\psi_i$ and  $\phi|_{U_i}$, respectively. Moreover, the second row of \eqref{eq: QuasiComplex} is the restriction to $U_i$ of a globally defined complex $\text{cone}(\phi)$.

Since $\mathcal{H}^i(\mathcal{F}^{\bullet})=0$ for $i>0$ and $\phi|_{U_i}: \mathcal{O}_{U_i}^{\infty}\mathcal{M}|_{U_i}\twoheadrightarrow \mathcal{H}^0(\mathcal{F}^{\bullet,\infty})|_{U_i}$ is surjective, the cohomology of $\text{cone}(\phi)$ concentrates in degree $-1$, which we denote by $\mathcal{H}^{-1}$. In particular we know that the $0$th cohomology of $\text{cone}(\phi)$ vanishes, hence the $0$th cohomology of the first row of \eqref{eq: QuasiComplex} also vanishes. Therefore the map
$$
\gamma_i: \mathcal{O}_{U_i}^\infty \mathcal{M}^{-1}_i \oplus \mathcal{O}_{U_i}^\infty \mathcal{M}|_{U_i}\to \mathcal{O}_{U_i}^\infty \mathcal{M}^{0}_i
$$
in \eqref{eq: QuasiComplex} is surjective. Hence $\ker(\gamma_i)$ is a locally free sheaf on $U_i$.

We can then truncate the first row of \eqref{eq: QuasiComplex} at degree $-1$ without changing the second row. The resulting complexes are
{
\begin{equation} \label{eq: IllusieShorterComplex}
\begin{tikzcd}
0 \arrow[r] & \mathcal{O}_{U_i}^\infty \mathcal{M}^{-l}_i\arrow[r] \arrow[d] & \cdots \arrow[r] \arrow[d] & \mathcal{O}_{U_i}^\infty \mathcal{M}^{-2}_i\arrow[r] \arrow[d]& \ker(\gamma_i)\arrow[r] \arrow[d] & 0\arrow[r] \arrow[d]       & \cdots  \\
0 \arrow[r] & 0 \arrow[r]                              & \cdots \arrow[r]  &  0 \arrow[r]            & \mathcal{O}_{U_i}^\infty \mathcal{M}|_{U_i} \arrow[r, "\phi|_{U_i}"]                                            & {\mathcal{F}^{0,\infty}|_{U_i}} \arrow[r] &  \cdots 
\end{tikzcd}	
\end{equation}}
Notice that the vertical map is still a quasi-isomorphism.

If $l=1$, then $\mathcal{M}_i^{j}=0$ for $j\leq -2$. Hence \eqref{eq: IllusieShorterComplex} becomes
{
\begin{equation} \label{eq: IllusieShorterComplex l=1}
\begin{tikzcd}
0 \arrow[r] & \ker(\gamma_i)\arrow[r] \arrow[d] & 0\arrow[r] \arrow[d]       & \cdots  \\
  0 \arrow[r]            & \mathcal{O}_{U_i}^\infty \mathcal{M}|_{U_i} \arrow[r, "\phi|_{U_i}"]                                            & {\mathcal{F}^{0,\infty}|_{U_i}} \arrow[r] &  \cdots 
\end{tikzcd}	
\end{equation}}
The quasi-isomorphism between the two rows implies that
\begin{equation}\label{eq: kernel gamma}
\ker(\gamma_i)=\ker (\phi)|_{U_i} \text{ on each } U_i.
\end{equation}
 Let $\mathcal{K}:=\ker (\phi)$. \eqref{eq: kernel gamma} tells us that $\mathcal{K}$ is a locally free sheaf of $\mathcal{O}_{G_0}^\infty$-modules on $G_0$. Hence the complex $\mathcal{K}\to \mathcal{O}_{G_0}^\infty \mathcal{M}$ gives the resolution of $\mathcal{F}^{\bullet, \infty}$.

Now we consider the case $l\geqslant 2$. Although $\ker(\gamma_i)$ is locally free on $U_i$, it may not be extended to $G_0$. To solve this problem\footnote{We thank Guangzhe Xu for pointing out this problem and sharing the idea of the proof hereafter.}, we modify \eqref{eq: IllusieShorterComplex} by adding $ \mathcal{O}_{U_i}^\infty\mathcal{M}^{0}_i$ to degrees $-2$ and $-1$ in the first row, resulting in the following quasi-isomorphism of complexes
{\small
\begin{equation} \label{eq: IllusieShorterComplex2}
\begin{tikzcd}
0 \arrow[r] & \mathcal{O}_{U_i}^\infty \mathcal{M}^{-l}_i \arrow[r] \arrow[d] & \cdots \arrow[r] \arrow[d] & \mathcal{O}_{U_i}^\infty \mathcal{M}^{-2}_i\oplus \mathcal{O}_{U_i}^\infty\mathcal{M}^{0}_i \arrow[r] \arrow[d]& \ker(\gamma_i)\oplus \mathcal{O}_{U_i}^\infty\mathcal{M}^{0}_i\arrow[r] \arrow[d] & 0\arrow[r] \arrow[d]       & \cdots  \\
0 \arrow[r] & 0 \arrow[r]                              & \cdots \arrow[r]  &  0 \arrow[r]            & \mathcal{O}_{U_i}^\infty \mathcal{M}|_{U_i} \arrow[r, "\phi|_{U_i}"]                                            & {\mathcal{F}^{0,\infty}|_{U_i}} \arrow[r] &  \cdots 
\end{tikzcd}	
\end{equation}
}
Since
$$
\gamma_i: \mathcal{O}_{U_i}^\infty \mathcal{M}^{-1}_i\oplus \mathcal{O}_{U_i}^\infty \mathcal{M}|_{U_i}\to \mathcal{O}_{U_i}^\infty \mathcal{M}^{0}_i
$$
is surjective, we have an isomrphism of $\mathcal{O}_{U_i}^\infty$-modules
\begin{equation}\label{eq: isom of kernel} \ker(\gamma_i)\oplus \mathcal{O}_{U_i}^\infty\mathcal{M}^{0}_i\cong \mathcal{O}_{U_i}^\infty \mathcal{M}^{-1}_i\oplus \mathcal{O}_{U_i}^\infty \mathcal{M}|_{U_i}.
\end{equation}
Since $l\geqslant 2$, the right hand side of \eqref{eq: isom of kernel}  is the restriction of a sheaf of free $\mathcal{O}_{G_0}^\infty$-modules on $G_0$. We obtain the following quasi-isomorphism of complexes
{\small
\begin{equation} \label{eq: IllusieShorterComplex3}
\begin{tikzcd}
0 \arrow[r] & \mathcal{O}_{U_i}^\infty \mathcal{M}^{-l}_i \arrow[r] \arrow[d] & \cdots \arrow[r] \arrow[d] & \mathcal{O}_{U_i}^\infty \mathcal{M}^{-2}_i\oplus \mathcal{O}_{U_i}^\infty\mathcal{M}^{0}_i \arrow[r] \arrow[d]& \mathcal{O}_{U_i}^\infty \mathcal{M}^{-1}_i \oplus \mathcal{O}_{U_i}^\infty \mathcal{M}|_{U_i}\arrow[r] \arrow[d] & 0\arrow[r] \arrow[d]       & \cdots  \\
0 \arrow[r] & 0 \arrow[r]                              & \cdots \arrow[r]  &  0 \arrow[r]            & \mathcal{O}_{U_i}^\infty \mathcal{M}|_{U_i} \arrow[r, "\phi|_{U_i}"]                                            & {\mathcal{F}^{0,\infty}|_{U_i}} \arrow[r] &  \cdots 
\end{tikzcd}	
\end{equation}
}

Let $\mathcal{O}_{U_i}^\infty (\tilde{\mathcal{M}_i}^{\bullet})$ and $\tilde{\mathcal{F}}^{\bullet,\infty}$ denote the first and second row of \eqref{eq: IllusieShorterComplex3}, respectively. Then on each $U_i$, $\mathcal{O}_{U_i}^\infty (\tilde{\mathcal{M}_i}^{\bullet})$ is a resolution of $\tilde{\mathcal{F}}^{\bullet,\infty}$  with length $l-1$. Moreover each component of   $\mathcal{O}_{U_i}^\infty (\tilde{\mathcal{M}_i}^{\bullet})$  is a free sheaf except the leftmost one. Then by the induction hypothesis, there exists a  bounded complex of finite-dimensional $C^{\infty}$-, left $\mathcal{G}$-vector bundles $\tilde{Q}^{\bullet}$ on $G_0$ and a quasi-isomorphism
$$
\tilde{Q}^{\bullet}\overset{\sim}{\longrightarrow} \tilde{\mathcal{F}}^{\bullet,\infty}.
$$
Also, we can pick $\tilde{Q}^{\bullet}$ such that $\tilde{Q}^i$ is a trivial $\mathcal{G}$-vector bundle except for the leftmost nonzero $\tilde{Q}^i$. Put 
\begin{equation}
Q^{\bullet}:=\tilde{Q}^{\bullet}\to \mathcal{O}_{G_0}^\infty \mathcal{M}\to 0.
\end{equation}
The $Q^{\bullet}$ gives a resolution of $\mathcal{F}^{\bullet,\infty}$. We completed the induction.
\end{proof}


\bibliography{orbifold}{}

\bibliographystyle{alpha}

\end{document}